\documentclass[a4paper,10pt]{article}
\usepackage{styles}
\usepackage{braket}
\usepackage{soul}
\usepackage[toc,page]{appendix}

\usepackage{authblk}

\title{A comparison of European and Asian options under Markov additive processes}
\author{Larbi Alili\footnote{The University of Warwick, CV47AL, Coventry, UK. \textit{E-mail}: \textbf{l.alili@warwick.ac.uk}} \:and David Woodford\footnote{The University of Warwick, CV47AL, Coventry, UK. \textit{E-mail}: \textbf{david.l.woodford@bath.edu}}\footnote{This work was supported by EPSRC as part of the MASDOC DTC, Grant No. EP/H023364/1.}}
\date{\today}

\begin{document}

\maketitle

\begin{abstract}
  We provide results relating to the integrability, uniform integrability and local integrability of exponential MAPs, which are natural extensions of exponential \levy models. Then, we use Mellin transform and partial integro-differential equation methods to value European options under a such a model. Finally, a comparison is made between the price of a European call option and that of an Asian call option. 
\end{abstract}

\keywords{Markov additive process, European option, Asian option, Mellin transform, Integrability, Martingale}

\subject{ Primary 60J75; 91B70; Secondary 60G99; 91G20}


\section{Introduction}
The aim of this paper is to study and compare European and Asian option prices when the underlying asset is modelled by the exponential of a \textit{Markov Additive Process (MAP)}. Such a model is considered in \cite{palmowski2019} in the context of optimal portfolio selection, whilst several authors have considered the pricing of options under various subspaces of MAPs, for example see \cite{bo2010markov},\cite{ellioSiu13} and \cite{naik93}. However, we are unaware of any examples in the literature that consider option pricing with the full range of MAPs.

Consider a process $\{(J_t,\xi_t):t\geq0\}$ taking values in $E\times\R$, for some finite set $E$, adapted to a filtration $\left(\mathcal{F}_t\right)_{t\geq0}$. This process is a MAP if, for any bounded continuous function $f:E\times\R\rightarrow\R^+$ and $s,t>0$,
\begin{align}
    \label{eqn:MapDefn}
    \E\left[ f(J_{t+s},\xi_{t+s}-\xi_t) \:\middle|\: \mathcal{F}_t \right]
    =\E_{J_t,0}\left[ f(J_s,\xi_s) \right],
\end{align}
where $\E_{\alpha,x}$ denotes expectation conditional on $(J_0,\xi_0)=(\alpha,x)\in E\times\R$. For each $\alpha\in E$, the notation $\Prob_\alpha$ denotes the probability measure conditional on $J_0=\alpha$ and $\E_\alpha$ will denote the corresponding expectation.  A more general definition of a MAP may allow $E$ to be countable or even uncountable, as in, for example, \cite{asmussen2003applied}. In the model considered in this paper, $\exp(\xi_t)$ corresponds to the price of a risky asset at time $t>0$ and $J_t$ corresponds to some notion of the market regime, at time $t$. In our examples, we consider the Lamperti-Kiu case from \cite{Rivero2011}, where the state space of $J$ is $\{+,-\}$.

The problem of optimal portfolio selection for this model was considered in \cite{palmowski2019}. Within \cite{palmowski2019}, it was also shown that there is an enlarged market, which is arbitrage free and complete, and which has an equivalent martingale measure.
Markov modulated jump diffusions were proposed in \cite{naik93}, where the asset price follows a geometric Brownian motion with parameters depending upon a continuous time, finite state space Markov chain. This markov chain denotes the `state' of the market. This model also allows additional jumps in the price process when the market state changes. However, this model does not include the presence of jumps during the time spent in each market state. Similar models are considered in, for example, \cite{nikita08}.

Elliot and Siu considered models in \cite{ellioSiu13}, \cite{elliot07} and \cite{siu14}, where the asset price process follows the exponential of a \levy process, with a \levy measure and drift determined by a Markov chain denoting the market state. They use a generalisation of the Essher transform to obtain an equivalent martingale measure, under which a Partial Integro-Differential Equation (PIDE) can be derived to value European options. However, these models don't allow for a jump in price caused by a change in the market state and assume that the diffusion coefficient is constant. Similar models are also considered in, for example, \cite{COSTABILE14} and \cite{ramponi12}.

In \cite{bo2010markov}, a Markov modulated jump diffusion is suggested as a model for FX spot rates, where the states of the Markov chain correspond to the sovereign rating of the corresponding countries or regions. The model considered is the exponential of a Markov modulated compound Poisson process, with log-normal jumps. However, there are no jumps assoicated with the rating changes. The authors are able to use a generalised Esscher transform to value European options under this model.



One often quoted reason for buying Asian put and call options rather than European ones is that they can be cheaper. However, we show that in some cases, when the strike is near the origin, an Asian call option is more expensive than the corresponding European one. This is analogous to the result under the Black-Scholes model given in \cite[Chapter 5, Proposition 3.1]{gemanYor}. 

This paper is set out as follows. 
In Section \ref{sec:integrability}, we provide some results on the integrability and uniform integrability of $Y_t\coloneqq\exp(\xi_t)$ which are needed within the later sections.
In Section \ref{sec:europeanOptions}, 
we adapt the Mellin transform and PIDE methods of \levy processes to price European options under an exponential MAP model. 
In Section \ref{sec:martingaleConditions}, we look at the martingale properties of an exponential MAP and use this to compare the prices of Asian and European options.

\section{Integrability and Uniform Integrability of exponential MAPs}
\label{sec:integrability}
    Let $(J,\xi)$ be a MAP and $Y_t\coloneqq \exp(\xi_t)$ for all $t\geq 0$.
    In Sections \ref{sec:europeanOptions} and \ref{sec:martingaleConditions} we will want to consider the expectation of $Y$.
    Equivalent conditions for the existence of this expectation are given by Theorem \ref{thm:intLampKiu}. 

    There is a well known decomposition of MAPs, for example see \cite[pp 310, Part C, Chapter XI, Section 2]{asmussen2003applied}.
    It states that, for each $\alpha\in E$, there exists a \levy process,  $\xi^{(\alpha)}$, with characteristic triplets $(a_\alpha,\sigma_\alpha,\mu_\alpha)$. For each $\alpha,\beta\in E$, there exists an exponentially distributed random variable, $\zeta_{\alpha,\beta}$, with rate $q_{\alpha,\beta}\geq0$ and a random variable,  $U_{\alpha,\beta}$, taking values in $\R$, with measure $\nu_{\alpha,\beta}$. There are sequences $(\xi^{(\alpha,k)})_{k\in\N}, (\zeta_{\alpha,\beta,k})_{k\in\N}$ and $(U_{\alpha,\beta,k})_{k\in\N}$, which are i.i.d. copies of $\xi^{(\alpha)}$, $\zeta_{\alpha,\beta}$ and $U_{\alpha,\beta}$, respectively, such that the sequences are also independent of each other. To simplify notation, also define $q_\alpha\coloneqq -q_{\alpha,\alpha}\coloneqq\sum_{\gamma\in E\setminus\{\gamma\}}q_{\alpha,\gamma}$.
    
    These objects are such that the process $J$ is a continuous time Markov chain with transition rate matrix $(q_{\alpha,\beta})_{\alpha,\beta\in E}$. Let $\{T_n\}_{n\in\N_0}$ denote the times at which $J_t$ changes value, with the convention $T_0=0$. Then, $\zeta_{k}\coloneqq\zeta_{J_{T_{k-1}},T_k,k}=T_{k}-T_{k-1}$ for each $k\in\N$, hence, $T_n:= \sum_{k=0}^{n-1}\zeta_k$. Setting,
    \begin{align*}
        N_t         := \max_{n\in\N_0}\left\{T_n \leq t\right\}, 
        \qquad\text{and}\qquad
        \sigma_t    := t - T_{N_t},
    \end{align*}
    the process $(\xi_t,t\geq0)$ is given by
    \begin{align}
        \label{eqn:lampKiuDecomp}
        \xi_t &= \xi^{(N_t)}_{\sigma_t} + \sum_{k=0}^{N_t-1}\left( \xi^{(k)}_{\zeta_k-} + U_k \right),
        \qquad t\geq0,
    \end{align}
    where, for each $k\in\N$, we let $\xi^{(k)}:=\xi^{(J_{T_{k-1}},J_{T_k})}$  and $U_k:=U_{J_{T_{k-1}},J_{T_k}}$.
    
    Moreover, any process of this form is a MAP. We will  assume throughout that $J$ is irreducible and ergodic.

    
    
   We will also consider the matrix exponent of a MAP from \cite[pp 311, Section XI.2b]{asmussen2003applied}. For all $z\in\C$, such that $\E\left[e^{z\xi_1}\right]<\infty$,  it was shown that there exists a matrix $F(z)\in\C^{|E|\times|E|}$, such that for all $\alpha,\beta\in E$ and $t\geq0$,
    \begin{equation}
        \E\left[ e^{z\xi_t}; \: J_t = \beta \:\middle|\: J_0=\alpha \right] = \left(e^{tF(z)}\right)_{\alpha,\beta}.
    \end{equation}
    Moreover, $F(z)$ is of the form
    \begin{equation}
      \label{eqn:exponentExplicit}
      \left(F(z)\right)_{\alpha,\beta} = \begin{cases}
        \psi_\alpha(z) - q_{\alpha},    &\text{ if } \alpha=\beta;\\
        q_{\alpha,\beta}G_{\alpha,\beta}(z),                   &\text{ if } \alpha\neq \beta;
      \end{cases}
    \end{equation}
    where, $\psi_\beta(z):=\log \E[\exp(z\xi^{(\beta)}_1)]$ and $G_{\alpha,\beta}(z)=\E[\exp(zU_{\alpha,\beta})]$, for all $\alpha,\beta\in E$. A detailed discussion of the matrix exponent $F$ can be found in \cite[pp 311, Section XI.2b]{asmussen2003applied}. 

    %
    
    Suppose that $\kappa(z)$ is the principal eigenvalue of $F(z)$, that is, $\kappa(z)$ is the eigenvalue with largest real part. Then, it is known (see for example \cite[pp 9, Proposition 3.4]{Kuznetsov2014} and \cite{stephenson2017expfunc}) that $\kappa(0)=0$ and $\kappa$ is a continuous, convex function, over the regions of $\R$ in which it is well defined. The MAP $(J,\xi)$ is said to satisfy \cramers condition if there exists some $\theta>0$ such that $\kappa(\theta)=0$. Then, $\theta$ is referred to as \cramers number. If \cramers condition is not satisfied and $\kappa(z)<0$ for all $z>0$, set $\theta=+\infty$, whilst if $\kappa(z)>0$ for all $z>0$, set $\theta=0$. We then refer to $\theta$ as the extended \cramers number.
    
    An integrability result can now be established for $Y$, which will be required in Sections \ref{sec:europeanOptions} and \ref{sec:martingaleConditions}.
    
    \begin{thm}[Integrability of an exponential MAP]
        \label{thm:intLampKiu}
        Suppose  $(J,\xi)$ is a MAP, with decomposition (\ref{eqn:lampKiuDecomp}) and matrix exponent $F$, and that $Y_t\coloneqq\exp(\xi_t)$ for all $t\geq0$. Then, for all $p>0$, the following are equivalent:
        \begin{enumerate}
            \item \label{list:integrability:Lp}  $\E[Y_t^p]<\infty$ for all $t\geq0$;
            \item \label{list:integrabilityLpatT} There exists some $T>0$ such that $\E[Y_T^p]<\infty;$
            \item \label{list:integrability:local} $\{Y_t^p:t\geq0\}$ is locally integrable;
            \item \label{list:integrability:decomp} $\E\left[\exp\left(p\xi_1^{(\alpha)}\right)\right]<\infty$ and $\E\left[\exp\left(pU_{\alpha,\beta}\right)\right]<\infty$, for all $\alpha,\beta\in E$;
            \item \label{list:integrability:measures} $\mu_\alpha$ and $\nu_{\alpha,\beta}$ have $p$-exponential moments for all $\alpha,\beta\in E$;
            \item \label{list:integrability:Fp} $F(p)$ exists.
        \end{enumerate}
        Moreover, if a \cramers number, $\theta$, exists, then $\{Y^p_t:t\geq0\}$ is uniformly integrable, if and only if, $\theta>p$.
        
        \if{} 
                    Suppose that $(J,\xi)$ is a MAP with Lamperti-Kiu decomposition (\ref{eqn:LKdecomp}). Then, for all $t\geq0$, $Y_t^{(x)}:=x\exp(\xi_t)$ is integrable if and only if $\xi_1^\sigma$ and $U^{(\alpha,\beta)}$ have exponential moments for all $\sigma,\alpha,\beta\in E$. Moreover, if $Y$ is integrable, then there exists $\tilde{B},\hat{B}\in[0,\infty)$ such that
                    $
                        \tilde{B}^t \leq \E\left[\left|Y_t^{(1)}\right|\right] \leq \hat{B}^t.
                    $
        \fi
    \end{thm}    
    
    This theorem is of particular interest as it relates the integrability of $Y$ with properties of the components of its decomposition (\ref{eqn:lampKiuDecomp}). It also shows that existence of $F(p)$ is necessary and sufficient for integrability of $Y^p$ without conditioning on the states of $J$. Statement (\ref{list:integrability:local}) extends the corresponding result of \levy processes (see \cite[Exercise 29, pp 49]{protter2005stochastic}) and ensures that any local martingale results for $Y$ immediately transfer to martingale results.

    \begin{rem}
      \label{rem:integrability:p=1}
      Notice that $Y^p_t=\exp(p\xi_t)$ and that $(J,p\xi)$ is also a MAP. The \levy processes in the decomposition (\ref{eqn:lampKiuDecomp}) of $(J,p\xi)$ are given by $(p\xi^{(\alpha)})_{\alpha\in E}$ and the jumps induced by changes of $J$ are $(pU_{\alpha,\beta})_{\alpha,\beta\in E}$. The matrix exponent of $(J,p\xi)$ is $F(pz)$, hence if $\theta$ is \cramers number for $(J,\xi)$, then $\theta/p$ is \cramers number for $(J,p\xi)$. Therefore, we need only consider the case $p=1$. 
    \end{rem}
    
    The proof of Theorem \ref{thm:intLampKiu} makes use of the following lemma, which is due to the fact $\{N_t:t\geq0\}$ is an alternating renewal process. 

    \begin{lem}
        \label{lem:bddNt}
        There exists a Poisson process, $\{\eta_t:t\geq0\}$, of rate $\lambda\coloneqq\max_{\alpha\in E}q_\alpha$, such that $N_t\leq\eta_t$, for all $t\geq0$.
    \end{lem}
    \begin{proof}
        \if{} 
                Recalling the Lamperti-Kiu decomposition (\ref{eqn:LKdecomp}), 
                define a new sequence of exponentially distributed random variables $\hat{\zeta}^k$, with rate $\lambda$. Then, for any positive bounded decreasing function $f:\R\rightarrow[0,\infty)$, we have 
                    $\E[f(\hat{\zeta}^k)] \geq \E[f(\zeta^k)]$ 
                for all $k\in\N_0$.
                
                Fix $t\geq0$ and consider  $N_t$ as a function of $(\zeta^k)_{k\in\N}$. Then, $N_t$ is a weakly decreasing function with respect to each $\zeta^k$. Thus, for any positive bounded increasing function $g:\R\rightarrow[0,\infty)$, notice that $g(N_t)$ is weakly decreasing with respect to each $\zeta^k$ and so
                \begin{equation*}
                    \E\left[ g(N_t(\hat{\zeta}^1,\hat{\zeta}^2,\dots)) \right] \geq \E\left[ g(N_t(\zeta^1,\zeta^2,\dots)) \right].
                \end{equation*}
                However, since the $\left(\hat{\zeta}^k\right)_{k\in\N_0}$ are i.i.d. exponential random variables of rate $\lambda$, \\ $P_t:=N_t(\hat{\zeta}^1,\hat{\zeta}^2,\dots)$ is a Poisson process of rate $\lambda$.
        \fi{}
        For each $i\in\N_0$, since $\zeta_i\sim\Exp(q_i),$ there is a random variable $X_i$, uniformly distributed on $[0,1]$, such that $\zeta_i = -\log(X_i)/q_i$. For each $n\in\N_0$, we have the inequality
        \begin{align*}
            T_n 
          \coloneqq \sum_{i=0}^{n-1}
          \zeta_i
            = \sum_{i=0}^{n-1}\frac{-\log(X_i)}{q_i}
            \geq \sum_{i=0}^{n-1}\frac{-\log(X_i)}{\lambda}
            = \sum_{i=0}^{n-1} \hat{\zeta}_i
            \eqqcolon \hat{T}_n,
        \end{align*}
        where $\{\hat{\zeta}_i\}_{i\in\N_0}$ is an i.i.d. sequence of exponential random variables of rate $\lambda$. Then, $\{\eta_t\coloneqq \text{argmax}_{n\in\N}\{T_n<t\}:t\geq0\}$, is Poisson process of rate $\lambda$, such that $\eta_t\geq N_t$, for all $t\geq0$.
        
    \end{proof}
    
    The following lemma establishes the equivalence of statements (\ref{list:integrability:Lp}) and (\ref{list:integrability:decomp}) of Theorem \ref{thm:intLampKiu}.
    \begin{lem}
        \label{lem:Lamperti-Kiu-integrability}
        $Y_t$ is integrable for all $t\geq0$, if and only if, $\E\left[\exp\left(\xi_1^{(\alpha)}\right)\right]<\infty$ and $\E\left[\exp\left(U_{\alpha,\beta}\right)\right]<\infty$, for all $\alpha,\beta\in E$. 
    \end{lem}
    
    \begin{proof}
        First, we suppose that the exponential moments of $\xi_1^{(\alpha)}$ and $U_{\alpha,\beta}$ exist for all $\alpha,\beta\in E$ and show that $Y_t$ is integrable for all $t\geq0$. Let $\hat{\xi} \coloneqq \xi^{(\hat{\alpha})}$, where $\hat{\alpha}\coloneqq\text{argmax}_{\alpha\in E}\left\{\E\left[\exp(\xi^{(\alpha)}_1)\right]\right\}$.
        Moreover, let $\left(\hat{\xi}^{(k)}\right)_{k\in\N}$
        be a sequence of i.i.d. copies of $\hat{\xi}$,
        which are also independent of $\xi^{(\alpha,k)}$, for all $k\in\N_0$ and $\alpha\in E$. Then, by the independence structure of the decomposition (\ref{eqn:lampKiuDecomp}), for all $k\in\N_0$ and $t\geq0$,
        \begin{align*}
          \E\left[ \exp\left(\xi_t^{(k)}\right) \right] = \E\left[ \exp\left(\xi_1^{(k)}\right) \right]^t &\leq \E\left[ \exp\left(\hat{\xi}^{(k)}_1\right) \right]^t = \E\left[ \exp\left(\hat{\xi}_t^{(k)}\right) \right].
        \end{align*}
        Hence, for $\sigma(t)\in[0,\infty)$ as defined in (\ref{eqn:lampKiuDecomp}),
        \begin{align}
            \label{eqn:xiHatInequality}
            \E\left[ \exp\left(\xi_{\sigma(t)}^{(k)}\right) \right] 
            &= \E\left[  \E\left[\exp\left( \xi^{(k)}_1 \right) \right]^{\sigma(t)} \right]
            \leq \E\left[  \E\left[\exp\left( \hat{\xi}^{(k)}_1 \right) \right]^{\sigma(t)} \right]
            = \E\left[ \exp\left(\hat{\xi}_{\sigma(t)}^{(k)}\right) \right].
        \end{align}
        
        Let $\mathcal{G}$ be the $\sigma$-algebra generated by $(T_n)_{n\in\N}$. Then, by the tower property and independence of the \levy processes from the other components of the decomposition (\ref{eqn:lampKiuDecomp}),
        \begin{equation}
            \label{eqn:eqpMomentDecomp}
            \E[Y_t] 
            =\E\left[ \E\left[ \exp\left( \xi_{\sigma(t)}^{(N_t)} + \sum_{k=0}^{N_t-1}\xi_{\zeta_{k-}}^{(k)}\right)\:\middle|\: \mathcal{G} \right] 
            \E\left[ \exp\left(\sum_{k=0}^{N_t-1} U_{k} \right) \:\middle|\: \mathcal{G} \right] \right].
        \end{equation}
        Consider the first of the two conditional expectations of (\ref{eqn:eqpMomentDecomp}). By independence and (\ref{eqn:xiHatInequality}), we have
        \begin{align*}
            \E\left[ \exp\left( \xi_{\sigma(t)}^{(N_t)} + \sum_{k=0}^{N_t-1}\xi_{\zeta_{k-}}^{(k)}\right)\:\middle|\: \mathcal{G}  \right]
            &=\E\left[ \exp\left( \xi_{\sigma(t)}^{(N_t)}\right) \:\middle|\:\mathcal{G} \right]\prod_{k=0}^{N_t-1}\E\left[ \exp\left( \xi_{\zeta_{k-}}^{(k)} \right) \:\middle|\: \mathcal{G} \right]\\
            &\leq\E\left[ \exp\left( \hat{\xi}_{\sigma(t)}^{(N_t)}\right) \:\middle|\: \mathcal{G} \right]\prod_{k=0}^{N_t-1}\E\left[ \exp\left( \hat{\xi}_{\zeta_{k-}}^{(k)} \right) \:\middle|\: \mathcal{G} \right]\\
            &=\E\left[ \exp\left( \hat{\xi}_{\sigma(t)}^{(N_t)} + \sum_{k=0}^{N_t-1}\hat{\xi}_{\zeta_{k-}}^{(k)}\right)\:\middle|\: \mathcal{G}  \right].
        \end{align*}
        Moreover, by definition, $t = \sigma(t) + \sum_{k=0}^{N_t-1}\zeta_{k}$ and since the increments of a \levy process are i.i.d.,
        \begin{align*}
            \E\left[ \exp\left( \xi_{\sigma(t)}^{(N_t)} + \sum_{k=0}^{N_t-1}\xi_{\zeta_{k-}}^{(k)}\right)\:\middle|\: \mathcal{G} \right]
            \leq\E\left[\exp\left( \hat{\xi}_t \right) \right].
        \end{align*}
            \if{}
            and similarly
            \begin{align*}
                \E\left[ \exp\left( \xi_{\sigma(t)}^{(x)} + \sum_{k=0}^{N_t^{(x)}-1}\xi_{\zeta^{s(x,k)-}}^{s(x,k)}\right)\:\middle|\: \sigma\left((T_n)_{n\in\N}\right)  \right] 
                \geq \E\left[\exp\left( \tilde{\xi}^{(x)}_t \right) \right].
            \end{align*}
            \fi 
        From its definition, $\hat{\xi}$ has finite exponential moments only if $\xi^{(\alpha)}$ does for all $\alpha\in E$, and so, the same holds for $\xi$.
        
        Now, consider the second conditional expectation of (\ref{eqn:eqpMomentDecomp}) and notice that $\left\{N_t:t\geq0\right\}$ is $\mathcal{G}$ measurable.
        Hence, it follows that
        \begin{equation*}
            \E\left[ \exp\left( \sum_{k=0}^{N_t-1} U_{k} \right) \:\middle|\: \mathcal{G} \right]
            = \prod_{k=0}^{N_t-1} \E\left[ \exp\left(U_{k}\right) \:\middle|\: \mathcal{G} \right]
            \leq \prod_{k=0}^{N_t}\hat{V},
        \end{equation*}
        where,
        $
            \hat{V} 
            := \max\left\{ \E\left[ \exp(U_{\alpha,\beta}) \right]\::\: \alpha,\beta\in E \right\} 
            \geq \max_{k\in\N}\E\left[ \exp\left( U_k \right) \:\middle|\: \mathcal{G} \right].
        $
        Notice that $\hat{V}<\infty$, if and only if, $\E\left[\exp\left(U_{\alpha,\beta}\right)\right]<\infty$, for all $\alpha,\beta\in E$. Recall from Lemma \ref{lem:bddNt}, that $N_t\leq \eta_t$, for all $t>0$, where $\{\eta_t:t\geq0\}$ is a Poisson process of rate $\hat{\lambda}:=\max_{\alpha\in E}q_\alpha$. Therefore, by standard results for Poisson processes,
        \begin{align*}
            \E\left[ \exp\left( \sum_{k=0}^{N_t-1} U_{k} \right) \right] 
            &\leq \E\left[ \prod_{k=0}^{\eta_t-1} \hat{V} \right]
            =\E\left[ \hat{V}^{\eta_t} \right]
            \leq\exp\left(\frac{t(\hat{V}-1)}{\hat{\lambda}}\right).
        \end{align*}
                    \if{}
                    \begin{align*}
                        \E\left[ \exp\left( \sum_{k=0}^{N_t^{(x)}-1} U^{s(x,k)} \right) \right] \geq 0.
                    \end{align*}
                    \fi
        Hence, we have
        $
        \E[Y_t] \leq  \hat{B}^t, 
        $
        where
        $\hat{B} := \E\left[\exp\left(\hat{\xi_1}\right)\right]\exp\left((\hat{V}-1)\hat{\lambda}^{-1}  \right)$ is a constant.
        Clearly, if $\E[\exp(\xi^{(\alpha)})]<\infty$ and $\E[\exp(U_{\alpha,\beta})]<\infty$ all $\alpha,\beta\in E$, then $\hat{B}$ is finite. Thus, we have a sufficient condition for $\E\left[Y_t\right]<\infty$, for all $t\geq0$.
        
        To prove necessity, suppose one of the components, $\xi^{(\alpha)}$ or $U_{\alpha,\beta}$ for some $\alpha,\beta\in E$, of the decomposition (\ref{eqn:lampKiuDecomp}) fails to have exponential moments. Then, with positive probability, that component appears in the product (\ref{eqn:eqpMomentDecomp}) and hence $\E[Y_t]=\infty$ for all $t>0$.
        
                \if{}
                    To prove necessity, first suppose there exists $\sigma\in E$ such that $\xi^\sigma$ doesn't have exponential moments. If $\Prob(J_0=\sigma)>0$ then, by positivity, for $t\geq0$,
                    \begin{align*}
                        \E[|Y_t|] 
                        &\geq \E[|Y_t|\:|\: J_0=\sigma, N_t=0]\Prob(J_0=\sigma )\Prob(N_t=0\:|\: J_0=\sigma)
                        = \E[\exp(\xi^\sigma_t)]\Prob( J_0=\sigma)e^{-q^\sigma t}
                    \end{align*}
                    Then, since $\Prob( J_0=\sigma)e^{-q^\sigma t}>0$ and by the assumption $\E[\exp(\xi^\sigma_t)]=\infty$, we have $\E[|Y_t|]=\infty$.
                    
                    Now, suppose $\Prob(J_0=\sigma)=0$ so that $J_0\neq\sigma$ a.s.. In this case, $\Prob(J_t=\sigma)>0$ for any $t>0$, since the Markov chain $J$ is irreducible and time homogeneous. Consider the evolution from this point. In particular, using the Markov additive property,
                    \begin{align*}
                        \E[|Y_t|] 
                        &\geq \E\left[|Y_t|;J_s=\sigma \: \forall s\in\left[ \frac{t}{2},t \right]\right] \\
                        &= \E\left[|Y_{t/2}|\:\middle|\: J_s=\sigma \: \forall s\in\left[ \frac{t}{2},t \right] \right]\E\left[\exp\left(\xi^\sigma_{t/2}\right)\right]\Prob\left(J_s=\sigma \:\forall s\in\left[ \frac{t}{2},t \right]\right).
                    \end{align*}
                    Then, since $Y_t$ takes values in $\R^*\coloneqq\R\setminus\{0\}$, the first expectation must be strictly positive. Since $K$ is a time homogeneous irreducible Markov chain, $\Prob\left(J_s=\sigma\:\forall s\in\left[\frac{t}{2},t\right]\right)>0$. Hence, $\E[|Y_t|]=\infty$ by the assumption $\E[\exp(\xi^\sigma_t)]=\infty$.
                
                    Similarly, if there exists $\alpha,\beta\in E$ such that $U^{(\alpha,\beta)}$  doesn't have a finite exponential moment, then decomposing $Y_t$ gives
                    \begin{align*}
                        \E[|Y_t|] &\geq \E[|Y_t|; J_t=\beta, J_{T_{N_t}-}=\alpha]\\
                        &= \Prob( J_t=\beta, J_{T_{N_t}-}=\alpha ) \E[\exp(U^{(\alpha,\beta})] \\
                        &\qquad\E\left[ \exp\left( \xi_{\sigma(t)}^{N_t} +
                        \xi_{\zeta^{N_t-1}-}^{N_t-1}
                        +\sum_{k=1}^{N_t-2}\left( \xi_{\zeta^{k-}}^{k} + U^{k} \right) \right) \:\middle|\: J_t=\beta, J_{T_{N_t}-}=\alpha \right].
                    \end{align*}    
                    This is infinite since the first expectations is infinite by assumption and the other terms are strictly positive.
                    
                    
                    Thus, the process $\{Y_t,t\geq0\}$ is integrable if and only if $\xi^{\sigma}$ and $U^{\alpha,\beta}$ have exponential moments for all $\sigma,\alpha,\beta\in E$.    
                \fi
    \end{proof}

    We now consider the conditions required for uniform integrability of $Y$. The following adaptation of \cite[pp 174, Section 2, Lemma 1]{WILLEKENS1987173} will be needed in the proof of the uniform integrability statement of Theorem \ref{thm:intLampKiu}.
    
    \begin{lem}
        \label{lem:willekensMAP}
        For any $T>0$ and $0<u_0<u$, the following bound holds:
        \begin{align}
            \label{eqn:willikensMAP}
            \Prob\left( \sup_{t\in[0,T]} \xi_t \geq u \right) \leq \frac{\Prob\left( \xi_T>u-u_0 \right)}{\min_{\alpha\in E}\Prob_\alpha\left( \inf_{s\in[0,T]} \xi_s \geq -u_0 \right)}.
        \end{align}
    \end{lem}
    \begin{proof}
        Consider the stopping time $S_u\coloneqq\inf\{t\geq0\:|\:\xi_t>u\}$. Since $\xi$ is \cadlag, $\xi_{S_u}\geq u$, hence
        \begin{align*}
            \Prob(S_u<T;\xi_T<u-u_0)
            &\leq \Prob(S_u<T;\xi_T-\xi_{S_u} < -u_0 )\\
            &\leq \Prob\left(S_u<T; \inf_{s\in[S_u,S_u+T]}(\xi_s - \xi_{S_u}) < -u_0 \right)\\
            &=\sum_{\alpha\in E}\Prob(S_u<T; J_{S_u}=\alpha )\Prob_\alpha\left( \inf_{s\in[0,T]}\xi_s < -u_0 \right).
        \end{align*}
        Then, we can obtain the inequality
        \begin{align*}
            \Prob(\xi_T\geq u-u_0) 
            &\geq \Prob(S_u<T) - \Prob(S_u<T;\xi_T<u-u_0)\\
            &\geq \Prob(S_u<T) - \sum_{\alpha\in E}\Prob(S_u<T;J_{S_u}=\alpha)\Prob_\alpha\left(\inf_{s\in[0,T]} <-u_0\right)\\
            &\geq \Prob(S_u<T)\min_{\alpha\in E}\Prob_\alpha\left(\inf_{s\in[0,T]}\geq-u_0\right),
        \end{align*}
        which can be rearranged to obtain (\ref{eqn:willikensMAP}). 
        
    \end{proof} 

    It is known that there are no strictly locally integrable \levy processes \cite[Exercise 29, pp 49]{protter2005stochastic}. By a straightforward adaptation of the proof, the same is true for exponentials of \levy processes. A corresponding result for $Y$ is derived in the following lemma. 
        
    \begin{lem}
        \label{lem:noStrictLoclMG}
        If $\left\{Y_t:t\geq0\right\}$ is locally integrable, then it is also integrable. Moreover, if $Y$ is a local martingale, then it is also a true martingale.
    \end{lem}
    \begin{proof}
      Suppose $Y$ is locally integrable and let $\{\tau_n\}_{n\in\N}$ be a localising sequence of stopping times.
      Define a new stopping time $\tau\coloneqq \min_{n\in\N}\left\{\tau_n:\tau_n>T_1\right\}$.

      Suppose $J_0=\alpha\in E$ and let $Y^{(\tau)}_{t}:=Y_{\tau\wedge t}$, for $t\geq0$, be the process $Y$ stopped at $\tau$. Since $T_1$ is also a stopping time, by local integrability $\E[Y_{\tau\wedge T_1}]<\infty$. However, since $T_1<\tau$ and since $J$ is a Markov chain,
      \begin{align*}
        \E_\alpha[Y_{\tau\wedge T_1}]
        =\E\left[\exp\left(\xi^{(\alpha)}_{T_1}\right)\right]\sum_{\beta\in E\setminus\{\alpha\}}\E[\exp(U_{\alpha,\beta}))]\frac{q_{\alpha,\beta}}{q_{\alpha,\alpha}}.
      \end{align*}
      Thus, if $\E\left[\exp\left(\xi^{(\alpha)}_{T_1}\right)\right]=\infty$ or $\E\left[\exp\left(U_{\alpha,\beta}\right)\right]=\infty$ for any $\beta\in E$, then $\E_\alpha[Y_{\tau\wedge T_1}]=\infty$ contradicting local integrability.

      Now consider $J_0=\beta\neq\alpha$. Then, let $S=\min_{n\in\N}\left\{J_{T_n}=\alpha\right\}$ and notice that for all $T>0$,
      \begin{align*}
        \E_\beta\left[Y_{T};T>S\right]
        =\int_0^T\E_\beta\left[Y_s\middle|S=s\right]\E_\alpha[Y_{T-s}]\Prob_\beta(S\in ds).
      \end{align*}   
      Since $J$ is ergodic, it follows that $\Prob_\beta(S<T)>0$, hence $Y$ is locally  integrable with respect to $\Prob_\beta$ only if it is locally integrable with respect to $\Prob_\alpha$ also.
      
      Hence, for any inital distribution of $J$, the process $Y$ is locally integrable, only if $\E\left[\exp\left(\xi^{(\alpha)}\right)\right]<\infty$ and $\E\left[\exp\left(U_{\alpha,\beta}\right)\right]<\infty$, for all $\alpha,\beta\in E$.
      However, by Lemma \ref{lem:Lamperti-Kiu-integrability}, these are precisely the conditions for $Y$ to be integrable. Hence, $Y$ is not stictly locally integrable.

      Now consider the second claim and suppose that $Y$ is a local martingale. For ease of notation, for each $T>0$ let $\bar{\xi}_T\coloneqq\sup_{t\in[0,T]}\xi_t$ and $\bar{Y}_T\coloneqq\sup_{t\in[0,T]}Y_t$.

      Then, for $K>1$, by integration by parts,
      \begin{align*}
        \E\left[ \exp\left( \bar{\xi}_T \right); \bar{\xi}_T > \log(K) \right]
        &= \lim_{x\rightarrow\infty}-\exp(x)\Prob\left(\bar{\xi}_T\geq x\right) + K\Prob\left(\bar{\xi}_T\geq\log(K)\right) + \int_{\log(K)}^\infty\exp(x)\Prob(\bar{\xi}_T\geq x)dx\\
        &= \lim_{x\rightarrow\infty}-x\Prob\left(\bar{\xi}_T\geq\log(x)\right)+ K\Prob\left(\bar{\xi}_T\geq \log(K)\right) + \int_K^\infty \Prob\left(\bar{\xi}_T\geq \log(x)\right) dx.
      \end{align*}
      Since $K>1$, we can choose $u_0\in(0,\log(K))$. Then, by applying Lemma \ref{lem:willekensMAP},
      \begin{align*}
           H(u_0)\E\left[ \bar{Y}_t; \bar{Y}_t>K \right]
            \leq \lim_{y\rightarrow\infty}-y\Prob\left(\bar{\xi}_T\geq\log(y)-u_0\right)+ K\Prob(\xi_T\geq \log(K) - u_0) + \int_K^\infty \Prob(\xi_T\geq \log(y) - u_0) dy,
        \end{align*}
        where $H(u_0)\coloneqq\min_{\alpha\in E}\Prob_\alpha\left( \inf_{s\in[0,T]}\xi_s\geq-u_0 \right)$ and taking $K$ and $u_0$ sufficently large ensures $H(u_0)>0$. Moreover, since $Y_T$ is integrable, $\lim_{y\rightarrow\infty}-y\Prob\left(\bar{\xi}_T\geq\log(y)-u_0\right)=0$. Then, rewriting in terms of $Y_T$ gives
        \begin{align}
          \label{eqn:supYT-K0}
          H(u_0)\E\left[ \bar{Y}_T; \bar{Y}_T>K \right]
          \leq
          K\Prob(Y_T\geq Ke^{-u_0}) + e^{u_0}\int_{Ke^{-u_0}}^\infty \Prob(Y_T\geq w)dw
          = e^{u_0}\E\left[ Y_T;Y_T\geq Ke^{-u_0} \right].
        \end{align}
        Hence, $
          \E\left[\bar{Y}_T\right] \leq K + \frac{e^{u_0}}{H(u_0)}\E\left[Y_T\right]<\infty,
          $ where the final inequality is due to the integrability of $Y$, that follows from the fact $Y$ can not be strictly locally integrable.

        Then, since $\tau_n\rightarrow\infty$ as $n\rightarrow\infty$ a.s., it follows that $Y^{(\tau_n)}_T\rightarrow Y_T$ as $n\rightarrow\infty$ a.s.. For each, $n\in\N$, the inequality $Y^{(\tau_n)}_T\leq\bar{Y}_T$ holds and by the above arugment $\bar{Y}_T$ is integrable. Hence, by the dominated convergence theorem, for each $t\in[0,T]$, a.s.
        \begin{align*}
          \E\left[Y_T\middle|\mathcal{F}_t\right]
          = \lim_{n\rightarrow\infty}\E\left[Y_T^{(\tau_n)}\middle|\mathcal{F}_t\right]
          = \lim_{n\rightarrow\infty} Y_t^{(\tau_n)}
          = Y_t,
        \end{align*}
        and so $Y$ is a martingale.
    \end{proof}

    We are now in a position to prove Theorem \ref{thm:intLampKiu}.
    
    \textit{Proof of Theorem \ref{thm:intLampKiu}}\newline
        Following Remark \ref{rem:integrability:p=1} after the statement of the theorem, we will show the result for the case $p=1$.    
        The equivalence of (\ref{list:integrability:Lp}) and (\ref{list:integrability:decomp}) and then of (\ref{list:integrability:Lp}) and (\ref{list:integrabilityLpatT}) is given by Lemma \ref{lem:Lamperti-Kiu-integrability}.
        The equivalence of (\ref{list:integrability:Lp}) and (\ref{list:integrability:local}) is given by Lemma \ref{lem:noStrictLoclMG}.
        The equivalence of (\ref{list:integrability:decomp}) and (\ref{list:integrability:measures}) follows from the result for \levy processes given in \cite[pp 159, Chapter 5, Section 25, Theorem 25.3]{sato1999levy}.
        The equivalence of (\ref{list:integrability:decomp}) and (\ref{list:integrability:Fp}) follows from equation (\ref{eqn:exponentExplicit}).
        
        Now consider the final statement regarding uniformly integrability.        
        It is known that there is a real left eigenvector $h\in\R^{|E|}$ of $F(1)$, corresponding to the principal eigenvalue $\kappa(1)$, which has non-negative enteries and is such that $\sum_{\alpha}h_\alpha=1$, for example see \cite[pp 5, Section 1, Proposition 1.3]{stephenson2017expfunc}. Thus, $h$ may be used as the initial distribution over $E$ of $J$. Moreover, $h$ is also a left eigenvector of $e^{tF(1)}$, corresponding to the eigenvalue $e^{t\kappa(1)}$. Let $\Prob_h$ and $\E_h$ denote the probability measure and corresponding expectation, respectivley, when $J$ has initial distribution given by $h$.

        For the case $\theta\leq1$, we first show that $Y$ is not uniformly integrable with respect to $\Prob_h$ and then use this to prove that $Y$ is not uniformly integrable with respect to any initial distribution of $J$.
        In this case,  $\kappa(1)\geq0$ hence, $h$ is a left eigenvector of $e^{tF(1)}$, corresponding to the eigenvalue $e^{t\kappa(1)}\geq1$. Thus,
        \begin{align*}
          \E_h\left[Y_t\right]
          = \sum_{\alpha\in E}\sum_{\beta\in E}h_\beta\E_\beta\left[ Y_t; J_t=\alpha \right]
          =\sum_{\alpha\in E}\left(h e^{tF(1)} \right)_\alpha
          =e^{t\kappa(1)}\sum_{\alpha\in E}h_\alpha
          \geq1,
        \end{align*}
        for all $t\geq 0$. However, under \cramers condition it is known that
        $\lim_{t\rightarrow\infty} t^{-1}\xi_t = \kappa'(0)<0$ almost surely and hence also in probability (for example see \cite[pp 313, Chapter XI, Section 2, Corollary 2.8]{asmussen2003applied} and \cite[pp 9, Section 2.3]{Kuznetsov2014}).
        By choosing $\epsilon\in(0,-\kappa'(0))$, there exists $\tau_1>0$ such that $\exp(t(\kappa'(0)+\epsilon))<\frac{1}{2}$ for all $t>\tau_1$. Moreover, by convergence in probability, for all $\delta>0$ there exists $\tau_2>0$ such that, for all $t>\tau_2$,
         \begin{align*}
             \delta 
             \geq \Prob_h\left(t^{-1}\xi_t- \kappa'(0)>\epsilon\right) 
             = \Prob_h(Y_t>\exp(t(\kappa'(0) + \epsilon))) ,
         \end{align*}
         and so for $t>\max(\tau_1,\tau_2)$,
         \begin{align*}
           \Prob_h\left( Y_t>\frac{1}{2} \right)
           \leq \Prob_h\left( Y_t > \exp\left(t \left(\kappa'(0) + \epsilon   \right) \right)  \right)
           \leq \delta,
         \end{align*}
         that is, $\Prob_h(Y_t>1/2)\rightarrow0$ as $t\rightarrow\infty$.
         
         Now suppose for contradiction that $Y$ is uniformly integrable with respect to $\Prob_h$. Then, for all $\gamma>0$, there exists $K>0$ such that $\E_h\left[ Y_t; Y_t>K \right]<\gamma$, for $t>0$. Hence, for all $t\geq0$,
         \begin{align*}
             1 
             \leq \E_h[Y_t]
             &= \E_h\left[ Y_t; Y_t<\frac{1}{2} \right] + \E_h\left[Y_t; \frac{1}{2}\leq Y_t\leq K \right] + \E_h\left[ Y_t; Y_t> K \right]\\
             &\leq \frac{1}{2} + K\Prob_h\left(Y_t\geq\frac{1}{2}\right) + \gamma.
         \end{align*}
         By taking the limit as $t\rightarrow\infty$ and using the above result, we obtain
         $1\leq \frac{1}{2} + \gamma$,
         which is clearly a contradiction for $\gamma<\frac{1}{2}$. Thus, in the case $\theta\leq1$, $Y$ isn't uniformly integrable with respect to $\Prob_h$.

         However, if $Y$ is not uniformly integrable with respect to $\Prob_h$,  then there must exist an $\alpha\in E$, such that $Y$ is not uniformly integrable with respect to $\Prob_\alpha$.        
         Now consider any $\beta\in E$. Then, for any $K>0$,
         \begin{align*}
            \E_\beta\left[ Y_t; Y_t > K \right]
            &\geq \E_\beta\left[ \exp\left( \xi_t \right); \xi_t>\log(K); J_1 = \alpha \right]\\
            &\geq \E_\beta\left[ \exp(\xi_1) \hat{\E}_\alpha\left[ \exp\left(\hat{\xi}_{t-1}\right); \hat{\xi}_{t-1}> \log(K) - \xi_1  \right] ; J_1=\alpha \right],
         \end{align*}
         where $\left(\hat{J},\hat{\xi}\right)$ is an independent and identically distributed copy of $(J,\xi)$, with corresponding expectation $\hat{\E}$.
         However, since $Y$ is not uniformly integrable with respect to $\Prob_\alpha$, there exists $\delta>0$ such that, \\$\limsup_{t\rightarrow\infty} \E_\alpha\left[ \exp(\xi_t); \xi_t > \log(K) \right] > \delta$, for all $K>0$.
         Thus,
         $
            \limsup_{t\rightarrow\infty} \E_\beta\left[ Y_t; Y_t>K \right]
            \geq \delta\E_\beta\left[ \exp\left(\xi_1\right); J_1=\alpha\right].
         $
         Then, since $\E_\beta\left[ \exp(\xi_1);J_1=\alpha \right]>0$, we don't have uniform integrability of $Y$ with respect to $\Prob_\beta$ for any $\beta\in E$ and so $Y$ is not uniformly integrable for any initial distribution of $J$, whenver $\theta\leq 1$.

            \if{}
                        We now consider the case $\hat{B}<1$. Suppose $\epsilon>0$ then, for all $K,t>0$,
                        \begin{align*}
                            \E\left[ |Y_t|;|Y_t|\geq K \right]
                            &\leq \E\left[ |Y_t| \right]
                            \leq \hat{B}^t,
                        \end{align*}
                        thus there exists $T>0$ such that $ \E\left[ |Y_t|;|Y_t|\geq K\right]<\epsilon $ for all $t>T$. 
            \fi
        
            \if{}
                        Since $F(1)$ is diagonalisable over $\C$ \textbf{REFERENCE} there is an  invertible matrix $P\in\C^{|E|\times|E|}$ such that 
                        \begin{align*}
                            F(1) = P\diag(\lambda_1, \cdots,\lambda_n)P^{-1}
                        \end{align*}
                        where $\lambda_1,\cdots,\lambda_n\in\C$ are the eigen value of $F(1)$ hence $\lambda_i<\kappa(1)<0$ for each $i\in\{1,\cdots,n\}$. From this we have
                        \begin{align*}
                            e^{tF(1)} = P\diag(e^{t\lambda_1},\cdots,e^{t\lambda_n})P^{-1}
                        \end{align*}
            \fi

      Now suppose that $\theta>1$. Then, $\kappa(1)<0$ and it follows that 
      $\E_h[Y_t;Y_t>K]\leq\sum_{\alpha\in E} h_\alpha e^{t\kappa(1)}\rightarrow0$,
      as $t\rightarrow\infty$. Moreover, since each entry of $h$ is strictly positive, $\E_\alpha[Y_t;Y_t>K]\rightarrow0$ as $t\rightarrow\infty$, for all $\alpha\in E$.
      Hence, for all $\epsilon>0$, there exists $T>0$ such that $\E_\alpha[Y_t;Y_t>K]<\epsilon$, for all $t>T$ and $K>0$.
        
      We now consider $t\in[0,T]$ and $K>1$.
      \if{}
      For ease of notation, let $\bar{Y}_T\coloneqq\sup_{t\in[0,T]}Y_t$ and $\bar{\xi}_T\coloneqq\sup_{t\in[0,T]}\xi_t$. Then, by integration by parts, we have
      \begin{align*}
        \sup_{t\in(0,T)} \E\left[ Y_t; Y_t>K \right]
        \leq \E\left[ \exp\left( \bar{\xi}_T \right); \bar{\xi}_T>K \right]
        = K\Prob(\bar{\xi}_T\geq \log(K)) + \int_K^\infty \Prob(\bar{\xi}_T\geq \log(y)) dy.
      \end{align*}
        Since $K>1$, we can choose $u_0\in(0,\log(K))$. Then, by applying Lemma \ref{lem:willekensMAP},
        \begin{align*}
            \sup_{t\in[0,T]} \E\left[ Y_t; Y_t>K \right]
            \leq \frac{K\Prob(\xi_T\geq \log(K) - u_0) + \int_K^\infty \Prob(\xi_T\geq \log(y) - u_0) dy}{H(u_0)},
        \end{align*}
        where $H(u_0)\coloneqq\min_{\alpha\in E}\Prob_\alpha\left( \inf_{s\in[0,T]}\xi_s\geq-u_0 \right)$. Setting $K_0 = Ke^{-u_0}$ and rewriting in terms of $Y_T$, we have
        %
        \begin{align*}
            \sup_{t\in[0,T]} \E\left[ Y_t; Y_t>K \right]
          \leq
          \frac{ e^{u_0}K_0\Prob(Y_T\geq K_0) + e^{u_0}\int_{K_0}^\infty \Prob(Y_T\geq w)dw}{H(u_0)}
          =\frac{ e^{u_0}\E\left[ Y_T;Y_T\geq K_0 \right]}{ H(u_0) }.
        \end{align*}
        \fi
        Taking the limit as $K\rightarrow\infty$ in (\ref{eqn:supYT-K0}) and letting $\bar{Y}_T\coloneqq\sup_{t\in[0,T]}Y_t$, gives
        \begin{align*}
          \lim_{K\rightarrow\infty}\sup_{t\in(0,T)} \E\left[ Y_t; Y_t>K \right]
          \leq \lim_{K\rightarrow\infty}\E\left[\bar{Y}_T;\bar{Y}_T>K\right]
          \leq \lim_{K\rightarrow\infty}\frac{ e^{u_0}\E\left[ Y_T;Y_T\geq Ke^{-u_0} \right]}{ H(u_0) }
          = 0,
        \end{align*}
        since $Y_T$ is integrable, for some $u_0$ sufficently large.
        Combined with the result for $t>T$, this implies $\{Y_t:t\geq0\}$ is uniformly integrable.


            \if{}
                We also have the related result.
                \begin{lem}
                    \label{lem:willekensMAPReturn}
                    For $j\in E$ let $\tau_j$ be the first time that $J$ returns to its initial state $j\coloneqq J_0$. Then, for any $0<u<u_0$ we have
                    \begin{align*}
                        \Prob_j\left( \sup_{t\in[0,\tau_j]} \xi_t \geq u \right) \leq \frac{\Prob_j\left( \xi_{\tau_j}>u-u_0 \right)}{\min_{i\in E}\Prob_i\left( \xi_{\tau_j} \geq -u_0 \right)}.
                    \end{align*}
                \end{lem}
                \begin{proof}
                    Let $S_u\coloneqq\inf\{t\geq0\::\:\xi_t>u\}$. Then, $\xi_{S_u}\geq u$ and so
                    \begin{align*}
                        \Prob_j(S_u<\tau_j;\xi_{\tau_j}<u-u_0)
                        &\leq \Prob_j(S_u<\tau_j;\xi_{\tau_j}-\xi_{S_u} < -u_0 )\\
                        &=\sum_{i\in E}\Prob_j\left(S_u < \tau_j; \xi_{\tau_j}-\xi_{S_u}<-u_0;J_{S_u}=i \right)\\
                        &=\sum_{i\in E}\Prob_j\left(S_u < \tau_j;J_{S_u}=i \right)\Prob_i\left( \xi_{\tau_j} < -u_0 \right)
                    \end{align*}
                    Then,
                    \begin{align*}
                        \Prob_j(S_u<\tau_j)
                        &\leq \Prob_j(\xi_{\tau_j}\geq u-u_0) + \Prob_j(\xi_{\tau_j}<u-u_0; S_u<\tau_j)\\
                        &\leq \Prob_j(\xi_{\tau_j}\geq u-u_0) + \sum_{i\in E}\Prob_j(S_u<\tau_j;J_{S_u}=i)\Prob_i\left( \xi_{\tau_j}<-u_0 \right)
                    \end{align*}
                    which can be rearranged to give
                    \begin{align*}
                        \sum_{i\in E}\Prob(S_u<\tau_j;J_{S_u}=i)\Prob_i\left( \xi_{\tau_j} \geq -u_0 \right) 
                        \leq \Prob_j(\xi_{\tau_j}\geq u - u_0).
                    \end{align*}
                    Thus,
                    \begin{align*}
                        \Prob(S_u<\tau_j) \leq \frac{\Prob_j(\xi_{\tau_j}\geq u-u_0)}{\min_{i\in E}\Prob_i\left( \xi_{\tau_j} \geq -u_0 \right)}
                    \end{align*}
                \end{proof}
                
                \begin{lem}
                    Let $\tau>0$ be the first return time of $\{J_t,t\geq0\}$ to $J_0$. Then, 
                    \begin{align*}
                        \E\left[ \sup_{t\in[0,\tau]} e^{\xi_t} \right] < \infty
                    \end{align*}
                \end{lem}
                \begin{proof}
                    In the $|E|=2$ case:
                    \begin{align*}
                        \E\left[ \sup_{t\in[0,T_2]} e^{\xi_t} \right]
                        &\leq \E\left[ \max\left\{ \sup_{t\in[0,T_1)}e^{\xi_t}, e^{\xi_{T_1}}, \sup_{t\in(T_1,T_2)}e^{\xi_t}, e^{\xi_{T_2}} \right\} \right]\\
                        &\leq \E\left[ \sup_{t\in[0,T_1)}e^{\xi_t} \right] + \E\left[ e^{\xi_{T_1}} \right] + \E\left[ \sup_{t\in(T_1,T_2)}e^{\xi_t}\right] + \E\left[ e^{\xi_{T_2}} \right] \\
                        &\leq \E\left[ \sup_{t\in[0,T_1)}e^{\xi_t} \right] + \E\left[ e^{\xi_{T_1}} \right] + \E\left[ \sup_{t\in(T_1,T_2)}e^{\xi_t}\right] + \E\left[ e^{\xi_{T_2}} \right]
                    \end{align*}
                \end{proof}

                The following result gives sufficient conditions for uniform integrability of $Y$.
                \begin{lem}
                    Let $\tau_j$ denote the first return time of $J$ to $j\coloneqq J_0$. If $C_1\coloneqq\E_j|Y_{\tau_j}| < 1$ then $\{Y_t,t\geq0\}$ is uniformly integrable. 
                \end{lem}
                \begin{proof}
                    Suppose $K>0$, then for $t>0$
                    \begin{align*}
                        \E\left[ |Y_t|; |Y_t|>K \right]
                        &\leq \E \left[ |Y_t| \right]\\
                        &=\E\left[ \E\left[ \exp(\xi_{T(t)}) \exp( \xi_{t}-\xi_{T(t)} ) \:\middle|\: \sigma(T(t)) \right]\right]\\
                        &=\E\left[ \E\left[ \exp(\xi_{T(t)}) \:\middle|\: \sigma(T(t)) \right] \E\left[\exp( \xi_{t}-\xi_{T(t)} ) \:\middle|\: \sigma(T(t)) \right]\right]\\
                        &=\E\left[ C_1^{M_t} \E_j\left[ \exp(\hat{\xi}_{t-T(t)})\:\middle|\:\sigma(T(t))\right] \right].
                    \end{align*}
                    However, \textbf{check we have enough independence here!!!!}
                    \begin{align*}
                        \E_j\left[ \exp(\hat{\xi}_{t-T(t)})\:\middle|\:\sigma(T(t))\right]
                        &\leq \E_j\left[ \sup_{s\in(T(t),S(t)}\exp(\xi_{s-T(t)}) \right]\\
                        &=\E_j\left[ \sup_{s\in(0,\tau_j)}\exp(\xi_s)  \right]
                        \eqqcolon C_2
                    \end{align*}
                    which is finite by the previous lemma and the assumption $\E_j|Y_{\tau_j}|<1$. Thus,
                    \begin{align*}
                        \E\left[ |Y_t|; |Y_t|>K \right]
                        \leq C_2\E\left[ C_1^{M_t} \right]\rightarrow 0
                    \end{align*}
                    as $t\rightarrow\infty$ since $C_1<0$ (\textbf{check convergence is strong enough}).
                \end{proof}
            \fi


        \if{}    
            The following result means that the integrability of $Y$ can't be recovered by conditioning on the state of $J$ at any fixed time. 
            
            \begin{lem}
                \label{lem:completeIntFail}
                Suppose $(J,\xi)$ is a MAP such that $\E\left[ \exp(\xi_T)\:|\:J_T=\sigma \right]=\infty$ for some $T>0$ and $\sigma\in E$. Then, $\E\left[ \exp(\xi_t) \: | \: J_t=\gamma \right]=\infty$ for all $t>0$ and $\gamma\in E$.
            \end{lem}
            \begin{proof}
                Suppose there is a $T\geq0$ and $\sigma\in E$ such that $\E[\exp(\xi_T)\:|\: J_T=\sigma]=\infty$. Then, by Theorem \ref{thm:intLampKiu}, for all $t\geq0$, $Y_t$ isn't integrable and, for some $\alpha,\beta\in E$, one of $\xi^\alpha_1$ and $U^{\alpha,\beta}$ doesn't have exponential moments. 
                
                Fix $t>0$ and $\gamma\in E$. First, suppose $\xi^\sigma_1$ doesn't have exponential moments.  Then, since $K$ is a time homogeneous irreducible Markov chain, $p\coloneqq\Prob(J_s=\sigma, \:\forall s\in[t/4,3t/4]; J_t=\gamma )>0$ and
                \begin{align*}
                    \E\left[ \exp(\xi_t);J_t=\gamma \right]
                    &\geq \E\left[ \exp(\xi_t);J_t=\gamma; J_s=\sigma, \:\forall s\in[t/4,3t/4] \right]\\
                    &\geq p\E\left[ \exp(\xi_{t/4}) \:\middle|\: J_{t/4}=\sigma \right]\E\left[ \exp\left( \xi_{3t/4} - \xi_{t/4} \right) \:\middle|\: J_s=\sigma\:\forall s\in[t/4,3t/4] \right]\\
                    &\qquad\qquad\qquad\times\E\left[ \exp(\xi_{t} - \xi_{3t/4}) \:\middle|\: J_{3t/4}=\sigma; J_t=\gamma \right]\\
                    &=\E\left[ \exp(\xi_{t/4}) \:\middle|\: J_{t/4}=\sigma \right]\E\left[ \exp\left( \xi^\sigma_{t/2} \right)\right]\E\left[ \exp(\xi_{t/4}) \:\middle|\: J_0=\sigma;J_{t/4}=\gamma \right].
                \end{align*}
                Hence, if $\xi_1^\sigma$ doesn't have exponential moments then, $\E[\exp(\xi_t);J_t=\gamma]=\infty$ also.
                
                We now consider the case where there exists $\alpha,\beta\in\ E$ such that $\E[\exp(U^{(\alpha,\beta)})]=\infty$ and again fix $t>0$. Let $\tau:=\inf\{s\geq0\:|\: J_{s-}=\alpha,J_s=\beta\}$. Then, $p\coloneqq\Prob(\tau<t\:;\:J_t=\gamma)>0$ by properties of the Markov chain $K$. Hence,
                \begin{align*}
                    \E\left[ \exp(\xi_t); J_t=\gamma \right]
                    &\geq \E\left[ \exp(\xi_t);J_t=\gamma;\tau<t \right]\\
                    &= p\E\left[ \exp(\xi_{\tau})\:\middle|\: J_t=\gamma;\: \tau<t \right]\E\left[ \exp\left( \xi_t - \xi_{\tau} \right)\:\middle|\:\tau<t;J_t=\gamma \right]\\
                    &=p \E\left[ \exp(\xi_{\tau-} +U^{(\alpha,\beta)})\:\middle|\: J_t=\gamma;\: \tau<t \right]
                    \E\left[ \exp\left( \xi_{t} -\xi_{\tau} \right)\:\middle|\:\tau<t;J_t=\gamma \right]\\
                    &=p \E\left[ \exp(\xi_{\tau-})\:\middle|\: J_t=\gamma;\: \tau<t \right]
                    \E\left[ \exp(U^{(\alpha,\beta)}) \right]
                    \E\left[ \exp\left( \xi_{t} -\xi_{\tau} \right)\:\middle|\:\tau<t;J_t=\gamma \right].
                \end{align*}
                Since the first and the third expectations are non-zero and the middle expectation is infinite by assumption, we have that $\E[\exp(\xi_t);\: J_t=\gamma]=\infty$.

                        \if{}      
                        \begin{align*}
                            &\E[\exp(\xi_t)\:;\:J_t=\gamma] \\
                            &\qquad= \E\left[ \exp\left( \xi_{\sigma(t)} + \sum_{k=0}^{N_t-1}\left( \xi_{\zeta^{s(x,k)}}^{s(x,k)} + U^{s(x,k)} \right) \right) \:;\: J_t=\gamma  \right]\\
                            &\qquad= \E\left[  \E\left[\exp(\xi_{\sigma(t)}) \:\middle|\: \mathcal{G} \right]\mathbbm{1}_{\{J_t=\gamma\}}
                            \prod_{k=0}^{N_t-1}\left( \E\left[ \exp\left(\xi_{\zeta^{s(x,k)}}^{s(x,k)} \right) \:\middle|\: \mathcal{G} \right]\E\left[ \exp\left(U^{s(x,k)} \right) \:\middle|\: \mathcal{G} \right] \right)  \right]\\
                            &\qquad\geq \E\left[ \mathbbm{1}_{\{J_t=\gamma\}} \E\left[\exp(\xi_{1}^-)\right]^{\sigma(t)} \E\left[\exp(\xi^{s(x,0)}_{1})\right]^{\zeta^{s(x,0)}}\E\left[\exp(\xi^{s(x,1)}_{1})\right]^{\zeta^{s(x,1)}}  \right.\\
                                &\qquad\qquad\left.\cdot\prod_{k=0}^{N_t-1}\left( \E\left[ \exp\left(\xi_{\zeta^{s(x,k)}}^{s(x,k)} \right) \:\middle|\: \mathcal{G} \right]\E\left[ \exp\left(U^{s(x,k)} \right) \:\middle|\: \mathcal{G} \right] \right)  \right]\E\left[ \exp(U_+) \right]\E\left[ \exp(U_-) \right]
                        \end{align*}

                        and  from the lack of exponential moments of one of $\xi^+_1,\xi^-_1,U^+,U^-$ and their independence from $\zeta^{s(x,k)}$ at least one of the inner expectations must be infinite giving the required result.
                    \fi
                
            \end{proof}
        \fi    
    \qed

\section{Pricing of European Options}
    \label{sec:europeanOptions}
    A European option on an asset with price process $\{Y_t\::\:t\geq0\}$ is a contract which at its maturity, some fixed time $T\geq0$, pays out $H(Y_T)$, where the payoff function, $H:\R^+\rightarrow\R$, is predetermined. In the case of a European call option, where the owner of the option has the right but not the obligation to buy the asset at some predetermined strike price $k$ at maturity, the payoff function is given by $H(x)\coloneqq\max(x-k,0)$. We will suppose that the risk free rate of interest is fixed at $r$. 
    
    Throughout the remainder of this paper we will assume that there is a Markov chain $J$, such that under a risk neutral probability measure $\Prob$, the process $(J,\log(Y))$ is a MAP. The Markov chain $J$ corresponds to the state of the market, allowing the behaviour of the price process to change when the market state changes.
    In \cite{palmowski2019}, it is shown that the market can be made complete by adding additional securities related to the jumps of $Y$ and changes of $J$. Then, an equivalent martingale measure can be found under which $(J,\log(Y))$ remains a MAP.
    Let the $\{\mathcal{F}_t\}_{t\geq0}$ be the filtration of the equivalent martingale measure. For ease of notation, we set $\xi_t\coloneqq\log(Y_t)$ for all $t\geq0$.
    
    By standard no arbitrage arguments, the price of the European option, with payoff $H$ and maturity $T$, at time $t\in(0,T)$ is given by
    \begin{align*}
         e^{-r(T-t)}\E\left[H(Y_T)\:\middle|\:\mathcal{F}_t\right].
    \end{align*}
    From the Markov additive property, this is a function of the current value of the MAP, $(J_t,\xi_t)$, and the time to maturity, $T-t$, and is given by 
    \begin{align*}
        e^{-r(T-t)}\E_{(J_t,\xi_t)}\left[ H\left( \hat{Y}_{T-t} \right) \right],
    \end{align*}
    where $(\hat{J},\hat{Y})$ is an independent and identically distributed copy of $(J,Y)$. Throughout the remainder of this section we will denote by $C_H(y,\alpha,\tau)$ the price of the European option, with payoff function $H$ and time until maturity $\tau$, if the current market state is given by $(\alpha,y)\in E\times\R^+$. That is, for $(\alpha,y)\in E\times\R^+$ and $0\leq\tau$,
    \begin{align*}
      C_H(\alpha,y,\tau) = e^{-r\tau}\E_{(\alpha,\log(y))}\left[H\left(\hat{Y}_{\tau}\right)\right].
    \end{align*}
    

    Under an exponential \levy model, two common techniques for pricing European options are integral transform methods, for example see \cite{fadugba2016valuation}, \cite{carr1999option}, and solving a Partial Integro-Differential Equation (PIDE), for example see \cite[Chapter 12]{tankov2003financial}. We will adapt these two methods to exponential MAP models.

\if{}
\subsection{Fourier Transform} 
    \label{sssec:fourierTransform}
    The Fourier transform approach is useful when we know the matrix exponent of a process but not the corresponding density function. Even for simple MAPs, it can be hard to find a closed form expression for their transition densities, whereas the Matrix exponent can be easily written down. Whilst analytical Fourier inversions are often not known, the Fast Fourier Transform (see for example \cite{burrusfast}) provides an efficient way of computing them numerically, making this a useful method in practice.
 
    For each $\sigma\in E$, let $g_\sigma(x) := H(e^x)$ so that the European option payoff function is given by
    $ %
            H(Y_T) = \sum_{\sigma\in E}g_\sigma(\xi_T)\mathbbm{1}_{\{J_T=\sigma\}}.
    $ %
        
    We will follow the convention that, for a function $f\in L^2(\R)$, it's Fourier transform and Fourier inverse transform are, respectively, given by
    \begin{equation*}
        \hat{f}(u) := \left( \mathcal{F}f \right)(u) := \int_\R f(u)e^{ixu}dx
        \qquad\text{and}\qquad
        f(x) := \left( \mathcal{F}^{-1}\hat{f}  \right)(x) = \frac{1}{2\pi}\int_\R\hat{f}(u)e^{-iux}dx.
    \end{equation*}
    By analogy with \cite[Section 3.1]{carr1999option}, for $R\in\R$ and $a,b\in E$ we consider the intermediate function
    \begin{align*}
        c_{a,b}^{(R)}(T,y) := e^{-Ry}\E\left[ H(Y_T) \:;\: J_T=b \:\middle|\: J_0=a,\: \xi_0=y \right]
    \end{align*}
    with domain $\R^+\times\R$. The following lemma concerns the Fourier transform of $c^{(R)}_{a,b}$.
    
    \begin{lem}
        \label{lem:dampedFourier}
        Let $a,b\in E$ and $T\geq0$. If there exists an $R\in\R$ such that
        \begin{equation}
            \label{eqn:fourierRfinite}
            \E_a[|Y_{T}|^R;J_T=b] < \infty,
        \end{equation}
        \begin{equation}
            \label{eqn:fourierGfinite}
            g_b(x)e^{-Rx}\in L^1(\R)\cap L^2(\R)
        \end{equation}
        and
        \begin{equation}
            \label{eqn:fourierGL1}
            \hat{g}_b(u+iR)\in L^1(\R)
        \end{equation}
        then, as a function of $x$, $c^{(R)}_{a,b}(T,x)\in L^2(\R)\cap L^\infty(\R)$ and for all $u\in\R$,
        \begin{align}
            \label{eqn:smallCFT}
            c^{(R)}_{a,b}(T,y) = \mathcal{F}^{-1}_u\left\{\hat{g}_b(u+iR)\left( e^{TF(R-iu)} \right)_{a,b}\right\}(y),
        \end{align}
        where the inverse Fourier transform is taken with respect to $u$.
    \end{lem}
    \begin{proof}
        For fixed $T>0$, let
        $
            p_{a,b}(x) := \frac{\partial}{\partial x}\Prob\left( \xi_T \leq x; J_T=b  \:\middle|\: J_0=a,\: \xi_0=0 \right). 
        $ 
         where we consider the derivative in the distributional sense in necessary. Then, making use of the Markov additive property in the first equality; using Fubini's theorem and a change of variables in the second equality; and identifying a Fourier transform and the Matrix exponent in the last equality, we have
        \begin{multline*}
            \hat{c}^{(R)}_{a,b}(T,u) 
            = \int_\R e^{iuy} e^{-Ry} \int_\R g_b(x+y)p_{a,b}(x) dx dy\\
            = \int_\R e^{x(R-iu)} p_{a,b}(x) \int_\R e^{iy(u+iR)}g_b(y) dy dx
            = \hat{g}_b(u+iR) \left(e^{TF(R-iu)}\right)_{a,b}.
        \end{multline*}
        
        From (\ref{eqn:fourierRfinite}) and the definition of the matrix exponent,
        \begin{align*}
            |\left( e^{TF(R-iu)} \right)_{a,b}| \leq \E_a\left[ |Y_T|^R; J_T=b \right] < \infty. 
        \end{align*}
        Combining this with (\ref{eqn:fourierGL1}), we can conclude $\hat{c}^{(R)}_{a,b}(T,u)\in L^1(\R)$.
        Thus, we can apply a Fourier inversion to recover $c^{(R)}_{a,b}$ and obtain the result of the (\ref{eqn:smallCFT}).
        
        Moreover, since $g(x)e^{-Rx}\in L^1(\R)\cap L^2(\R)$, by Plancherel's Theorem \cite[Chapter 4.3, pp 187, Theorem 1]{evans2010partial}, 
        $
            \hat{g}(u+iR) = \mathcal{F}\{g(x)e^{-Rx}\}\in L^\infty(\R)\cap L^2(\R).    
        $
                \if{}
                The $L^2$-norm of $\hat{c}^{(R)}_{a,b}$ is then given by
                \begin{align*}
                    \|\hat{c}^{(R)}_{a,b}(T,\cdot)\|^2_{L^2(\R)} 
                    &= \int_\R \left( \hat{g}_b(u+iR)\left( e^{TF(R-iu)} \right)_{a,b}\right)^2 du\\
                    &\leq \int_\R \hat{g}_b^2(u+iR)\E_a\left[ Y_T^R \right]^2 du\\
                    &= \E_a\left[ Y_T^R \right]^2 \|\hat{g}_b^2(u+iR)\|^2_{L^2(\R)}\\
                    &<\infty
                \end{align*}
                \fi
        Hence,
        \begin{align*}
            \|\hat{c}^{(R)}_{a,b}(T,\cdot)\|^2_{L^2(\R)}
            \leq \E_a\left[ Y_T^R \right]^2 \|\hat{g}_b^2(u+iR)\|^2_{L^2(\R)}<\infty
        \end{align*}
        and
        \begin{align*}
            \| \hat{c}^{(R)}_{a,b} \|_{L^\infty(\R)}
            &= \sup_{u\in\R} \left| \hat{g}_b(u+iR)\left( e^{TF(R-iu)} \right)_{a,b} \right|
            \leq \| \hat{g}_b(u+iR) \|_{L^\infty(\R)}\E_a\left[ |Y_T|^R; J_T=b \right]
            <\infty
        \end{align*}
        so, $\hat{c}^{(R)}_{a,b}\in L^\infty(\R)\cap L^2(\R)$.
        Then, $c^{(R)}_{a,b}\in L^2(\R)\cap L^\infty(\R)$ follows from Placherel's theorem and standard integral results.
        
    \end{proof}

    By summing up (\ref{eqn:smallCFT}) over all $b\in E$ and applying the Fourier inversion theorem, we can obtain the following result.
    
    \begin{cor}[Fourier Transform Pricing under an exponential MAP]
        \label{cor:FTeuropeanOption}
    
        Suppose that for each $a,b\in E$, there exists $R_b\in\R$, such that
        \begin{equation*}
            \E_\sigma\left[ |Y_T|^{R_b}; J_T=b \right]<\infty,
        \end{equation*}
        \begin{equation*}
            g_b(x)e^{-R_b} \in L^1(\R) \cap L^2(\R)
        \end{equation*}
        and
        \begin{equation*}
            \hat{g}_b(u+iR)\in L^1(\R)
        \end{equation*}
        then, 
        \begin{equation*}
            C_H(y,a,t,T)= \frac{e^{-r(T-t)}}{2\pi}\sum_{b\in E}\int_\R e^{\log|y|(R_b-iu)}\hat{g}_b(u+iR_b)\left( e^{(T-t)F(R_b-iu)} \right)_{a,b}du.
        \end{equation*}
    \end{cor}
    
    \begin{proof}
        Initially we consider $r=t=0$ and sum up over the possible values of $J_T$ to obtain
        \begin{align*}
            C_H(y,\gamma,0,T) 
                &= \sum_{j\in E} \E_{(\gamma,y)}\left[ g_j(\xi_T); J_T=j \right]\\
                &= \sum_{j\in E} e^{R_j\log|y|} c^{(R_j)}_{\gamma,j}(0,T,\log|y|) \\
                &= \frac{1}{2\pi}\sum_{j\in E}\int_\R e^{\log|y|(R_j-iu)}\hat{g}_j(u+iR_j)\left( e^{TF(R_j-iu)} \right)_{\gamma,j} du
        \end{align*}
        where the last inequality comes from Lemma \ref{lem:dampedFourier}. 
        Then, by the Markov additive property, we can also obtain the result for $t\neq0$. If $r\neq0$ then we multiply by the discount factor $e^{-r(T-t)}$.

    \end{proof}
    
    As an example we will consider a European call option with strike $k\in\R^+$. In this case recall that $H(x) \coloneqq \max(x-k,0)$,  hence for all $\sigma\in E$ we have $g_\sigma(x) = \max(e^x - k,0)$.
    
    Then, computing the Fourier transform of $g_\sigma$ and taking $R_\sigma>1$, we have for all $u\in\R$,
    \begin{align*}
        \hat{g}_\sigma(u+iR_\sigma)
        &= k^{-R_\sigma+iu+1}\int_0^\infty e^{y(iu-R_\sigma+1)} - e^{y(iu-R_\sigma)} dy 
        = \frac{k^{iu-R_\sigma +1}}{(iu-R_\sigma +1)(iu-R_\sigma)}
        <\infty
    \end{align*}
    hence, for all $b\in E$ and $k\in\R^+$ we have the bound
    $
        |\hat{g}_b(u+iR_b)| 
        = \mathcal{O}(u^{-2})
    $  
    as $u\rightarrow\pm\infty$. Since $u,R_b\in\R$, then for $R_b\neq1$ there are no real roots for $u$ of the denominator, hence $\hat{g}_b(u+iR_b)\in L^2(\R)\cap L^1(\R)$ as a function of $u$. Then, the price of the European call option with strike $k\in\R^+$ is given by
    \begin{equation*}
        C_H(y,a,t,T)= \frac{e^{-r(T-t)}k}{2\pi}\sum_{b\in E}\int_\R \frac{e^{(\log(k)-\log(y))(iu-R_b)}}{(iu-R_b +1)(iu-R_b)}\left( e^{(T-t)F(R_b-iu)} \right)_{a,b}du.
    \end{equation*}

                \if{}
                \begin{prop}
            
                    In the case of a European call option with strike $k\in\R$, if $k>0$ then
                    \begin{equation*}
                        C_H(k,y,t,T)=   \frac{e^{-r(T-t)}}{2\pi}\int_\R e^{\log|y|(R_+-iu)}\hat{g}_+(u+iR_+)\left( e^{(T-t)F(R_+-iu)} \right)_{\sgn(y),+}du
                    \end{equation*}
                    and if $k<0$ then
                    \begin{multline*}
                        C_H(k,y,t,T)=e^{-r(T-t)}\left(|y|\left( e^{TF(1)} \right)_{\sgn(y),+} -k\Prob_y(Y_T>0)\right) \\+ \frac{e^{-r(T-t)}}{2\pi}\int_\R e^{\log|y|(R_--iu)}\hat{g}_-(u+iR_-)\left( e^{(T-t)F(R_--iu)} \right)_{\sgn(y),-}    du     
                    \end{multline*}   
                    where $\hat{g}_\pm$ is as specified above 
                            \if{}
                            \begin{equation*}
                                \hat{g}_+(u + iR_+) = \frac{k^{iu-R_++1}}{(iu-R_+)(iu-R_++1)} \qquad\text{ if }k>0,
                            \end{equation*}
                    
                            \begin{equation*}
                                \hat{g}_-(u +iR_-) =  \begin{cases}
                                                        \frac{(-k)^{iu-R_-+1}}{(iu-R_-)(iu-R_-+1)}  &\text{ if } k<0\\
                                                        0                                   &\text{ if } k\geq0
                                                    \end{cases}
                            \end{equation*}
                            \fi
                    and
                    \begin{equation*}
                        \Prob_y(Y_T>0) = \frac{1}{q_++q_-}\left( q_- + \sgn(y)e^{-(T-t)(q_++q_-)}q_{\sgn(y)} \right).
                    \end{equation*}
                \end{prop}
                \begin{proof}
                    We have already seen for $k\geq0$ that $g_-(x)=0$ and hence $\hat{g}_-(x)=0$ so the first result is obtained from Corollary \ref{cor:FTeuropeanOption} directly.
                    
                    If $k<0$ then $e^{-Rx}g_+(x)\notin L^1(\R)\cup L^2(\R)$ for any $R\in\R$ so Corollary \ref{cor:FTeuropeanOption} can't be applied directly. However, by Lemma \ref{lem:dampedFourier}
                    \begin{align*}
                        e^{rT}C_H(k,y,0,T)
                        &= \E_y\left[ g_+(\xi_T)\:; \: Y_T>0 \right] + \E_y\left[ g_-(\xi_T)\:;\: Y_T<0 \right]\\
                        &= \E_y\left[ Y_T \:;\: Y_T>0 \right] - k\Prob(Y_T>0) + e^{R\log|y|}c_{\sgn(y),-}^R(T,\log|y|)
                    \end{align*}
                    and then, by the definition of the matrix exponent, $\E_y\left[ Y_T \:;\: Y_T>0 \right]=|y|\left(e^{TF(1)}\right)_{\sgn(y),+}$, and standard Markov chain results give $\Prob(Y_T>0)$ so we can apply the inverse Fourier transform to $c^R_{\sgn(y),-}(T,\log|y|)$ to obtain the result.    
                \end{proof}    
                \fi

\fi

\subsection{Mellin Transform of \texorpdfstring{$C_H(\cdot,\cdot,\cdot)$}{C_H(.,.,.)}}
    Similarly to the Fourier transform methods used for \levy process (for example, see \cite{carr1999option} and \cite{fadugba2016valuation}), we consider a Mellin transform approach to pricing European options under an exponential MAP model. We consider both calls and puts, as well as general payoff functions $H:\R^+\rightarrow\R$.
    Following the standard convention, the Mellin transform of a function $f:\R^+\rightarrow\R$ is given by
    \begin{align*}
        \mathcal{M}\{f\}(u) \coloneqq \int_0^\infty x^{u-1}f(x)\,dx.
    \end{align*}

    We first consider the case of call and put options, which have pay off functions $H_+(x)\coloneqq(x-k)^+$ and $H_-(x)\coloneqq(k-x)^+$, respectivley. Let
   $C_\alpha(k) \coloneqq C_{H_+}(\alpha,1,T)$
    and
   $ P_\alpha(k) \coloneqq C_{H_-}(\alpha,1,T)$
    denote the prices of European call and put option, respectivley, with strike $k>0$ and time to maturity $T>0$, when the current state of the MAP is $(\alpha,0)\in E\times\R$.
    
    \begin{prop}
        \label{prop:putcallMellinTransform}
        If $\E[Y_T^{1+s}]<\infty$ for some $s>0$, then, for all $\alpha\in E$ and $u\in\C$ with $\Re(u)\in(0,s)$,
        \begin{align}
          \int_0^\infty k^{u-1}C_\alpha(k)\,dk 
          = \frac{e^{-rT}}{u(u+1)}\sum_{\beta\in E} \left(e^{TF(u+1)}\right)_{\alpha,\beta},
        \end{align}
        and, if $\E[Y_T^{-s}]<\infty$ for some $s>0$, then, for all $\alpha\in E$ and $\Re(u)\in(-s,-1)$,
        \begin{align}
          \int_0^\infty k^{u-1}P_\alpha(k)\,dk 
          = \frac{e^{-rT}}{u(u+1)}\sum_{\beta\in E} \left(e^{TF(u+1)} \right)_{\alpha,\beta}.
        \end{align}
    \end{prop}
    \begin{proof}
        For each $\alpha,\beta\in E$, let $p_{\alpha,\beta}$ denote the density of $Y_T\mathbbm{1}_{\{J_T=\beta\}}$ under the measure $\Prob_\alpha$. Then, for $u\in\C$ with $\Re(u)\in(0,s)$, we have
        \begin{align*}
            e^{rT}\int_0^\infty k^{u-1}C_\alpha(k) dk
            &= \int_0^\infty k^{u-1}\sum_{\beta\in E}\E_\alpha\left[ (Y_T-k)^+; J_T=\beta \right] dk\\
            &= \sum_{\beta\in E}\int_0^\infty k^{u-1} \int_k^\infty (y-k)p_{\alpha,
            \beta}(y)dy dk.
        \end{align*}
        Since $\Re(u)\in(0,s)$, using Fubini's theorem yields, 
        \begin{align*}
            \int_0^\infty k^{u-1} \int_k^\infty (y-k)p_{\alpha,\beta}(y)dy dk
            &=\int_0^\infty\int_0^y k^{u-1}(y-k)p_{\alpha,\beta}(y) dk dy
            = \int_0^\infty \frac{y^{u+1}p_{\alpha,\beta}(y)}{u(u+1)} dy,
        \end{align*}
        from which, using the definition of $e^{TF(z)}$, we obtain
        \begin{align*}
            e^{rT}\int_0^\infty k^{u-1}C_\alpha(k) dk
            =\sum_{\beta\in E}\int_0^\infty \frac{y^{u+1}p_{\alpha,\beta}(y)}{u(u+1)} dy
            = \sum_{\beta\in E}\frac{\left(e^{TF(u+1)}\right)_{\alpha,\beta}}{u(u+1)}.
        \end{align*}
        A similar calculation, with $\Re(u)<(-s,-1)$, yields the result for the put option.
        
    \end{proof}

    \if{}
                We can also consider the Mellin transform of a European option with respect to the current asset price. This is particularly useful when we want to consider a general payoff function $H:\R^+\rightarrow\R$.
                \begin{lem}
                    \label{lem:spaceMellinAsianOpt}
                    Let $C_i(y)$ denote the price of a European option with maturity $T$ and fixed strike $k$ when the current asset price is $Y_0=y$ and the market state is $J_0=j\in E$. Also, let $p_T(s)$ denote the density of $Y_T/y$ where $Y_T$ is the price of the underlying at time $T$. Then,
                    \begin{align}
                        \label{eqn:mellinOptionSpot}
                        C_i(y) 
                        = \mathcal{M}^{-1}_u\left\{\frac{k^{u+1}\mathcal{M}_s\{p_T\}(1-u)}{u(u+1)}\right\}(y)
                        = \sum_{j\in E}\mathcal{M}^{-1}_u\left\{ \frac{k^{u+1}}{u(u+1)}\left( e^TF(-u) \right)_{i,j} \right\},
                    \end{align}
                    where the inverse Mellin transform is calculated along the line $a+i\R$ with $a<-1$.
                \end{lem}
                \begin{proof}
                    Suppose that $\Re(u)<-1$. Then,
                    \begin{align*}
                        \mathcal{M}_y\left\{C\right\}(u)
                        &= \int_0^\infty y^{u-1}C(y) dy\\
                        &=\int_0^\infty y^{u-1} \int_0^\infty (ys-k)^+p_T(s) ds dy.
                    \end{align*}
                    By applying Fubini's theorem we have
                    \begin{align*}
                        M_y\left\{C\right\}(u)
                        &= \int_0^\infty \int_{k/y}^\infty y^{u-1}(ys-k)p_T(s) ds dy\\
                        &= \int_0^\infty \int_{k/s}^\infty y^{u-1}(ys-k)p_T(s)dy ds\\
                        &= \int_0^\infty \left[ \frac{y^{u+1}s}{u+1} - \frac{ky^u}{u} \right]^{y=\infty}_{y=k/s} p_T(s) ds
                    \end{align*}
                    and since $\Re(u)<-1$ this can be evaluated to
                    \begin{align*}
                        M_y\left\{C\right\}(u)
                        &= \int_0^\infty \left( -\frac{k^{u+1}s^{-u}}{u+1} + \frac{k^{u+1}s^{-u}}{u}\right) p_T(s) ds\\
                        &= \int_0^\infty\frac{k^{u+1}s^{-u}}{u(u+1)}p_T(s)ds\\
                        &= \frac{k^{u+1}}{u(u+1)}\int_0^\infty s^{-u}p_T(s) ds\\
                        &= \frac{k^{u+1}}{u(u+1)}\mathcal{M}_s\{p_T\}(-u+1).
                    \end{align*}
                \end{proof}
    \fi
    
    By taking the Mellin transform with respect to the current asset price, we can consider a European option with a general payoff function $H:\R^+\rightarrow\R$.
    
    \begin{prop}
        \label{prop:mellinGeneralPayoff}
        Suppose that the price process of an asset is given by $\{Y_t:t\geq0\}$ with initial condition $(J_0,Y_0)=(\alpha,y)\in E\times\R^+$. Let $H:\R^+\rightarrow\R$ be the payoff function of a European option and suppose that there exists an $s\in\R$ such that $\E[Y_T^{-s}]<\infty$ and $\left\{\mathcal{M}H\right\}(s)$ exists.
        Then, the Mellin transform of $C_H(\alpha,y,T)$, with respect to $y$, is given by
        \begin{align*}
            \left\{\mathcal{M}C_H(\alpha,\cdot,T)\right\}(z)
            = e^{-rT}\left\{ \mathcal{M}H \right\}(z)\sum_{\beta\in E}\left( e^{TF(-z)} \right)_{\alpha,\beta},
        \end{align*}
        for all $\alpha\in E$ and $z\in\C$ such that $\Re(z)=s$.

    \end{prop}
    
    \begin{proof}
        Fix $\alpha\in E$. Considering the Mellin transform of $C_H(\alpha,y,T)$ with respect to $y$ and using the Markov additive property we have, for $z\in\C$ with $\Re(z)\in(s-\epsilon,s+\epsilon)$,
        \begin{align*}
           e^{rT} \left\{ \mathcal{M}C_H \right\}(z)
            &= \int_0^\infty x^{z-1}C_H(\alpha,x,T) dx
              = \int_0^\infty x^{z-1}\sum_{\beta\in E}\int_0^\infty H(xu)p_{\alpha,\beta}(u) du dx
              ,
        \end{align*}
        where $p_{\alpha,\beta}$ is the density of $\mathbbm{1}_{\{J_T=\beta\}}Y_T$ with respect to $\Prob_\alpha$.
        Then, using Fubini's theorem and the substitution $y=xu$, yields
        \begin{align*}
            e^{rT}\left\{ \mathcal{M}C_H \right\}(z)
            &= \sum_{\beta\in E}\int_0^\infty p_{\alpha,\beta}(u)\int_0^\infty x^{z-1}H(xu) dx du
            = \sum_{\beta\in E}\int_0^\infty p_{\alpha,\beta}(u)u^{-z}\int_0^\infty y^{z-1}H(y) dy du.
        \end{align*}
        By separating the integrals we obtain
        \begin{align*}
            e^{rT}\left\{ \mathcal{M}C_H \right\}(z)
            &= \sum_{\beta\in E} \int_0^\infty y^{z-1}H(y) dy \int_0^\infty p_{\alpha,\beta}(u)u^{-z} du 
            =\sum_{\beta\in E} \left\{\mathcal{M}H\right\}(z)\left(e^{TF(-z)} \right)_{\alpha,\beta},
        \end{align*}
        where, by the assumptions, both $\{\mathcal{M}H\}(z)$ and $e^{TF(-z)}$ exist for all $z\in\C$ such that $\Re(z)=s$.
      \end{proof}

      \begin{rem}
        By standard results for Mellin transforms, the condition that $\left\{\mathcal{M}H\right\}(s)$ exists for some $s\in\R$ is satisfied if
        \begin{align}
          \label{eqn:Hasymptotics}
          \frac{H(x)}{x^{-s+\epsilon}}\rightarrow 0 \quad\text{ as }\quad x\rightarrow 0^+
          \qquad\text{and}\qquad
          \frac{H(x)}{x^{-s-\epsilon}}\rightarrow 0 \quad\text{ as }\quad x\rightarrow +\infty.
        \end{align}
        In particular, (\ref{eqn:Hasymptotics}) is statisfied for any $H:\R^+\rightarrow\R$ that is bounded by a polynomial, including the payoff functions $(x-k)^+$ and $(k-x)^+$ of a call and put, respectivley.
      \end{rem}

\subsection{Partial Integro-Differential Equation Approach}
    Another method to obtain the prices of European options under an exponential MAP model is through a Partial Integro-Differential Equation (PIDE), similar to the one derived for exponential \levy models in \cite{tankov2003financial}. To use this method, some regularity results on the option prices are required.
    

    The following proposition gives regularity conditions of $\rho^{(t)}_\alpha$, the density of $\xi_t$ for $t>0$ when $J_0=\alpha\in E$, from which the regularity of $p^{(t)}_\alpha$, the density of $Y_t$ when $J_0=\alpha$, follows. 
        
        \begin{lem}
            \label{lem:matrixexpint}
            Suppose that for all $\alpha\in E$, either $\sigma_\alpha^2>0$ or the \levy measure, $\mu_\alpha$, satisfies
            \begin{align}
                \label{eqn:levySmoothnessCondition}
                \liminf_{r\downarrow0}r^{\gamma-2}\int_{[-r,r]} |x^2| \mu_\alpha(dx) > 0,
            \end{align}
            for some $\gamma\in(0,2)$. Then, $\rho^{(t)}_\beta\in C^\infty(\R)$ and $\frac{\partial^n }{\partial x^n}\rho^{(t)}_\beta(x)\in L^1(\R)$, for all $n\in\N_0$ and $\beta\in E$.
        \end{lem}
        
        \begin{proof}
          From \cite[pp 8, Proposition 2.5 (v)]{sato1999levy}, for $\beta\in E$ and $z\in\R$
            \begin{align*}
              |\exp\left(\psi_\beta(iz)\right)|^2
              &= \exp\left( \psi_\beta(iz) + \psi_\beta(-iz) \right)
              =\exp\left( 2\int_\R\left( \cos(zx) -1 \right ) \mu(dx) \right).
            \end{align*}
            Moreover, following \cite[pp 190, Chapter 5, Proposition 28.3]{sato1999levy},  under condition (\ref{eqn:levySmoothnessCondition}), we have that, for small enough $r>0$,
            \begin{align*}
              \int_\R\left(\cos(zx)-1\right)\mu_\beta(dx) \leq -c_1z^\gamma,
            \end{align*}
            for all $z\in\R$, for some constant $c_1>0$. Thus,
            \begin{align*}
              2\Re\left(\psi_\beta(iz)\right)
              = \log\left( \left| \exp \left( \psi_\beta(iz)  \right)\right|^2 \right)
              \leq -2c_1z^\gamma,
            \end{align*}
            and so, $|\psi_\beta(iz)|\rightarrow\infty$ as $z\rightarrow\pm\infty$.
            In the case that (\ref{eqn:levySmoothnessCondition}) does not hold, we have assumed $\sigma_\beta^2>0$ and hence we still have that $|\psi_\beta(iz)|\rightarrow\infty$ as $z\rightarrow\pm\infty$.
            However, for each $\alpha,\beta\in E$ and all $z\in\R$, we have
            $
            |G_{\alpha,\beta}(iz)|
            \leq \E\left[ \left| U_{\alpha,\beta}^{iz}  \right| \right]
            =1.
            $
            Hence,
            for all $\epsilon>0$, there exists $R>0$ such that for all $|z|>R$ with $z\in\R$, we have
            $\epsilon\left(F(iz)\right)_{\alpha,\alpha}\geq \sum_{\beta\in E\setminus\{\alpha\}} \left(F(iz)\right)_{\alpha,\beta} $
          and
          $\epsilon\left(F(iz)\right)_{\alpha,\alpha}\geq \sum_{\beta\in E\setminus\{\alpha\}} \left(F(iz)\right)_{\beta,\alpha}  $.
          Then, for sufficently small $\epsilon$, $\|e^{F(iz)}\|_{L^1} =\max_{\alpha\in E}| e^{\psi_\alpha(iz)}|\leq exp\left(-2c_1z^\gamma\right)$, where $\|\cdot\|$ is the matrix norm induced by the $L^1$ norm.

          \if{}
                        Hence, there exists matricies $\Lambda,P\in\C^{|E|\times|E|}$ such that $\Lambda$ is diagonal, $\|P\|=1$ and
                        \begin{align*}
                            F(iz) = P^{-1}(iz)\Lambda(iz)P(iz).
                        \end{align*}
                        Then,
                        \begin{align*}
                            e^{F(iz)} = P^{-1}(iz) e^{\Lambda(iz)} P(iz),
                        \end{align*}
                    \fi

            
          Under the assumptions of the lemma,
          we may apply \cite[pp 190, Chapter 5, Proposition 28.1]{sato1999levy} to obtain that $\rho^{(t)}_\beta\in\C^\infty$ and $\lim_{x\rightarrow\pm\infty}\frac{\partial^n}{\partial x^n}\rho^{(t)}_\beta(x)\rightarrow0$ for all $n\in\N_0$ and $\beta\in E$.
          Moreover, we see that $\frac{\partial^{n}}{\partial x^{n}}\rho^{(t)}_\beta(x)\in L^1(\R)$.

            


            
        \end{proof}
    
        From this lemma we can now deduce smoothness with respect to $y$ of the option price $C_H(\gamma,y,T)$.

                \if{}
                    Since if $f^{(n-1)}\in L^1(\R)$ then $\lim_{x\rightarrow\pm\infty}f^{(n-1)}(x)=0$. So, by the FTC,
                    \begin{align*}
                        \int_\R f^{(n)}(x) dx = \lim_{x\rightarrow\infty}f^{(n-1)}(x) - \lim_{x\rightarrow-\infty}f^{(n-1)}(x) = 0
                    \end{align*}
                    hence $f^{(n)}\in L^1(\R)$.
                \fi
            
            
            
        
        \begin{cor}
            \label{cor:optionDiffableSpace}
            Suppose there exists an $a\in\R$, such that $x^{-(a+1)}H(x)\in L^1(\R^+)$ and $\E[Y_T^{a-1}]<\infty$. Then, under the conditions of Lemma \ref{lem:matrixexpint}, the  price of the option, $C_H(\alpha,y,T)$, is infinitely continuously differentiable as a function of $y\in\R^+$, for all $T>0$ and $\alpha\in E$.
        \end{cor}
        \begin{proof}
            For all $y\in\R^+$, $T\in\R^+$ and $\alpha\in E$, we have
            \begin{align*}
                C_H(\alpha,y,T) &= e^{-rT}\E_{(\alpha,y)}\left[ H(Y_{T}) \right]
                    =e^{-rT}\int_0^\infty H(x)p_\alpha^{(T)}\left(\frac{x}{y}\right)dx.
            \end{align*}
            
            Then, writing this in convolution form gives
            \begin{align*}
                C_H(\alpha,y,T)
                &=e^{-rT}y^{a+1}\int_0^\infty \left(x^{-a}H(x) p_\alpha^{(T)}\left(\frac{x}{y}\right)\frac{x^{a+1}}{y^{a+1}}\frac{1}{x} \right) dx\\
                &=e^{-rT}y^{a+1}\left(\left\{ x^{-a}H(x) \right\} * \left\{ p_\alpha^{(T)}(x^{-1})x^{-(a+1)} \right\}  \right)(y),
            \end{align*}
            where $*$ denotes the multiplicative convolution.
            Since we have assumed $x^{-(a+1)}H(x)\in L^1(\R^+)$, 
            by differentiation of a convolution, $C_H$ is $n$-times differentiable with respect to $y$, whenever $p_\alpha^{(T)}(x^{-1})x^{-(a+1)}$ is $n$-times differentiable with respect to $x$ on $\R^+$ and $\frac{\partial^m}{\partial x^m}\left( p_\alpha^{(T)}(x^{-1})x^{-(a+1)}\right)\in L^1(\R^+)$, for $m=1,\dots,n$.
            We now proceed with the proof of this condition.
            
            By the assumptions of the corollary and a change of variables, we have that
            \begin{align*}
                \int_0^\infty p_\alpha^{(T)}(x^{-1})x^{-(a+1)} dx
                = \int_0^\infty p_\alpha^{(T)}(z)z^{a-1} dz
                = \E\left[ Y_{T}^{a-1} \right]
                <\infty,
            \end{align*}
            hence, $p_\alpha^{(T)}(x^{-1})x^{-(a+1)}\in L^1(\R^+)$.
            Then, since we have sufficient differentability by Lemma \ref{lem:matrixexpint}, we can repeatedly apply the fundamental theorem of calculus to obtain\\ $\frac{\partial^n}{\partial x^n} \left(p_\alpha^{(T-t)}(x^{-1})x^{-(a+1)}\right)\in L^1(\R^+)$ for all $n\in\N$.
            
                \if{}
                    Since if $f^{(n-1)}\in L^1(\R)$ then $\lim_{x\rightarrow\pm\infty}f^{(n-1)}(x)=0$. So, by the FTC,
                    \begin{align*}
                        \int_\R f^{(n)}(x) dx = \lim_{x\rightarrow\infty}f^{(n-1)}(x) - \lim_{x\rightarrow-\infty}f^{(n-1)}(x) = 0
                    \end{align*}
                    hence $f^{(n)}\in L^1(\R)$.
                \fi
            
            
            
        \end{proof}
        
        
        \begin{lem}
            \label{lem:optionTimeSmoothness}   
            Suppose that there exists $c\in\R$ and $\epsilon>0$, such that $\E\left[ Y_{T}^{-c} \right]<\infty$ and $\{\mathcal{M}H\}(c)$ exists.
            Moreover, suppose that, for each $\beta\in E$, either $\sigma_\beta^2>0$, or the \levy measure, $\mu_\beta$, satisfies
            \begin{align*}
                \liminf_{r\downarrow0}r^{\epsilon-2}\int_{[-r,r]} |x^2| \mu_\beta(dx) > 0,
            \end{align*}
            for some $\epsilon\in(0,2)$.
            Then, the time derivative of $C_H$ exists and is continuous for $T>0$.
        \end{lem}
        \begin{proof}
          Since $\E\left[Y_T^{-c}\right]<\infty$, for each $\beta\in E$, we have $\lim_{x\rightarrow\pm\infty}e^{-cx}\rho_\beta^{(T)}(x)=0$. Moreover, $e^{-cx}\rho_\beta^{(T)}(x)\in C^\infty(\R)$ by Lemma \ref{lem:matrixexpint}, hence we have that $\frac{\partial}{\partial x}e^{-cx}\rho_\beta^{(T)}(x)\in L^1(\R)$.
          Then, by induction, we see that $\frac{\partial^n}{\partial x^n}e^{-cx}\rho_\beta^{(T)}(x)\in L^1(\R)$, for all $n\in\N$.
            
            By considering the Fourier transform, we have, for each $n\in\N$,
            \begin{align*}
                \left|u^n\E\left[ e^{(-c+iu)\xi_T} \right]\right|
                &= \left|u^n\mathcal{F}\left\{ e^{-cx}\rho_\beta^{(T)}(x) \right\}(-u)\right|
                = \left(2\pi\right)^{-n}\left| \mathcal{F}\left\{ \frac{\partial^n}{\partial x^n}e^{-cx}\rho_\beta^{(T)}(x) \right\}(-u) \right|,
            \end{align*}
            and so $u^n\E\left[ e^{(-c+iu)\xi_T} \right]\in L^\infty(\R)$. Thus, $\E\left[ e^{(-c+iu)\xi_T} \right]=o(u^{-n})$ as $u\rightarrow\pm\infty$ for all $n\in\N$ and, in particular, $\E[e^{(-c+iu)\xi_T}]\in L^n(\R)$ as a function of $u\in\R$, for all $n\in\N$.

            For each $\alpha\in E$, differentiating the result of Proposition \ref{prop:mellinGeneralPayoff} with respect to $T$ gives
            \begin{align}
              \label{eqn:mellinTimeDeriv}
                \frac{\partial}{\partial T}\left\{ \mathcal{M}_yC_H(\alpha,y,T) \right\}(s)
              = e^{-rT}\left\{ \mathcal{M}H \right\}(s)\sum_{\beta\in E}\left( \left(F(-s)-rI\right)e^{TF(-s)} \right)_{\alpha,\beta}
              ,
            \end{align}
            for all $s\in c+i\R$, since by the assumptions both $\left\{ \mathcal{M}H \right\}(s)$ and $\E\left[ Y_T^{-s} \right]$ exist and are finite.
            
            Suppose $s=c+iu$ with $u\in\R$, then for all $\alpha,\beta\in E$, we have
            \begin{align*}
                |G_{\alpha,\beta}(-(c+iu))|
                =\left|\E\left[ e^{-(c+iu)U_{\alpha,\beta}} \right]\right|
                \leq \E\left[\left| e^{-(c+iu)U_{\alpha,\beta}} \right|\right]
                =\E\left[ e^{-cU_{\alpha,\beta}} \right]
                <\infty,
            \end{align*}
            by Theorem \ref{thm:intLampKiu}, since $\E\left[Y_T^{-c}\right]<\infty$. Hence, $G_\beta(-s)$ is bounded over $c+i\R$.
            Since $\psi_\sigma$ is the Laplace exponent of a \levy process, we have that $\psi_\sigma(-(c+iu))=\mathcal{O}(u^2)$ as $|u|\rightarrow\infty.$ 
            We have already shown that  $e^{TF(-(c+iu))}=\left(\E_{\alpha}\left[ e^{-(c+iu)\xi_T};J_T=\beta \right]\right)_{\alpha,\beta\in E}=o(u^{-n})$ for all $n\in\N$ and by assumption $\left\{\mathcal{M}H\right\}(-(c+iu))$ is bounded over $-c+i\R$. Combining each of these terms shows that $\frac{\partial}{\partial T}\left\{ \left\{\mathcal{M}_yC_H(\alpha,y,T)\right\}(-c+iu) \right\}\in L^1(\R)$ with respect to $u$ and therefore we can apply the Mellin inverse transform along the line $-c+i\R$.         
        
            By standard results, the Mellin transform $\{\mathcal{M}H\}(s)$ is analytic, whilst $F(s)$ is analytic from Lemma \ref{lem:matrixexpint}. Moreover, it is clear that the result of Proposition \ref{prop:mellinGeneralPayoff} and (\ref{eqn:mellinTimeDeriv}) are both continuous functions of $T$.
            Hence, by the Leibniz integral rule, we pass the derivative through the integral in the definition of the inverse Mellin transform to obtain
            \begin{align*}
                \frac{\partial}{\partial T}C_H(\alpha,y,T)
                = \mathcal{M}^{-1}_s \left\{e^{-rT} \left\{ \mathcal{M}H \right\}(s)\sum_{\beta\in E}\left( \left(F(-s)-rI\right)e^{TF(-s)} \right)_{\alpha,\beta}  \right\}(y).
            \end{align*}            
            
            Now suppose $0<\tau<T$. Since for each $\alpha,\beta\in E$, we have $(e^{F(-(c+iu))})_{\alpha,\beta}\in L^1(\R)$ as a function of $u$, there exist real numbers $a<b$ such that $\|e^{F(-(c+iu))}\|<1$ for all $u\in\R\setminus(a,b)$. Let $M:=\max\left(\sup_{u\in[a,b]}\|e^{F(-(c+iu))}\|,1\right)$. Then, $M$ is finite since the interval $[a,b]$ is compact and $\| e^{F(-(c+iu))} \|$ is continuous. Hence, for all $t\in (T-\tau,T+\tau)$,
            \begin{align*}
                \left\|e^{tF(-(c+iu))}\right\| 
                \leq B(u) \coloneqq
                \begin{cases}
                        M^{T+\tau} ,                           & \text{ if } u \in[a,b];\\
                        \left\|e^{(T-\tau)F(-(c+iu))}\right\|,      & \text{ otherwise. }
                    \end{cases}
            \end{align*}
            Notice that the right-hand side is an $L^1$ function in $u$ and is constant with respect to $t\in(T-\tau,T+\tau)$. We can then obtain the bound, for all $t\in(T-\tau,T+\tau)$,
            \begin{align*}
                \left| \frac{\partial}{\partial T}\left( \left\{\mathcal{M}C_H(\alpha,y,T)\right\}(s) \right) \right|
                \leq e^{-r(T-\tau)}\left|\left\{ \mathcal{M}H \right\}(s)\right|\sum_{\beta\in E}\| F(-s)-rI\|B(u)\in L^1(\R). 
            \end{align*}
            Thus, by using the dominated convergence theorem in the integral of the inverse Mellin transform, $\frac{\partial}{\partial T}C_H(\alpha,y,T)$ is continuous in $T$. 
            
                \if{}
                            Using Corollary \ref{cor:FTeuropeanOption} for each $j\in E$,
                            \begin{align*}
                                \frac{\partial}{\partial T}C_H(T,y,j) = \sum_{b\in\{+,-\}} e^{R_b\log|y|}\mathcal{F}^{-1}\frac{\partial}{\partial T}\left(\hat{c}^{(R_b)}_{j,b}(T,\cdot) \right)(\log|y|)
                            \end{align*}
                            and so the time derivative exists and is continuous in $T$.
                \fi

            \if{}
            Moreover, since this is also satisfied by the characteristic functions of the \levy processes, it is true for the entries of $F(R-iu)$. Thus the time derivative of $C_H$ is rapidly decreasing if and only if $e^{TF(R-iu)}$  is rapidly decreasing which is a consequence of Proposition \ref{prop:matrixexpint}. 
            \fi 

        \end{proof}

                    \if{}
                    For European call options on $Y_T$ with $k<0$ we have seen that $g_+$ doesn't satisfy the conditions of Proposition \ref{prop:optionTimeSmoothness}. Instead we must study the time differentiablity of $|y|\left( e^{TF(1)} \right)_{\sgn(y),+} - k\Prob(Y_T>0)$. However the derivative of $e^{TF(1)}$ is $F(1)e^{TF(1)}$ and from Markov chain theory we know $\Prob(Y_T>0)$ is also continuously differentiable with respect to $T$.
                    \fi

        \begin{lem}
            \label{lem:optPriceSqInt}
            If the payoff function $H:\R^+\rightarrow\R$ is Lipschitz and $\E[Y_T^2]<\infty$, then $\E\left[C_H(J_t,Y_t,T-t)^2\right]<\infty$, for all $t\in(0,T)$.
        \end{lem}
        \begin{proof}
            Using Jensen's inequality and the tower property, followed by the Markov additive property, we have
            \begin{align*} 
                \E\left[ \left( C_H(J_t,Y_t,T-t) \right)^2 \right]
                &\leq e^{-2r(T-t)}\E\left[ \hat{\E}_{(J_t,\log(Y_t))}\left[ H( \hat{Y}_{T-t} )^2 \right] \right]
                = e^{-2r(T-t)}\E\left[  H(Y_T)^2 \right],
            \end{align*}
            where $\hat{Y}$ is an independent but indentically distributed copy of $Y$ and $\hat{\E}$ is the corresponding expectation.
            However, $H$ is Lipschitz with some constant $h$, so
            \begin{align*}
                \E\left[ \left( C_H(J_t,Y_t,T-t) \right)^2 \right]
                &\leq h^2e^{-2r(T-t)}\E\left[Y_T^2\right] +e^{-2r(T-t)}H(0),
            \end{align*}
            hence, $C_H(J_t,Y_t,T-t)$ has second moments if $Y_T$ does.
            
        \end{proof}
        
    Under certain conditions, the following proposition expresses the price of a European option as the solution of a PIDE with the payoff function as a boundary condition.

    \begin{prop}[PIDE for European Option prices]
      \label{prop:PIDE-europeanOptionPrices}
      Suppose that: \begin{enumerate}
            \item \label{lst:pide:secondMoments} $Y_T$ has finite second moments;
            \item \label{lst:pide:continuity} for each $\beta\in E$,
                \begin{enumerate}
                    \item either, $\sigma_\beta^2>0$;
                    \item or, there exists an $\epsilon\in(0,2)$ such that  $\liminf_{r\downarrow0}r^{\epsilon-2}
                \int_{-\epsilon}^\epsilon |x^2|\mu_\beta(dx)>0$;
                \end{enumerate}
            \item \label{lst:pide:lipschitz} the payoff function $H$ is Lipschitz;
            \item \label{lst:pide:hPoly} there exists an $s>2$ such that $H(x)\leq x^s$ in some neighbourhood of $0$. 
        \end{enumerate}
        Then, $C_H(\alpha,y,t)$ is twice continuously differentible with respect to $y$ and once continuously differentiable with respect to $t$ in the domain $E\times\R^+ \times\R^+$. Moreover, it  satisfies the PIDE
        \begin{align}
            \label{eqn:pide}
            \begin{aligned}
                0=& -rC_H(\alpha,y,t) - \partial_t C_H(\alpha,y,t) \\
                &+\partial_y C_H(\alpha,y,t) y\left( a_{\alpha} + \frac{\sigma^2_{\alpha}}{2} + \int_\R\left(e^u-1-u\mathbbm{1}_{|u|\leq1}\right)\mu_{\alpha}(du) \right)\\
                &+ \frac{1}{2}\partial_{y^2} C_H(\alpha,y,t) y^2\sigma^2_{\alpha} \\
                &+ \int_\R \left(C_H(\alpha,ye^u,t)-C_H(\alpha,y,t)  - y(e^u-1)\partial_{y}C_H(\alpha,y,t)\right) \mu_{\alpha}(du)\\
                &+\sum_{\gamma\in E\setminus\{\alpha\}}\int_\R \left( C_H(\gamma,ye^{u},t) - C_H(\alpha,y,t)  \right)\nu_{\alpha,\gamma}(du),
            \end{aligned}
        \end{align}
        on the domain $E\times\R^+\times\R^+$, with the boundary condition
        \begin{align}
            \label{eqn:pideBoundary}
            C_H(\alpha,y,0) = H(y), \qquad \forall (\alpha,y)\in E\times\R^+.
        \end{align}
    \end{prop}
    
    \begin{proof}
      First, fix the maturity of the option $T>0$ and we will show the result for the domain $E\times\R^+\times(0,T)$. Then, since $T>0$ is arbitary, this can be extended to the full domain $E\times\R^+\times\R^+$.

      From condition (\ref{lst:pide:hPoly}) it follows that $H(0)=0$, whilst by the Lipschitz property, $|H(x)|\leq hx$, for all $x\geq0$, where $h$ is the Lipschitz constant.
      Hence, for any $a>1$, $|x^{-(a+1)}H(x)| \leq h|x|^{-a}$. For $a<s-1$,  by (\ref{lst:pide:hPoly}),  $|x^{-(a+1)}H(x)|\leq x^{-a-1+s}$ in a neighbourhood of $0$. Thus, for all $a\in(1,s-1)$, it follows that $x^{-(a+1)}H(x)\in L^1(\R^+)$ and so $\{\mathcal{M}H\}(-a)$ exists.

      For $a\in(1,3)$, we have $\E[Y_T^{a-1}]<\infty$ by the assumption $Y_T$ has second moments. Since $s>2$, we can take $a\in(1,3)\cap(1,s-1)$, then  the conditions of Corollary \ref{cor:optionDiffableSpace} hold and so $C_H(\alpha,y,T)$ is infinitley continuously differentiable as a function of $y\in\R^+$, for all $T>0$ and $\alpha\in E$.

      Moreover, the conditions on $H$ of Lemma \ref{lem:optionTimeSmoothness} are satisfied for $c=-a$, whilst the other conditions of Lemma \ref{lem:optionTimeSmoothness} are given directly by the assumtions of Proposition \ref{prop:PIDE-europeanOptionPrices}. Conquently, $C_H(\alpha,y,T)$ is continuously differentiable with respect to $T$, for all $(\alpha,y,T)\in E\times\R^+\times\R^+$.

        For each $\alpha,\beta\in E$, denote by $\tilde{N}_\beta$ the compensated Poisson random measure associated with the \levy process $\xi^{(\beta)}$. Also, let $\mathcal{M}_{\alpha,\beta}$ be the Poisson random measure associated with the jumps of $\xi$ induced by a change in $J$ from $\alpha$ to $\beta$, which has intensity $q_{\alpha,\beta}$. Then, let $\tilde{\mathcal{M}}_{\alpha,\beta}$ denote the compensated Poisson random measure associated with $\mathcal{M}_{\alpha,\beta}$ and let $\nu_{\alpha,\beta}(du)ds$ be the corresponding density. 
      
        Since $\{Y_t:t\geq0\}$ is integrable, by an adaptation of the semi-martingale decomposition given in \cite[pp 10]{Doring2012JumpSDE}, we have
        \begin{align}
          \label{eqn:YsemiMartingale}
          \begin{aligned}
            Y_t -Y_0=& \int_0^t\sigma_{J_s}Y_s dW(s)
            + \int_0^t Y_{s-} \left(a_{J_s} + \frac{\sigma_{J_s}^2}{2} + \int_{|u|\leq1} \left( e^u -1-u \right) \mu_{J_{s-}}(du) \right) ds \\
            & + \int_0^t\int_{|u|\leq 1} Y_{s-}(e^u-1)\tilde{N}_{J_{s-}}(ds,du)    + \int_0^t\int_{|u|>1} Y_{s-}(e^u-1)N_{J_{s-}}(ds,du) \\
            & + \sum_{\gamma\in E\setminus\{\alpha\}}\int_0^t\int_{\R}Y_{s-}(e^u-1)\mathcal{M}_{J_{s-},\gamma}(ds,du),
          \end{aligned}
        \end{align}
        where $(W_s)_{s\geq0}$ is a Brownian motion.

        Consider the discounted price process $\hat{C}_H(J_t,Y_t,T-t) := e^{-rt}C_H(J_t,Y_t,T-t)$, which is a local martingale under the risk neutral measure. Then, since $C_H$ is continuously differentiable with respect to $T$, we can apply Ito's Lemma to $\hat{C}_H$ to obtain, after some simplification,
        \begin{align*}
            \hat{C}_H(J_t,Y_t,T-t) - \hat{C}_H(J_0,Y_0,T)
            &= \int_0^t a(s) ds + M_t,
        \end{align*}
        where,
        \begin{align}
          \label{eqn:defnCtsA}
          \begin{aligned}
            a(s) &=  -\partial_t \hat{C}_H(J_{s-},Y_{s-},T-s)
            + \partial_x \hat{C}_H(J_{s-},Y_{s-},T-s)Y_{s-}\left( a_{J_{s-}} + \frac{\sigma^2_{J_{s-}}}{2} \right)
            \\&+ \frac{1}{2}\partial_{x^2} \hat{C}_H(J_{s-},Y_{s-},T-s) Y_{s-}^2\sigma^2_{J_{s-}}\\
            &+ \int_{|u|\geq1} \left(\hat{C}_H(J_{s-},Y_{s-}e^u,T-s)-\hat{C}_H(J_{s-},Y_{s-},T-s)  \right) \mu_{J_{s-}}(du)\\
            &+ \int_{|u|<1} \left( \hat{C}_H(J_{s-},Y_{s-}e^u,T-s )-\hat{C}_H(J_{s-},Y_{s-},T-s) - uY_{s-}\partial_{x}\hat{C}_H(J_{s-},Y_{s-},T-s) \right)\mu_{J_{s-}}(du)\\
            &+ \sum_{\gamma\in E\setminus\{J_{s-}\}}\int_\R \left( \hat{C}_H(\gamma,Y_{s-}e^{u},T-s) - \hat{C}_H(J_{s-},Y_{s-},T-s)  \right)\nu_{J_{s-},\gamma}(du)
            \end{aligned}
          \end{align}

        and
        \begin{multline*}
            M(t) = \int_0^t\sigma Y_{s-}\partial_x \hat{C}_H(J_{s-},Y_{s-},T-s) dB_s
                +\int_0^t \int_{\R} \left(\hat{C}_H(Y_{s-}e^u)-\hat{C}_H(J_{s-},Y_{s-},T-s) \right) \tilde{N}_{ J_{s-} } (ds,du)\\
                +\sum_{\gamma\in E\setminus\{J_{s-}\}}\int_0^t \int_{\R} \left(\hat{C}_H(\gamma,Y_{s-}e^u,T-s)-\hat{C}_H(J_{s-},Y_{s-},T-s) \right)  \mathcal{\tilde{M}}_{J_{s-},\gamma}(ds,du).
        \end{multline*}
        We now show that $M(t)$ is a square integrable martingale.
        Suppose $\alpha,\beta\in E$ and $x,y\in\R^+$. Then, for $s\in[0,T]$,
        \begin{align*}  
            |\hat{C}_H(\alpha,x,s) - \hat{C}_H(\beta,y,s)|
            &\leq \left| \E_\alpha[H(xY_s)] - \E_\beta[H(yY_s)] \right|\\
            &\leq \left| \E_\alpha\left[H(xY_s) - H(yY_s)\right] + \E_\alpha\left[  H(yY_s)\right] - \E_\beta\left[ H(yY_s) \right] \right|\\
            &\leq h|x-y|\E_\alpha\left[ Y_s \right] + h|y|\E_\alpha\left[Y_s\right] +h|y|\E_\beta\left[Y_s\right]\\
            &\leq h\left( |x-y|+2|y| \right)\max_{\gamma\in E}\E_\gamma\left[ Y_s \right].
        \end{align*}
        \if{}
                    If $\sigma\neq\gamma$, then 
                    \begin{align*}
                        &\left|\E_\sigma\left[ H(|x|Y_T) \right] - \E_\gamma\left[ H(|y|Y_T)  \right]\right|\\
                        &\qquad\leq \left|\E_\sigma\left[ H(|x|Y_T) - H(|x|) \right] \right| + \left| \E_\gamma\left[ H(|y|Y_T) - H(|y|) \right]\right| + \left| H(x) - H(|y|)\right| \\
                        &\qquad\leq h|x|\E_\sigma\left| Y_T-1  \right| + h|y|\E_\gamma\left|Y_T-1\right| + h|x-y|\\
                        &\qquad\leq h\left( \left(|x|+|y|\right)\left(\max_{j\in E}\E_j\left| Y_T-1 \right|\right) + |x-y|  \right).
                    \end{align*}
                    However, $|x-y|=|x|+|y|$ in the case $\sigma\neq\gamma$, 
                                    \if{}
                                    and hence
                                    \begin{align*}
                                        \left|\E_{\sgn(x)}\left[ G(|x|Y_T) \right] - \E_{\sgn(y)}\left[ G(|y|Y_T)  \right]\right|
                                        &\leq l\left( 2 + \max_{\sigma\in E}\E_\sigma|Y_T| \right)|x-y|
                                    \end{align*}
                                    \fi
                    hence applying this to the expectation formula of $\hat{C}_H$ gives
                    \begin{align*}
                        \left |\hat{C}_H(x) - \hat{C}_H(y) \right| 
                        &\leq he^{-r(T-t)} \left(2 + \max_{j\in E}\E_j\left| Y_{T-t} \right|\right) |x-y|
                    \end{align*}
                    and so $\hat{C}_H$ is Lipschitz with constant $he^{-r(T-t)}\left(2+\max_{\sigma\in E}\E_\sigma|Y_{T-t}|\right)$. 
                    \fi

        Hence, for any Poisson random measure $\mu_t$, such that $\int_\R(e^u-1)^2\mu_t(du)<\infty$ and $t\in[0,T]$, we have
        \begin{align*}
            & \E\left[ \int_0^t\int_\R \left( \hat{C}_H(J_{s}, Y_{s-}e^u,s )-\hat{C}_H(J_{s-},Y_{s-},s) \right)^2 \mu_{s-} (ds,du) \right] \\
            &\qquad\leq \E\left[ \int_0^t\int_\R \left(h(|e^u-1|+2)Y_{s-}\left( \max_{\alpha\in E} \E_{\alpha}\left[ |Y_{T-s}| \right]\right)\right)^2 \mu_{s-}(ds,du) \right]\\
            &\qquad\leq h^2\E\left[ \int_0^t \left(Y_{s-}\right)^2\left( \max_{\alpha\in E}\E_{\alpha}\left[ |Y_{T-s}| \right]\right)^2  \int_\R\left(|e^u-1|+2\right)^2 \mu_{s-}(du) ds \right]\\
            &\qquad<\infty,
        \end{align*}
        provided that $Y$ has finite first and second moments. By Theorem \ref{thm:intLampKiu}, since $Y_T$ has finite second moments, the measures $\mu_\alpha$ and $\nu_{\alpha,\beta}$ satisfy the condition on $\mu_t$ for all $\alpha,\beta\in E$. Hence, by \cite[pp 224, Chapter 4, Theorem 4.2.3]{applebaum2004levy} or \cite[Chapter 8, Proposition 8.8]{tankov2003financial}, the compensated Poisson terms of $M$ are square integrable martingales.
                    \if{}
                    \begin{multline*}
                        \int_0^t \int_{|u|\geq1} \left(\hat{C}_H(Y_{s-}e^u)-\hat{C}_H(Y_{s-}) \right) \tilde{N}_{\sgn(Y_{s-})}(ds,du)\\
                        +\int_0^t \int_{|u|<1} \left( \hat{C}_H( Y_{s-}e^u )-\hat{C}_H(Y_{s-}) \right) \tilde{N}_{\sgn(Y_{s-}}(ds,du)
                    \end{multline*}
                    is also a square integrable martingale.
                
                    Then, analogously to the $\tilde{N}$ case, we consider 
                    \begin{align*}
                        &\E\left[ \int_0^t\int_\R \left( \hat{C}_H(-Y_{s-}e^u) - \hat{C}_H(Y_{s-}) \right)^2 \nu_{\sgn(Y_{s-})}(du) ds \right]\\
                        &\qquad\leq \E\left[ \int_0^t\int_\R l^2 e^{-r2(T-s)}\left( \max_{\sigma\in\pm}\E_\sigma |Y_{T-s}-1| + 1 \right)^2\left| -Y_{s-}e^u - Y_{s-} \right|^2 \nu_{\sgn(Y_{s-})}(du) ds \right]\\
                        &\qquad\leq l^2 \E\left[ \int_0^t|Y_{s-}|^2 \left(\max_{\sigma\in\pm}\E_\sigma |Y_{T-s}-1| + 1 \right)^2 \int_\R (e^u+1)^2 \nu_{\sgn(Y_{s-})}(du)ds\right]
                    \end{align*}
                    hence, since $Y$ is square integrable which also implies $\E[\exp(2U_\pm)]<\infty$, by \cite[Chapter 8, Proposition 8.8]{tankov2003financial} the integral
                    \begin{align*}
                        \int_0^t \int_\R \left( \hat{C}_H( -Y_{s-}e^{u}) - \hat{C}_H(Y_{s-})  \right)\hat{\mathcal{M}}_{\sgn(Y_{s-})}(ds,du)
                    \end{align*}
                    is a square integrable martingale.
                    \fi

        For the Brownian integral in $M$, we consider $\partial_y \hat{C}_H$. For all $(\alpha,y,s)\in E\times\R^+\times[0,T]$, from Lipschitz continuity, we have
        $\left|\partial_y \hat{C}_H(\alpha,y,s)\right|\leq h\E_{\alpha}\left[ Y_{s} \right]$. Then, since $Y$ is square integrable,
        \begin{align*}
            \E\left[ \int_0^t \left( \sigma_{J_{s-}} Y_{s-} \partial_x\hat{C}_H(J_{s-},Y_{s-},T-s) \right)^2 ds  \right]
            \leq \E\left[ \int_0^t \left( \sigma_{J_{s-}} Y_{s-} h \E_{J_{s-}}\left[ Y_{T-s} \right] \right)^2 ds \right]
            <\infty
        \end{align*}
        and hence, by Ito's isommetry,
       $
            \int_0^t\sigma_{J_{s-}} Y_{s-}\partial_x \hat{C}_H(J_{s-},Y_{s-},T-s) dB_s
        $ 
        is a square integrable martingale.
        
        Therefore, we conclude that the process $M(t)$ is itself a square integrable martingale. Then, since $\hat{C}_H(J_T,Y_t,T-t)$ is a local martingale, it follows that 
       $
            \int_0^ta(s) ds = \hat{C}_H(J_t,Y_t,T-t) - \hat{C}_H(J_0,Y_0,T) - M(t)
        $ 
        is a local martingale. However, since this is an integral against $ds$, it is a continuous process with finite variation and must therefore be constant. Hence, $a(t)=0$ for all $t\in[0,T]$.
                    \if{}
                    From the definition of $a$ this gives
                    \begin{align*}
                        0=& -r\hat{C}_H(Y_s,s) + \partial_s \hat{C}_H(Y_{s-},s) \\
                        &+\partial_x \hat{C}_H(Y_{s-},s) Y_{s-}\left( a_{\sgn(Y_{s-})} + \frac{\sigma^2_{\sgn(Y_{s-})}}{2} + \int_\R\left(e^u-1-u\mathbbm{1}_{|u|\leq1}\right)\mu_{\sgn(Y_{s-})}(du) \right)\\
                        &+ \frac{1}{2}\partial_{x^2} \hat{C}_H(Y_{s-},s) Y_{s-}^2\sigma^2_{\sgn(Y_{s-})} \\
                            &+ \int_{|u|\geq1} \left(\hat{C}_H(Y_{s-}e^u)-\hat{C}_H(Y_{s-}) \right) \mu_{\sgn(Y_{s-})}(du)\\
                            &+ \int_{|u|<1} \left( \hat{C}_H(Y_{s-}e^u )-\hat{C}_H(Y_{s-}) - Y_{s-}(e^u-1)\partial_{x}\hat{C}_H(Y_{s-},s) \right)\mu_{\sgn(Y_{s-})}(du)\\
                            &+\int_\R \left( \hat{C}_H( -Y_{s-}e^{u}) - \hat{C}_H(Y_{s-})  \right)\nu_{\sgn(Y_{s-})}(du)
                    \end{align*}
                    \fi 

          Then, using
          \begin{align*}
            \partial_t \hat{C}_H\left(\alpha,y,t\right)
            = \partial_t e^{-r(T-t)}C_H\left(\alpha,y,t\right)
            = re^{-(T-t)}C_H\left(\alpha,y,t\right) + e^{-r(T-t)}\partial_tC_H\left(\alpha,y,t\right),
          \end{align*}
          making the substitution $t=T-s$ and multiplying by $e^{rs}$ in (\ref{eqn:defnCtsA})  gives the required PIDE.
                \if{}
                \begin{align*}
                    0=& -rC_H(Y_s,s) + \partial_s C_H(Y_{s-},s) \\
                    &+\partial_x C_H(Y_{s-},s) Y_{s-}\left( a_{\sgn(Y_{s-})} + \frac{\sigma^2_{\sgn(Y_{s-})}}{2} + \int_\R\left(e^u-1-u\mathbbm{1}_{|u|\leq1}\right)\mu_{\sgn(Y_{s-})}(du) \right)\\
                    &+ \frac{1}{2}\partial_{x^2} C_H(Y_{s-},s) Y_{s-}^2\sigma^2_{\sgn(Y_{s-})} \\
                         &+ \int_\R \left(C_H(Y_{s-}e^u)-C_H(Y_{s-})  - \mathbbm{1}_{\{|u|\leq1\}} Y_{s-}(e^u-1)\partial_{x}C_H(Y_{s-},s)\right) \mu_{\sgn(Y_{s-})}(du)\\
                        &+\int_\R \left( C_H( -Y_{s-}e^{u}) - C_H(Y_{s-})  \right)\nu_{\sgn(Y_{s-})}(du).
                \end{align*}
                \fi
                
    \end{proof}

    \begin{rem}
      The payoff function $H:\R^+\rightarrow\R:x\rightarrow(x-k)^+$ of a Euorpean call option with stirke $k>0$ satisfies the assumptions (3) and (4) of Proposition \ref{prop:PIDE-europeanOptionPrices}. Hence, if the price process corresponds to a MAP satisfying assumptions (1) and (2), then we can use Proposition \ref{prop:PIDE-europeanOptionPrices} to price European call options.
    \end{rem}

    \begin{rem}
      It is possible to recover the Mellin tranform expression of Proposition \ref{prop:mellinGeneralPayoff} from (\ref{eqn:pide}).
      By taking the Mellin transform of (\ref{eqn:pide}) and using the results:
      \begin{align*}
        \frac{\partial}{\partial t}\left\{\mathcal{M}_yC_H(\alpha,y,t)\right\}(s)
        &= \left\{\mathcal{M}_y\frac{\partial}{\partial t}C_H(\alpha,y,t)\right\}(s),\\
        \left\{\mathcal{M}_y y\frac{\partial}{\partial y}C_H(\alpha,y,t) \right\}(s)
        &= -s\left\{\mathcal{M}_yC_H(\alpha,y,t)\right\}(s),\\
        \left\{\mathcal{M}_y y^2C_H(\alpha,y,t)\right\}(s)
        &= s(s+1)\left\{\mathcal{M}_yC_H(\alpha,y,t)\right\}(s), \\
        \left\{\mathcal{M}_yC_H(\alpha,ye^u,t)\right\}(s)
        &=e^{-us}\left\{\mathcal{M}_yC_H(\alpha,y,t)\right\}(s),
      \end{align*}
      we have,
      \begin{align*}
        \partial_t\left\{\mathcal{M}_yC_H(\alpha,y,t)\right\}(s)
        &= \left\{\mathcal{M}_yC_H(\alpha,y,t)\right\}(s)\left( -r -sa_\alpha +s\int_\R u\mathbbm{1}_{\{|u|\leq1\}}\mu_\alpha(du) \right.\\
        &\left.+ \frac{1}{2}\sigma_\alpha^2s^2 + \int_\R\left( e^{-us} -1 \right)\mu_\alpha(du)
          +\sum_{\gamma\in E}\int_\R\left( -1\right)\nu_{\alpha,\gamma}(du)\right)\\
        &+ \sum_{\gamma\in E}\left\{\mathcal{M}_yC_H(\gamma,y,t)\right\}(s)\int_\R e^{-us}\nu_{\alpha,\gamma}(du).
      \end{align*}
      However, $\int_\R\nu_{\alpha,\gamma}(du)=q_{\alpha,\gamma}$ and $\sum_{\gamma\in E}\int_\R\nu_{\alpha,\gamma}(du)=q_\alpha$. Moreover, $\int_\R e^{-us}\nu_{\alpha,\gamma}(du)=q_{\alpha,\gamma}G_{\alpha,\gamma}(-s)$. Thus, if $m(s,t)\coloneqq\left(\left\{\mathcal{M}_yC_H(\alpha,y,t)\right\}(s)\right)_{\alpha\in E}$ is a column vector of dimension $|E|$, then,
      \begin{align*}
        \partial_t m(s,t)= \left(-rI + F(-s)\right) m(s).
      \end{align*}
      Hence, we obtain the solution
      \begin{align*}
        m(s,t) =e^{-rt} e^{tF(-s)}m(s,0),
    \end{align*}
    where, from the initial conditions we have  $m(s,0) =\left\{\mathcal{M}H(y)\right\}(s)(1,\cdots,1)^T$. Thus, we have the relation of Proposition \ref{prop:mellinGeneralPayoff}:
    \begin{align*}
      \left\{\mathcal{M}_y C_H(\alpha,y,t)\right\} = e^{-rt}\{\mathcal{M}H\}(s)\sum_{\beta\in E}\left(e^{tF(-s)}\right)_{\alpha,\beta}.
    \end{align*}
  \end{rem}

\subsection{Examples: \texorpdfstring{$|E|=2$}{|E|=2}}
In the following examples we consider the Lamperti-Kiu case, where $E=\{+,-\}$. An analouge of the construction (\ref{eqn:lampKiuDecomp}) is given in \cite[pp 2502, Theorem 6(i)]{Rivero2011}, \cite[pp 3, Section 1.2]{Kyp15} and is referred to as the Lamperti-Kiu decomposition. To simplify notation, for $\alpha\in\{+,-\}$ we will write $U_\alpha,q_\alpha$ and $G_{\alpha}$ for $U_{\alpha,-\alpha},q_{\alpha,-\alpha}$ and $G_{\alpha,-\alpha}$, respectivley.

    
    In this case, the characteristic equation of $F(z)$ can be solved analytically to obtain the eigenvalues
    \begin{align}
      \label{eqn:egienValuesOfF}
      \alpha(z) := \frac{1}{2}\left( \psi(z) + \Delta(z) \right)
        \qquad\text{and}\qquad
        \beta(z)  := \frac{1}{2}\left( \psi(z) - \Delta(z) \right),
    \end{align}
    where 
    $\psi_\pm^{(q)}(z):=\psi_\pm(z) - q_\pm$, 
    $\psi(z) := \psi_+^{(q)}(z) + \psi_-^{(q)}(z)$ 
    and
    \begin{align*}
        \Delta(z) &:=\sqrt{ \left( \psi_+^{(q)}(z) - \psi_-^{(q)}(z) \right)^2 + 4q_-q_+G_+(z)G_-(z)}.
    \end{align*}
    Let $p(x)=(x-t\alpha)(x-t\beta)$ be the characteristic polynomial of $tF$.
    Then, by considering the remainder on division by $p$, there exists a convergent series, $q(x)$, and a polynomial of degree 1, $r(x)\coloneqq s_0+s_1x$, such that
    $    e^x = q(x)p(x) + r(x)$. 
    Evaluating this at the eigenvalues, $t\alpha$ and $t\beta$, of $tF$, we have
    \begin{align*}
        e^{t\alpha} = 0 + s_0 + s_1t\alpha
        \qquad\text{and}\qquad
        e^{t\beta}  = 0 + s_0 + s_1t\beta.
    \end{align*}
    This system of equation can be solved to obtain
    \begin{align*}
        s_0(t,z) := \frac{\alpha e^{\beta t} - \beta e^{\alpha t}}{\alpha-\beta}
        \qquad\text{and}\qquad 
        s_1(t,z) := \frac{e^{\alpha t}-e^{\beta t}}{(\alpha-\beta)t}.
    \end{align*}
    By the Cayley-Hamilton theorem, it is known that $p(tF)=0$, hence
    \begin{align*}
        e^{tF(z)} = s_0(z)I +s_1(z)tF(z). 
    \end{align*}
    Then, using (\ref{eqn:egienValuesOfF}), the matrix exponential $e^{tF(z)}$ is given by
    \begin{equation}
        \label{eqn:matExpHyp}
        e^{F(z)t} = e^{\frac{t}{2}\psi(z)}\cosh\left( \frac{t}{2}\Delta(z) \right)I + \frac{e^{\frac{t}{2}\psi(z)}\sinh\left( \frac{t}{2}\Delta(z) \right)}{\Delta(z)}\left({\begin{array}{cc}
                    \psi^{(q)}_+(z)-\psi^{(q)}_-(z)     & 2q_+G_+(z)\\
                    2q_-G_-(z)      & \psi_-^{(q)}(z)-\psi_+^{(q)}(z)
                    \end{array}}\right),
    \end{equation}
    where $I$ is the identity matrix. This can be used to obtain explicit expressions for the Mellin transform of option prices.

    \if{}
    \begin{multline}
        \label{eqn:matExpHyp}
        e^{F(z)t} = e^{\frac{1}{2}(\psi_+^q + \psi_-^q)t }\\
                \left({\begin{array}{cc}
            \ch t\sqrt{(\psi_+^q - \psi_-^q)^2 + 4q_+q_-G_+G_-} & \sh t\sqrt{ (\psi_+^q - \psi_-^q)^2 + 4q_+q_-G_+G_- } \\
             \sh t\sqrt{ (\psi_+^q - \psi_-^q)^2 + 4q_+q_-G_+G_- }  & \ch t\sqrt{(\psi_+^q - \psi_-^q)^2 + 4q_+q_-G_+G_-} 
                    \end{array}}\right)
    \end{multline}
    where $\sh$ and $\ch$ denote the hyperbolic sine and cosine respectively.
    \fi 
    
    \begin{exmp}[Markov modulated compound Poisson processes with exponential jumps]
      \label{eg:compoundPoissonExpJump}
      Let $\lambda_+,\lambda_->0$, $U_\pm\sim\Exp(\lambda_\pm)$ and $\mu_\pm(dz)\coloneqq q_\pm\lambda_\pm e^{-\lambda_\pm z}dz$. Also set $a_\pm=\sigma_\pm=0$ and $q_+=q_-\eqqcolon q>0$. Then, the MAP $(J,\xi)$ corresponds to a Markov modulated compound Poisson process, of rate $q$, where the jumps are exponentially distributed with rate $\lambda_\alpha$, determined by the state of $J$. Then, for each $\alpha\in\{+,-\}$ and $z\in\C$ with $\Re(z)<\lambda_\alpha$, we have:
        \begin{align*}
            G_\alpha(z) &= \int_0^\infty e^{zu} \lambda_\alpha e^{-\lambda_\alpha u} du = \frac{\lambda_\alpha}{\lambda_\alpha - z},\\
            \psi_\alpha(z) &= \int_0^\infty \left( e^{uz} - 1 \right)\mu_\alpha(dz) = \int_0^\infty \left( e^{uz} - 1 \right)q\lambda_\alpha e^{-\lambda_\alpha u} du = q\left( G_\alpha(z) - 1 \right).
        \end{align*}
        Substituting this into (\ref{eqn:matExpHyp}) yields
        \begin{align*}
            e^{tF(z)} = \frac{e^{-2qt}}{(G_+(z)+G_-(z))}
            \left[
                e^{qt(G_+(z)+G_-(z))} \left( \begin{array}{cc}
                    G_+(z) & G_+(z) \\
                    G_-(z) & G_-(z)
                \end{array} \right)
                +
                \left( \begin{array}{cc}
                    G_-(z)  & -G_+(z) \\
                    -G_-(z) & G_+(z)
                \end{array} \right)
            \right].
        \end{align*}
            \if{}
                        and so element wise this can be written as
                        \begin{align*}
                            \left( e^{tF(z)} \right)_{i,j}
                            = \frac{e^{-2qt}}{(G_++G_-)}\left[ e^{qt(G_++G_-)}G_i + \chi_{i,j}G_{-j} \right]
                        \end{align*}
                        where $\chi_{i,j}=1$ if $i=j$ and $\chi(i,j)=-1$ if $i\neq j$.
            \fi
        Using Proposition \ref{prop:putcallMellinTransform}, for each $\alpha\in E$, we can compute the price of a European call option, via the inverse Mellin transform, as
        \ifDetails
          \begin{align*}
            C_\alpha(k) 
            &= \sum_{\beta\in E}\frac{1}{2\pi i}\int_{c+i\R} \frac{k^{-s}}{s(s+1)}\left(e^{TF(s+1)}\right)_{\alpha,\beta}ds
            = \sum_{\beta\in E}\frac{1}{2\pi i}\int_{c+1+i\R} \frac{k^{-(s-1)}}{(s-1)s}\left(e^{TF(s)}\right)_{\alpha,\beta}ds
          \end{align*}
        \fi 
        \begin{align*}
            C_\alpha(k) 
            &= \frac{1}{2\pi i}\int_{c+1+i\R} \frac{k^{-s+1}}{(s-1)s} \frac{2e^{-2qt}G_\alpha e^{qt(G_++G_-)}}{(G_++G_-)}ds + \frac{1}{2\pi i}\int_{c+1+i\R}\frac{k^{-s+1}e^{-2qt}}{s(s-1)} \left( \frac{G_{-\alpha} - G_{\alpha}}{G_++G_-} \right) ds,
        \end{align*}
        for $c\in(0,\min(\lambda_+,\lambda_-))$.

        For $\alpha\in E$, we now define the function
        \begin{align*}
          R_\alpha \coloneqq \delta_{1}(k) +  \sqrt{qt\lambda_\alpha}\mathbbm{1}_{\{k\geq1\}}k^{-\lambda_\alpha}\frac{I_1\left(2\sqrt{qt\lambda_\alpha\log(k)}\right)}{\sqrt{\log(k)}},
        \end{align*}
        where $I_1$ is the modified bessel function of the first kind (see Appendix \ref{app:specialFuncs}), and
        \begin{align*}
          D_\alpha \coloneqq  \begin{cases}
            d_1^{(\alpha)} k + d_2^{(\alpha)} , & \text{ if } k<1;
            \\
            d_3^{(\alpha)} k^{c_\alpha}, &\text{ if } k\geq1;
          \end{cases}
        \end{align*}
        where
        \begin{align*}
          d_1^{(\alpha)} &\coloneqq -\frac{1}{2}e^{-2qt},
          \qquad
          d_2^{(\alpha)} \coloneqq  \frac{e^{-2qt}\lambda_\alpha(\lambda_{-\alpha}-1)}{2\lambda_+\lambda_- - (\lambda_++\lambda_-)},
          \qquad
          d_3^{(\alpha)} \coloneqq \frac{e^{-2qt}(\lambda_\alpha-\lambda_{-\alpha} )}{2\left(2\lambda_+\lambda_- - (\lambda_++\lambda_-)\right)}
        \end{align*}
        and
        \begin{align*}
          c_\alpha = 1 -\frac{2\lambda_+\lambda_-}{\lambda_++\lambda_-}.
        \end{align*}

        Then, by applying the Mellin inversion theorem, it is shown in Appendix \ref{app:mellinInv} that
        \begin{align}
          \label{eqn:mellinCovolutionExpression}
          C_\alpha(k) = 2\{ D_\alpha * R_+ * R_- \}(k) + D_{-\alpha}(k) - D_\alpha(k),
        \end{align}
        where $*$ denotes the Mellin type convolution defined by
        \begin{align}
          \label{eqn:mellinConvolution}
          \{f*g\} (x) = \int_0^\infty f\left(y\right)g\left(\frac{x}{y}\right)\frac{1}{y}\,dy,
        \end{align}
        for any two functions $f,g:\R^+\rightarrow\R$.

        In Appendix \ref{app:seriesExpansion}, it is shown that $C_\alpha$ can then be written in series form as
        \begin{align*}
          C_\alpha(k)
          =&\sum_{n=0}^\infty\sum_{m=0}^\infty\sum_{r=0}^\infty 2\left( g(n,m,r,d_1^{(\alpha)},\lambda_-+1,,1k)+ g(n,m,r,d_2^{(\alpha)},\lambda_-,0,k)+g(n,m,r,d_3^{(\alpha)},\lambda_-,\lambda_-,k)\right)\\
           &+ \sum_{m=0}^\infty \frac{qt\lambda_+}{m!(m+1)!}
             \left( F(m,d_1^\alpha,\lambda_+,1,k)
             +F(m,d_2^\alpha,\lambda_+,0,k)
             +f(m,d_3^\alpha,\lambda_+,c_\alpha,k)
             \right)\\
          &+ \sum_{m=0}^\infty \frac{qt\lambda_{-}}{m!(m+1)!}
             \left( F(m,d_1^\alpha,\lambda_{-},2,k)
             +F(m,d_2^\alpha,\lambda_{-},1,k)
             +f(m,d_3^\alpha,\lambda_{-},c_\alpha+1,k)
             \right)\\
           &+\indicator{k<1}\left(d_1^{-\alpha}+d_2^{-\alpha}\right)
             +\indicator{k\geq1}\left(d_3^{-\alpha}k^{c_{-\alpha}} \right),
        \end{align*}
        where,
        \begin{align*}
          F(m,d,\lambda,c,k) &\coloneqq \frac{d k^c\Gamma(m+1,(\lambda+c+1)\log(k\vee1)))}{(\lambda_\alpha+c+1)^{m+1}},\\
          f(m,d,\lambda,c,k) &\coloneqq \frac{d k^c\gamma(m+1,(\lambda+c+1)\log(k\vee1)))}{(\lambda_\alpha+c+1)^{m+1}}
        \end{align*}
        and
        \begin{align*}
          g(n,m,r,d,l,c,k)
          = \frac{(qt)^{m+n+2}\sqrt{\lambda_+\lambda_-}(\lambda_--\lambda_+)^r(r+m)!}{m!(m+1)!(n+1)!r!(r+m+n+1)!} \frac{dk^c}{l^{r+m+n+2}}\Gamma(r+m+n+2,l\log(k)).
        \end{align*}

        Now consider evaluating this at $k=1$, the so called ``at the money'' option. In this case, the triple convolution becomes
        \begin{align*}
          \left( D_\alpha * J_+ * J_- \right)(1)
          &=qt\sqrt{\lambda_+\lambda_-}\sum_{n=0}^\infty\sum_{m=0}^\infty\sum_{r=0}^\infty \frac{qt^{m+n+1}(\lambda_--\lambda_+)^r(r+m)!}{m!(m+1)!(n+1)!r!(r+m+n+1)!}\\
          &\qquad\times
            \left( \frac{d_1^{(\alpha)}}{(\lambda_-+1)^{r+m+n+2}}\Gamma(r+m+n+2)
            + \frac{d_2^{(\alpha)}}{(\lambda_-)^{r+m+n+2}}\Gamma(r+m+n+2)
            \right),
        \end{align*}
        where the upper incomplete Gamma functions have become complete Gamma functions and the lower inncomplete Gamma functions have evaluated to $0$. Then, first considering the sum over $n$, we have
        \begin{align*}
          \left( D_\alpha * J_+ * J_- \right)(1)
          &= qt\sqrt{\lambda_+\lambda_-}\sum_{m=0}^\infty\sum_{r=0}^\infty \frac{(qt)^m(r+m)!(\lambda_--\lambda_+)^r}{m!(m+1)!r!}\\
          &\qquad\times\left(
            \frac{d_1^{(\alpha)}}{(\lambda_-+1)^{r+m+1}}\left( \exp\left( \frac{qt}{\lambda_-+1} \right) - 1  \right)
            + \frac{d_2^{(\alpha)}}{\lambda_-^{r+m+1}}\left( \exp\left(\frac{qt}{\lambda_-}\right)-1\right)
            \right).
        \end{align*}
        \ifDetails
          Then, making use of the identity
          \begin{align*}
            \pFq{1}{1}\left(r+1,2,x\right)
            = \sum_{m=0}^\infty \frac{(r+m)!}{r!}\frac{1}{(m+1)!}\frac{x^m}{m!},
          \end{align*}
          where $\pFq{1}{1}$ is the generalised hypergeometric function , we have
        \fi 
        Now, considering the sum over $m$ and identifying the hypergeometric function $\pFq{1}{1}$ (see Appendix \ref{app:specialFuncs}), gives
        \begin{align*}
         \left(D_\alpha*J_+*J_-\right)(1)
         &= qt\sqrt{\lambda_+\lambda_-}
         d_1^{(\alpha)}\left(\exp\left( \frac{qt}{\lambda_-+1}  \right)-1\right) \sum_{r=0}^\infty\frac{1}{\left(\lambda_-+1\right)^{r+1}}\pFq{1}{1} \left(r+1,2;\frac{qt}{\lambda_-+1}\right)\\
         &\qquad+ qt\sqrt{\lambda_+\lambda_-} d_2^{(\alpha)}\left(\exp\left( \frac{qt}{\lambda_-}  \right)-1\right) \sum_{r=0}^\infty\frac{1}{\lambda_-^{r+1}}\pFq{1}{1} \left(r+1,2;\frac{qt}{\lambda_-}\right).
       \end{align*}
       Moreover, for each $\alpha\in E$ and $c\in\R$
       \begin{align*}
         \sum_{m=0}^\infty \frac{qt\lambda_\beta}{m!(m+1)!}F(m,d^{(\alpha)},\lambda_\beta+c_\alpha+1,c,1)
         &= \sum_{m=0}^\infty \frac{qt\lambda_\beta}{m!(m+1)!}\frac{d^{(\alpha)}k^{c_\alpha}m!}{(\lambda_\beta+c_\alpha+1)^{m+1}}\\
         &=qtd^{(\alpha)}\lambda_\beta \left(\exp\left(\frac{1}{\lambda_\beta+c_\alpha+1}\right)-1\right),
       \end{align*}
       hence,
       \begin{align*}
         \left\{\left\{\indicator{x\geq1}\frac{1}{x}J_\beta(x)\right\} * D_\alpha(x)\right\}(1)
          = qtd^{(\alpha)}\lambda_\beta\left( \exp\left(\frac{1}{\lambda_\beta+2} + \frac{1}{\lambda_\beta+1} + \frac{1}{\lambda_\beta+c_\alpha+1}\right)   - 3\right).
       \end{align*}

       Combining these results,
       \begin{align*}
         C_\alpha(1)
         &=2qt\sqrt{\lambda_+\lambda_-}
         d_1^{(\alpha)}\left(\exp\left( \frac{qt}{\lambda_-+1}  \right)-1\right) \sum_{r=0}^\infty\frac{1}{\left(\lambda_-+1\right)^{r+1}}\pFq{1}{1} \left(r+1,2;\frac{qt}{\lambda_-+1}\right)\\
         &\qquad+ qt\sqrt{\lambda_+\lambda_-} d_2^{(\alpha)}\left(\exp\left( \frac{qt}{\lambda_-}  \right)-1\right) \sum_{r=0}^\infty\frac{1}{\lambda_-^{r+1}}\pFq{1}{1} \left(r+1,2;\frac{qt}{\lambda_-}\right)\\
         &\qquad+ qtd^{(\alpha)}\lambda_+\left( \exp\left(\frac{1}{\lambda_++2} + \frac{1}{\lambda_++1} + \frac{1}{\lambda_++c_\alpha+1}\right)   - 3\right)\\
         &\qquad+ qtd^{(\alpha)}\lambda_-\left( \exp\left(\frac{1}{\lambda_-+3} + \frac{1}{\lambda_-+2} + \frac{1}{\lambda_-+c_\alpha+2}\right)   - 3\right)\\
         &\qquad + k^{c_{-\alpha}}d_3^{(-\alpha)}.
       \end{align*}

        \if{}
                        so
                        \begin{align*}
                            \mathcal{M}\left\{ D_i * R_+ * R_- \right\}(s) 
                            &=\frac{k^{1}e^{-2qt}}{s(s-1)}\frac{G_i}{(G_++G_-)}\left[ \left( e^{qt(G_++G_-)}  - 1 \right) -\mathcal{M}R_+(s) - \mathcal{M}R_-(s) \right]\\
                        \end{align*}
                        which we can rearrange to obtain
                        \begin{align*}
                            \mathcal{M}\left\{ D_i*\left( R_+*R_- + R_+ + R_-  \right) \right\}(s)
                            = \frac{k^{1}e^{-2qt}}{s(s-1)}\frac{G_i}{(G_++G_-)}\left( e^{qt(G_++G_-)}  - 1 \right).
                        \end{align*}
                        Thus,
                        \begin{align*}
                            \left\{\mathcal{M}C\right\}(s) = \mathcal{M}\left\{ D_i * ( R_+ * R_- + R_+ + R_-) \right\}(s) + e^{-2qt}\mathcal{M}_k\left\{  (1-k)^+\right\}(s)
                        \end{align*}
                        and so
                        \begin{align*}
                            C_i(k) = \left\{ D_i * ( R_+ * R_- + R_+ + R_-) \right\}(k) + e^{-2qt}(1-k)^+.
                        \end{align*}        
            \fi

        \if{} 
                \textbf{-----------------------------}
                Now we consider
                \begin{align*}
                    E_i^c(k) 
                    \coloneqq
                    \frac{1}{2\pi i}\int_{c+1+i\R}\frac{k^{-(s-1)}}{(s-1)s}\frac{e^{-2qt}}{(G_++G_-)}e^{qt(G_++G_-)}G_i ds
                \end{align*}
                which also has a pole at $s^*$ with the residue
                \begin{align*}
                    \text{Res}(s^*)   
                    =\frac{k^{-(s^*-1)}e^{-2qt}(\lambda_{-i}-\lambda_i)}{2\lambda_+\lambda_- - (\lambda_+ + \lambda_-)}
                \end{align*}        
                since $e^{qt(G_+(s^*)+G_-(s^*)}=1$. However, there are also singularities at $\lambda_+$ and $\lambda_-$ coming from the exponential term. For $\sigma\in\{+,-\}$, by expanding $e^{qtG_\sigma(z)}$ we have that the integrand is given by
                \begin{align*}
                    \sum_{k=0}^\infty\frac{1}{k!}\left(\frac{qt\lambda_\sigma}{\lambda_\sigma-s}\right)^k\left( \sum_{n=0}^\infty \left( \frac{qt\lambda_{-\sigma}}{\lambda_{-\sigma}-s} \right)^n \right)\frac{k^{-(s-1)}}{(s-1)s}e^{-2qt}\frac{\lambda_i(\lambda_{-i}-s)}{\lambda_++\lambda_-}
                \end{align*}
                hence the residue at $\lambda_\sigma$ is given by
                \begin{align*}
                    \sum_{k=0}^\infty\frac{1}{k!}\left(qt\lambda_\sigma\right)^k  \frac{d^{k-1}}{ds^{k-1}}\left\{\exp\left(\frac{qt\lambda_{-\sigma}}{\lambda_{-\sigma}-s} \right)\frac{k^{-(s-1)}}{(s-1)s}e^{-2qt}\frac{\lambda_i(\lambda_{-i}-s)}{\lambda_++\lambda_-}\right\}\Bigg|_{s=\lambda_{\sigma}}
                \end{align*}      
                and so 
                \begin{align*}
                    E^c_i(k) = \text{Res}(s^*) + \text{Res}(\lambda_+) + \text{Res}(\lambda_-). 
                \end{align*}
                Putting this together we have
                \begin{align*}
                    C_i(k) = 2E^c_i(k) + D^c_{-i}(k) - D^c_{i}(k).
                \end{align*}
    \fi
    \end{exmp}

    \begin{exmp}[Skew-symmetric Markov modulated compound Poisson process]
      Simpler examples can be found in the case that $(1,1)^T$ is a right eigenvector of $F(x)$ for all $x>0$. This occurs if and only if
      \begin{align}
        \label{eqn:mellinVectorIff}
        \psi_+(x) + q_+\left(G_+(x) -1\right) = \psi_-(x) + q_-\left(G_-(x) - 1\right).
      \end{align}
      If this holds, then $(1,1)^T$ is also an eigenvector of $e^{TF(x)}$, for all $x,T>0$, corresponding to the eigenvalue $\exp\left(T(\psi_+(x) + q_+\left(G_+(x)-1\right))\right)$. Thus, from Proposition \ref{prop:putcallMellinTransform}, we have
      \begin{align*}
        \left\{\mathcal{M}_k C_\alpha(k)\right\}(u) = e^{-rT}\frac{\exp\left(T\left(\psi_+(u+1) + q_+\left(G_+(u+1)-1\right)\right)\right)}{u(u+1)},
      \end{align*}
      where $C_\alpha(k)$ denotes the price of a European call option with strike $k>0$ and maturity $T>0$, when $(J_0,Y_0)=(\alpha,1)\in E\times\R^+$.
      
      We now consider a particular example of when (\ref{eqn:mellinVectorIff}) holds. Suppose that $q_+=q_-\eqqcolon q$ and that each of the \levy processes corresponding to the states of $E$ is a compound Poisson processes (so $a_\pm=\sigma_\pm=0$). Moreover, suppose that the distribution of the jumps $U_\pm$, corresponding to chages in state of $J$ from $\pm$ to $\mp$, are given by $\nu_\pm$ and let the \levy measure of $\xi^\pm$ be given by $\mu_\pm\coloneqq q\nu_\mp$. Then, $\psi_\pm(u) = q(G_{\mp}(u)-2)$. 
      %
      If we make the further assumption that the densities of $U_\pm$ satisfy
      $\mu_+(x) + \mu_-(x) = 2e^x$
      for $x\in(-\infty,0)$, then we can immediatley compute that $\psi_\alpha(u) + q_\alpha(G_\alpha(u)-1) = -2qu(u+1)^{-1}$ for all $\alpha\in E$.
      Hence,
      \begin{align*}
        \left\{\mathcal{M}_k C_\alpha(k)\right\}(u)
        = e^{-rT}\exp\left(-\frac{2qT(u+1)}{u+2}\right)\frac{1}{u(u+1)}
        =e^{-(r+2q)T}\exp\left(\frac{2qT}{u+2}\right)\frac{1}{u(u+1)}.
      \end{align*}
      Define the function $R:\R^+\rightarrow\R$ by
      \begin{align*}
        R(k) \coloneqq  \delta_1(k) + \sqrt{2qT}\indicator{k\leq1} k^{2} \frac{I_1\left(2\sqrt{-2qT\log(k)}\right)}{\sqrt{-\log(k)}},
        \qquad k\geq0.
      \end{align*}
      Then, it is known that
      \begin{align*}
        \left\{ \mathcal{M}R \right\}(u)
        =\exp\left(\frac{2qT}{2+u}\right)
        \qquad\text{and}\qquad
        \mathcal{M}\left\{ (1-k)^+ \right\}(s) = \frac{1}{u(u+1)}.
      \end{align*}
      Hence, by the Mellin inversion theorem,
      \begin{align*}
        C_\alpha(k) = e^{-(r+2q)T}\left\{R(x)*(1-x)^+\right\}(k).
      \end{align*}
      Expanding the Mellin convolution, we have
      \begin{align*}
        &\left\{ R(x^{-1}) * (1-x)^+ \right\}(k)\\
        &\qquad= \int_0^\infty \delta_1(x)\left(1-\frac{k}{x}\right)^+\frac{1}{x}
          +  \sqrt{2qT}\indicator{x\leq1} x^{2} \frac{I_1\left(2\sqrt{-2qT\log(x)}\right)}{\sqrt{-\log(x)}}\left(1-\frac{k}{x}\right)^{+}\frac{1}{x} dx\\
        &\qquad= (1-k)^+ + \sqrt{2qT}\int_k^1 (x-k)\frac{I_1(2\sqrt{-2qT\log(x)})}{\sqrt{-\log(x)}}dx,
      \end{align*}
      whenever $k\leq1$, whilst $\left\{ R(x^{-1}) * (1-x)^+ \right\}(k)=0$ for $k>0$.
      However, using the series expansion of $I_1$ and a change of variables,
      \begin{align*}
        \int_k^1  \frac{I_1(2\sqrt{-2qT\log(x)})}{\sqrt{-\log(x)}} dx
        =\sum_{m=0}^\infty \frac{(2qT)^{m+1/2}}{m!(m+1)!}\int_{k}^1 \log\left(\frac{1}{x}\right)^m dx
        &= \sum_{m=0}^\infty \frac{(2qT)^{m+1/2}}{m!(m+1)!} \gamma(m+1,\log(1/k))
      \end{align*}
      and
      \begin{align*}
        \int_k^1  \frac{xI_1(2\sqrt{-2qT\log(x)})}{\sqrt{-\log(x)}} dx
        =\sum_{m=0}^\infty \frac{(2qT)^{m+1/2}}{m!(m+1)!}\int_{k}^1 x\log\left(\frac{1}{x}\right)^m dx
        &= \sum_{m=0}^\infty \frac{2^{-1/2}(qT)^{m+1/2}}{m!(m+1)!} \gamma(m+1,2\log(1/k)).
      \end{align*}
      Hence, for all $k\leq1$,
      \begin{multline}
        C_\alpha(k) = e^{-(r+2q)T}\left(1 - k + \sum_{m=0}^\infty
          \frac{(qT)^{m+1/2}2^{-1/2}\gamma(m+1,2\log(1/k))}{m!(m+1)!}
            - \frac{k(2qT)^{m+1/2}\gamma(m+1,\log(1/k))}{m!(m+1)!}\right)
      \end{multline}
      and $C_\alpha(k)=0$ for all $k>1$. Notice that  $C_\alpha(1)=0$, hence the option price is continuous at this transition point. Moreover, it is not surprising that $C_\alpha(k)=0$ for $k>1$ since the $\xi_t$ is a (weakly) decreasing processes, thus once $\xi<k$ the option can never regain its value. The maximal value of the call option is achieved when $k=0$. In this case,
      \begin{align*}
        C_\alpha(0)
        =\left(1-\frac{1}{\sqrt{2qT}}\right)e^{-(r+2q)T} + \frac{1}{\sqrt{2qt}}e^{-(r+q)T}.
      \end{align*}
      
    \end{exmp}

    Whilst  expressions for $C_\alpha$ were obtained in some of these examples, albeit with high levels of complexity, the main benefit of the Mellin transform approach is that it allows numerical computation of option prices via the Fast Fourier Transform. We can also use the Mellin transform expression to conduct sensitivity analysis of option prices.

\section{Comparison of European and  Asian Call Option Prices}
\label{sec:martingaleConditions}

An Asian option, with payoff function $H:\R^+\rightarrow\R$ and maturity $T\geq0$, on an asset with price process $\{Y_t:t\geq0\}$, is a contract which pays its owner $\int_{T_0}^TY_sds$ at time $T$, for some $T_0\in(0,T)$. Similarily to European options, an Asian option with payoff function $H(x)\coloneqq(x-k)^+$, for some $k>0$, is called an Asian call option, whilst if the payoff function is $H(x)\coloneqq(k-x)^+$, then it is called an Asian put option. In both cases, $k$ is referred to as the strike price.

Under the equivalent martingale measure, $\Prob$, considered in Section \ref{sec:europeanOptions}, the price of an Asian option at time $t<T$ is given by
\begin{align*}
  e^{-r(T-t)}\E\left[H\left( \int_{T_0}^T Y_s\, ds \right) \:\middle|\: \mathcal{F}_t \right].
\end{align*}
As in Section \ref{sec:europeanOptions}, we assume that the price process of the underlying asset is given by an exponential MAP model. That is, $Y_t\coloneqq\exp(\xi_t)$, for all $t\geq0$, where $(J,\xi)$ is a MAP. Then, following the simplifying steps of \cite{gemanYor}, for $t\in(T_0,T)$,
\begin{align*}
  \E\left[H\left(\int_{T_0}^TY_s\, ds\right)\:\middle|\:\mathcal{F}_t\right]
  &=\E\left[H\left(\int_{T_0}^te^{\xi_s}\,ds + e^{\xi_t}\int_{0}^{T-t}e^{\xi_{t+s}-\xi_t}\,ds \right)\:\middle|\:\mathcal{F}_t\right].
\end{align*}
Since $\gamma_t\coloneqq\int_{T_0}^t e^{\xi_s}ds$ and $e^{\xi_t}$ are both $\mathcal{F}_t$ measurable, there is a function $H_t(x)\coloneqq H(\gamma_t + e^{\xi_t}x)$, such that
\begin{align*}
  \E\left[H\left(\int_{T_0}^TY_s\, ds\right)\:\middle|\:\mathcal{F}_t\right]
  = \E\left[H_t\left(\int_{0}^{T-t}e^{\xi_{t+s}-\xi_t}\,ds\right)\:\middle|\:\mathcal{F}_t\right]
  = \E_{J_t}\left[H_t\left(\int_0^{T-t}e^{\hat{\xi}_s}\,ds\right)\right],
\end{align*}
where $\hat{\xi}$ is an independent copy of $\xi$ and the second equality follows from the Markov additive property. Hence, an understanding of the price of an Asian option can be obtained by studying the simpler object
\begin{align*}
  C_H^A(\alpha,y,T) \coloneqq e^{-r\tau}\E_{\alpha,\log(y)}\left[H\left(\int_0^T e^{\xi_s}\, ds \right)\right].
\end{align*}
We wish to make a comparison of the prices of European and Asian options under an exponential MAP model. To do this we will need the following Martingale properties of exponential MAPs, which we derive from Dynkin's formula.

    Then, let $A$ denote the (extended) generator of the Markov process $(J,Y)$, where $Y_t:=\exp(\xi_t)$ for all $t\geq0$ and $(J,\xi)$ is a MAP. Denote the domain of the extended generator by $\mathbb{D}(A)$. From \cite{Rivero2011}, it is known that, for a bounded continuous function $f\in \mathbb{D}(A)$, we have
    \begin{equation}
        \label{eqn:lampKiuGen}
        (Af)(\alpha,x) = \left( \mathcal{L}^{(\alpha)}\exp\circ f(\alpha,\cdot)\right)(\log|x|) + \sum_{\beta\in E\setminus\{\alpha\}}q_{\alpha,\beta}\left(\E\left[ f(\beta,x\exp(U_{\alpha\beta}))\right] - f(\alpha,x) \right),
    \end{equation}
    for all $(\alpha,x)\in E\times\R^+$, where $\mathcal{L}^{(\alpha)}$ is the generator of $\xi^{(\alpha)}$. We can now state a martingale condition for $Y$.

    
    \begin{thm}[Martingale condition for Lamperti-Kiu]
        \label{thm:mgCondLampKiu}
        Let $f:E\times\R^+\rightarrow\R:(\alpha,x)\rightarrow x$ and $\{\mathcal{F}_t\}_{t\geq0}$ be the natural filtration of $(J,\xi)$. Then, $Y$ is a martingale with respect of $\mathcal{F}$, if and only if, $f\in\mathbb{D}(A)$ and $(Af)(\alpha,1)=0$, for all $\alpha\in E$.
    \end{thm}
    \begin{proof}
        \textit{Sufficiency:}\newline
        First suppose $Y$ is a martingale and thus is integrable. Then, Theorem \ref{thm:intLampKiu}(\ref{list:integrability:decomp}) holds and under these conditions a semi-martingale decomposition of $Y$ is given by (\ref{eqn:YsemiMartingale}). This can then be rearranged to give, for all $t\geq0$,
        \begin{align*}
            Y_t = M_t + \int_0^t Y_{s-}\left( a_{J_s} + \frac{\sigma^2_{J_s}}{2} + \int_\R(e^u - 1)\nu_{J_{s-}}(du) + \int_\R\left( e^{u}-1-u\mathbbm{1}_{\{|u|\leq1\}} \right)\mu_{J_{s-}}(du) \right) ds,
        \end{align*}
        where $\{M_t;t\geq0\}$ is a martingale. However, by applying (\ref{eqn:lampKiuGen}) to $f:E\times\R:(\alpha,x)\rightarrow x$, we have
        \begin{align*}
            Y_t = M_t + \int_0^t (Af)(J_{s-},Y_{s-}) ds,
        \end{align*}
                    \if{}
                            Notice that the integrand on the last line of (\ref{eqn:SDErepresentationMAP}) includes the generator of the Lamperti-Kiu process applied to $f$. Since we have sufficient moments, by \cite[Section 4.3.2]{applebaum2004levy} the other integrals are  martingales. More precisely, there is a martingale $M$ with respect to $\mathcal{F}_t$, the natural filtration of $Y$, such that
                            \begin{equation*}
                                Y(t) - Y(0) = M(t) +  \int_0^t(Af)(Y_s,J_s)ds
                            \end{equation*}
                    \fi
        hence, $f\in\mathbb{D}(A)$. Then, since $Y$ is also a martingale,
        \begin{equation*}
            \int_0^t (Af)(J_s,Y_s)ds =  Y_s - M_s,
        \end{equation*}
        is  a martingale. Thus, for all $t>u>0$, first using the martingale property and then Fubini's theorem, we have
        \begin{align*}
            0 = \E\left[ \int_u^t(Af)(J_s,Y_s)ds \:\middle|\: \mathcal{F}_u \right]
            =\E_{J_u,\xi_u}\left[ \int_0^{t-u}(Af)(\hat{J}_s,\hat{Y}_s) ds \right]
            =\int_0^{t-u}\E_{J_u,\xi_u}\left[ (Af)(\hat{J}_s,\hat{Y}_s) \right]ds,
        \end{align*}
        where $(\hat{J},\hat{Y})$ is an independent but identically distributed copy of $(J,Y)$.
    However, from the definition of $f$ and since $(J,\log(Y))$ is a MAP, for all $(\sigma,a)\in\R^+\times E$, we have $(Af)(\sigma,a)= a(Af)(\sigma,1)$.
    Then, substituting this into the previous integral gives
    \begin{align*}
        0 &= \int_0^{t-u}\E_{J_u,\xi_u}[\hat{Y}_s(Af)(\hat{J}_s,1)] ds  ,
    \end{align*}
    for all $t>u>0$. By differentiating with respect to $t$ 
            \if{}
            we obtain
            \begin{equation*}
                0 = \E_{Y_u}[|\hat{Y}_{t-u}|(Af)(\sgn(\hat{Y}_{t-u}))] 
            \end{equation*}
            for all $t>u>0$. Thus for all $s>0$ 
            \fi 
    and setting $t=u+s$, we have, for $s,u>0$,
    \begin{equation*}
        0 = \E_{J_u,\xi_u}[\hat{Y}_{s}(Af)(\hat{J}_s,1)] \qquad\text{a.e.}.
    \end{equation*}
    
    To take the limit as $s\downarrow0$, we adapt part of the proof of \cite[Theorem 6(i)]{Rivero2011}. Splitting around the event $\{\hat{T}_1 > s\}$,
    \begin{align*}
        \E_{J_u,\xi_u}[\hat{Y}_{s}(Af)(\hat{J}_{s},1)] 
            = \E_{J_u,\xi_u}[\hat{Y}_{s}(Af)(\hat{J}_{s},1)\:|\: \hat{T}_1>s]\Prob_{J_u}(\hat{T}_1>s) 
            + \E_{J_u,\xi_u}[\hat{Y}_{s}(Af)(\hat{J}_{s},1)\:|\: \hat{T}_1\leq s]\Prob_{J_u}(\hat{T}_1\leq s) .
    \end{align*}
    Considering the first term,
    \begin{align*}
        \E_{J_u,\xi_u}[\hat{Y}_{s}(Af)(\hat{J}_s,1)\:|\: T_1>s]\Prob_{J_u}(T_1>s)
        &= \E_{J_u,\xi_u}[\hat{Y}_0\exp(\hat{\xi}^{(1)}_s)(Af)(J_u,1) ]e^{-sq^{(1)}}\\
        &= Y_u(Af)(J_u,1)\E[\exp(\hat{\xi}_1^{(1)})]^se^{-sq^{(1)}},
    \end{align*}
    where $q^{(1)}$ is the rate of the exponential distribution of the first sign change.
    Then, letting $s\downarrow0$, we obtain
    \begin{align*}
        \lim_{s\downarrow0}\E_{J_u,\xi_u}[\hat{Y}_s(Af)(\hat{J}_s,1)\:|\:\hat{T}_1>s]\Prob_{J_u}(\hat{T}_1>s) = Y_u(Af)(J_u,1).
    \end{align*}
    To see that the second term tends to 0 as $s\downarrow0$ recall that $Y$ is integrable and that $\Prob_{J_0}(T_1\leq s)=1-e^{-sq^{(1)}}\rightarrow 0$ as $s\downarrow0$. Then, we have the bound
    \begin{equation*}
        |Y_{s}(Af)(J_s,1)\mathbbm{1}_{\{T_1\leq s\}}|\leq|Y_{s}|\max_{\alpha\in E}|(Af)(\alpha,1)|    ,
    \end{equation*}
    which is integrable, since $E$ is finite. Moreover, because $Y_{s}(Af)(J_s,1)\mathbbm{1}_{\{T_1\leq s\}}\rightarrow0$ almost surely as $s\downarrow0$, the dominated convergence theorem yields
    \begin{align*}
        \lim_{s\downarrow0}\E_{J_u,\xi_u}[\hat{Y}_s(Af)(J_s,1); T_1\leq s] = 0 \qquad\text{a.s.}.
    \end{align*}
    Combining these results, we obtain
    \begin{align*}
        0 = \lim_{s\downarrow0}\E_{J_u,\xi_u}[\hat{Y}_s(Af)(\hat{J}_s,1)] = Y_u(Af)(J_u,1) \qquad\text{a.s.}
    \end{align*}
    and dividing by $Y_u$, which is non-zero, then gives
    $0 = (Af)(J_u,1)$ a.e..
    \if
    We now want to expand this and study it more closely. Note that since $Y_s\neq0$ for all $s\geq0$ we need only consider the cases where $Y_s$ is strictly positive or strictly negative. So a.e. we have
    \begin{align*}
        0=\E_{Y_u}\left[ |\hat{Y}_s|(Af)(\sgn(\hat{Y}_s)) \right]
        &= \E_{Y_u}\left[ |\hat{Y}_s|(Af)(\sgn(\hat{Y}_s)) \:|\: \hat{Y}_s>0 \right]\Prob_{Y_u}(\hat{Y}_s>0) \\&\qquad+\E_{Y_u}\left[ |\hat{Y}_s|(Af)(\sgn(\hat{Y}_s)) \:|\: \hat{Y}_s<0 \right]\Prob_{Y_u}(\hat{Y}_s<0)\\
        &= \E_{Y_u}\left[ \hat{Y}_s(Af)(1) \:|\: \hat{Y}_s>0 \right]\Prob_{Y_u}(\hat{Y}_s>0) \\&\qquad+\E_{Y_u}\left[ -\hat{Y}_s(Af)(-1) \:|\: \hat{Y}_s<0 \right]\Prob_{Y_u}(\hat{Y}_s<0)\\
        &= (Af)(1)\E_{Y_u}\left[ \hat{Y}_s \:|\: \hat{Y}_s\geq0 \right]\Prob_{Y_u}(\hat{Y}_s\geq0) \\&\qquad+(Af)(-1)\E_{Y_u}\left[ -\hat{Y}_s \:|\: \hat{Y}_s<0 \right]\Prob_{Y_u}(\hat{Y}_s<0).
    \end{align*}
    Taking the limit as $s\searrow 0$ the \cadlag property gives that $\hat{Y_s}\rightarrow Y_u$ under $\Prob_{Y_u}$. Then bounding $|Y_s|$ by $sup_{0\leq t\leq1}|Y_t|$ \textbf{(which is integrable with sufficient integrability of Y?)} for all $s\leq1$ we can apply dominated convergence to obtain
    \begin{align*}
        0 = (Af)(1)Y_u\mathbbm{1}_{Y_u>0} + (Af)(-1)|Y_u|\mathbbm{1}_{Y_u<0} \qquad\text{a.e.}
    \end{align*}
    \fi
    Since $J_u=\alpha$ with non-zero probability for any $u>0$, $\alpha\in E$, we have  $(Af)(\alpha,1)=0$, for all $\alpha\in E$.

    \textit{Necessity:}\newline
    Suppose that $f\in\mathbb{D}(A)$ and $(Af)(\alpha,1)=0$ for all $\alpha\in E$. Then, from equation (\ref{eqn:lampKiuGen}), for all $\alpha\in E$,
    \begin{align*}
      \left|a_{\alpha} + \frac{1}{2}\sigma^2_\alpha + \int_\R \left( e^u -1-u\mathbbm{1}_{|u|<1} \right) \mu_\alpha(du) \right|
      = |(\mathcal{L}^{(\alpha)}\exp)(\log(1))|
      < \infty 
    \end{align*}
    and
    \begin{align*}
      \left| q_{\alpha,\beta}\E\left[ 1-\exp(U_{\alpha,\beta}) \right]\right|
      <\infty.
    \end{align*}
    This implies that both $\xi^{(\alpha)}$ and $U_{\alpha,\beta}$ have exponential moments, thus by Theorem \ref{thm:intLampKiu}, $\E[Y_t]<\infty$ for all $t\geq0$.
    
    The assumption $(Af)(\alpha,1)=0$, for all $\alpha\in E$, combined with the multiplicative invariance property gives, for all $t>u>0$,
    \begin{align*}
        \int_u^t(Af)(J_s,Y_s)ds = \int_u^t Y_s(Af)(J_s,1) ds 
        = 0.
    \end{align*}
    
    Then, by the definition of the extended generator, we have that
    \begin{equation*}
        M(t) := Y_t - Y_0 - \int_0^t(Af)(J_s,1)ds = Y_t - Y_0,
    \end{equation*}
    is a martingale with respect to $\{\mathcal{F}_t\}_{t\geq0}$. Thus, $Y$ must also be a martingale.
    \end{proof}
    
    Using equation (\ref{eqn:lampKiuGen}) and the \levys-Khintchine formula to expand the generator of a MAP, the following condition for such a process to be a martingale can be found.
    
    \begin{cor}
        \label{cor:LKmg}
        The process $Y$ is a martingale if and only if, for all $\alpha\in E$,
        \begin{align}
          \label{eqn:martingalExpCondition1}
            \sum_{\beta\in E}q_{\alpha,\beta}\E\left[ \exp(U_{\alpha,\beta}) - 1 \right] = a_\alpha + \frac{1}{2}\sigma_\alpha^2 + \int_\R \left( \exp(y) -1 -\frac{y}{1+|y|} \mu_{\alpha}(dy) \right)
            <\infty,
        \end{align}
        or equivalently, for all $\alpha\in E$,
        \begin{equation}
            \label{eqn:martingaleExpCondition}
            \sum_{\beta\in E}q_{\alpha,\beta}\left( G_{\alpha,\beta}(1) - 1 \right) = \psi_\alpha(1) < \infty.
        \end{equation}
    \end{cor}
    \begin{proof}
    By Theorem \ref{thm:mgCondLampKiu}, the process $Y$  is a martingale if and only if $f\in\mathbb{D}(A)$ and $(Af)(\alpha,1)=0$, for all $\alpha\in E$ when $f:\R^+\times E\rightarrow\R:(\sigma,x)\rightarrow x$ and $A$ is the (extended) generator of the pair $(J,Y)$.
    
    Suppose equation (\ref{eqn:martingaleExpCondition}) holds, then for each $\alpha\in E$, both $\E\left[\exp(\xi_1^{(\alpha)})\right]<\infty$ and $\E\left[\exp\left(U_\alpha\right)\right]<\infty$, hence by Theorem \ref{thm:intLampKiu}, $Y$ is integrable. Moreover, if $Y$ is integrable then Theorem \ref{thm:intLampKiu} gives the finiteness requirement of equation (\ref{eqn:martingaleExpCondition}). 
    
    \if{}
        First we show that the integrability of $Y$ is equivalent to the finiteness condition in the Corollary. Since $\E[\exp(\xi_1^i)]<\infty$ for $i\in\{+,-\}$ the \levy processes have finite exponential moments. Also  if $q^i\neq0$ then $\E[\exp(U^i)+1]<\infty$ is equivalent to exponential moments for the overshoots $U^\pm$. The cases $q^i=0$ correspond to the degenerate cases where $Y$ is an exponential \levy process and thus $U^i=0$. Hence it follows from Theorem \ref{thm:intLampKiu} that the integrability of $Y$ is equivalent to the finiteness conditions of the Corollary.
    \fi
    
    We now look at the expectation requirement for a martingale. From equation (\ref{eqn:lampKiuGen}), the (extended) generator of $(J,Y)$ applied to $f:\R^+\times E\rightarrow\R:(\alpha,x)\rightarrow x$ gives
    \begin{align*}
        (Af)(\alpha,x) = \mathcal{L}^{(\alpha)}(\exp)(\log x) + \sum_{\beta\in E}xq_{\alpha,\beta}\E\left[ \exp(U_{\alpha,\beta}) - 1 \right],
    \end{align*}
    where $\mathcal{L}^{(\alpha)}$ is the extended generator of the  \levy process $\xi^{(\alpha)}$. 
            \if{}
            Expanding out $\mathcal{L}^+(f\circ\sgn(x)\exp)$ from the \levys-Kintchin formula and defining $g$ as
            \begin{equation*}
                g := f\circ\sgn(x)\exp = \sgn(x)\exp
            \end{equation*}
            gives
            \begin{align*}
                &\mathcal{L}^+(g)(x) 
                =\sgn(x)\exp(x)\left( a^+ + \frac{1}{2}(\sigma^+)^2 + \int_\R \exp(y) - 1 - \frac{y}{1+|y|} \mu^+(dy) \right)
            \end{align*}
            then taking $x=\log(1)$ we obtain
            \fi
            \if{}
                Then, using the \levy-Kintchine formula we have
                \begin{equation*}
                    (\mathcal{L}^\alpha f\circ\exp)(\log(1)) =  a^\alpha + \frac{1}{2}(\sigma^\alpha)^2 + \int_\R \exp(y) - 1 - \frac{y}{1+|y|} \mu^\alpha(dy).
                \end{equation*}
            \fi
    Thus, the condition $(Af)(\alpha,1)=0$ is equivalent to
    \begin{equation*}
        \sum_{\beta\in E}q_{\alpha,\beta}\left(\E[\exp(U_{\alpha,\beta})]-1 \right) = a_\alpha + \frac{1}{2}\sigma_\alpha^2 + \int_\R \left(\exp(y) - 1 - y\mathbbm{1}_{\{|y|\leq 1\}}\right) \mu_\alpha(dy) = \psi_\alpha(1).
    \end{equation*}
    \end{proof}

        As an example, suppose that $X$ is a spectrally negative \levy process, so that the Laplace exponent, $\psi(z)$, is defined for all $z\in\R^+$. Let the characteristic triplet be $(a_X,\sigma_X,\mu_X)$. Then, by \cite[pp 82]{kyprianou2014fluctuations} the process $\{\xi_t\coloneqq X_t - \psi(1)t:\: t\geq0\}$ is also a \levy process  and has characteristic triplet $(a_\xi,\sigma_\xi,\mu_\xi)$ given by
        \begin{align*}
            a_\xi := a_X - \psi(1)t, 
            \qquad\sigma_\xi := \sigma_X
            \qquad\text{and}\qquad
            \mu_\xi := \mu_X.
        \end{align*}
        Moreover, from \cite[pp 82]{kyprianou2014fluctuations} we know that the process $\{Y_t:=\exp(\xi_t):t\geq0\}$ is a martingale and $(J,Y)$ is a MAP for any constant Markov chain $J$. Hence we can check the conditions of Corollary \ref{cor:LKmg} are satisfied. In particular, because the Markov chain $J$ is constant, the left hand side of (\ref{eqn:martingalExpCondition1}) is $0$. The right hand side is given by
        \begin{multline*}
            a_\xi + \frac{\sigma_\xi^2}{2} + \int_{-\infty}^0\left(\exp(y)-1-\frac{y}{1+|y|}\right)\mu_\xi(dy) \\
            =a_X - \left( a_X + \frac{\sigma^2_X}{2} + \int_{-\infty}^0\left(\exp(y) -1-\frac{y}{1+|y|}\right)\mu_X(dy) \right) \\
            + \frac{\sigma_X^2}{2} + \int_{-\infty}^0\left(\exp(y)-1-\frac{y}{1+|y|}\right)\mu_X(dy)
            =0
        \end{multline*}
        and therefore Corollary \ref{cor:LKmg} is satisfied.
    
    \if{}
        It is worth noting that $Y$ can't be a uniformly integrable martingale and so Doob's martingale convergence theorem \cite[Theorem 10, Chapter I]{protter2005stochastic} doesn't apply. Moreover, $Y$ also can't be a closed martingale \cite[Theorem 13, Chapter I]{protter2005stochastic}.    
        
        Suppose for contradiction that $Y_t$ was a uniformly integrable Lamperti-Kiu Martingale. Then we could apply the optional stopping theorem \cite[Theorem 2.13]{ethier1986markov} at the first time the process switches sign, $T_1$, and obtain
        $
            Y_0 = \E[Y_{T_1}]
        $
        however this is impossible since $Y_{T_0}$ has the opposite sign of $Y_0$. We can therefore conclude that being a Martingale or Uniformly Integrable are mutually exclusive conditions for a Lamperti-Kiu process.
        
        This can also be seen directly through calculation. Consider the case $Y_0>0$ and the case $Y_0<0$ is analogous. The Optional Stopping Theorem, as given in \cite[Theorem 2.13]{ethier1986markov}, has assumptions weaker than Uniform Integrability one of which is
        $
            \E\left[ |Y_{T_0}| \right] < \infty
        $
        however, through the tower property,
        \begin{align*}
            \E[|Y_{T_1}|] 
            = \E\left[\E[\exp(\xi^+_1)]^{T_1}\right]\E\left[ \exp(U_1) \right]
            =q_+\int_0^\infty \exp((\psi_+(1)-q_+)t)dt\E\left[ \exp(U_1) \right]
        \end{align*}    
        which is finite if and only if $\psi_+(1):=\log(\E[\exp(\xi^+_1)])<q_+$. However, this contradicts the conditions of the Corollary which are
        \begin{align*}
            \log(\E[\exp(\xi^+_1)]) = q^i\E[\exp(U^i)+1] > q^i
        \end{align*}
        thus a Lamperti-Kiu process is a martingale only if it isn't uniformly integrable.
        
    \fi 

    The martingale result can be extended to obtain sub/super martingale conditions for $Y$. For convenience, if $A$ is the (extended) generator of a process and $f:E\times\R\rightarrow\R:(\sigma,x)\rightarrow x$, we introduce the notation $A^{(\alpha)} := (Af)(\alpha,1)$, for all $\alpha\in E$.
    
    \begin{prop}[Sub/Super Martingale conditions for Lamperti-Kiu processes]
        \label{prop:subSuperMgLk}
        An exponential MAP, with (extended) generator $A$, is a sub-martingale if and only if $0\leq A^{(\alpha)} <\infty$, for all $\alpha\in E$, and is a super-martingale if and only if $-\infty<A^{(\alpha)}\leq0$, for all $\alpha\in E$.
    \end{prop}
    \begin{proof}
        As for the martingale property in Theorem \ref{thm:mgCondLampKiu}, integrability in both cases is equivalent to $|A^{(\alpha)}|<\infty$ for all $\alpha\in E$. This is obtained by considering the decomposition of the (extended) generator given in equation (\ref{eqn:lampKiuGen}) and comparing it to Theorem \ref{thm:intLampKiu}.
        
        We now consider the expectation properties. First, suppose that $A^{(\alpha)}\geq 0$ for each $\alpha\in E$. Then, by the semi-martingale decomposition from \cite[Section 2.1, pp 10]{Doring2012JumpSDE}, used in the proof of Theorem \ref{thm:mgCondLampKiu},
        \begin{equation*}
            M_t := Y_t - Y_0 - \int_0^t(Af)(J_s,Y_s)ds,
        \end{equation*}
        is a martingale with respect to $\{\mathcal{F}_t\}_{t\geq0}$. For $0<s<t$, by letting $(\hat{J},\hat{Y})$ denote an independent but identically distributed copy of $(J,Y)$,
        \begin{align*}
           \E[Y_t\:|\:\mathcal{F}_s] - Y_s
           = \E\left[\int_s^t (Af)(J_u,Y_u) du \:\middle|\: \mathcal{F}_s\right]
           = \int_0^{t-s} \E_{J_s,Y_s}[(Af)(\hat{J}_u,\hat{Y}_u)] du.
        \end{align*}
                \if{}
                \begin{align*}
                    \E[Y_t\:|\:\mathcal{F}_s] 
                        &= \E[M_t\:|\:\mathcal{F}_s] + \E\left[ \int_0^t(Af)(Y_u,J_u)du \:\middle|\:\mathcal{F}_s\right] + Y_0\\
                        &= M_s + \int_0^s(Af)(Y_u)du + \E_{Y_s}\left[ \int_0^{t-s} (Af)(\hat{Y}_u,\hat{J}_u) du \right] + Y_0\\
                        &= Y_s + \E_{Y_s}\left[ \int_0^{t-s} (Af)(\hat{Y}_u) du \right]\\
                        &= Y_s +  \int_0^{t-s}\E_{Y_s}\left[ (Af)(\hat{Y}_u) \right] du
                \end{align*}
                \fi
        Applying multiplicative invariance to the generator of $(J,Y)$ gives
        \begin{align*}
            \int_0^{t-s}\E_{J_s,Y_s}\left[ (Af)(\hat{J}_u,\hat{Y}_u) \right]du
            = \int_0^{t-s}\E_{J_s,Y_s}\left[\hat{Y}_u(Af)(\hat{J}_u,1)\right] du
            \geq 0,
        \end{align*}
        where the last inequality is due to the assumption that $A^{(\alpha)}\geq0$ for all $\alpha\in E$. Hence $\E[Y_t\:|\:\mathcal{F}_s] \geq Y_s$ and so $Y$ is a sub-martingale.

        Conversely, suppose that $Y$ is a sub-martingale. Then, by integrability of sub-martingales, we still have that $|A^{(\alpha)}|<\infty$, for all $\alpha\in E$, and that
        \begin{equation*}
            M_t := Y_t - Y_0 - \int_0^t(Af)(J_s,Y_s)ds,
        \end{equation*}
        is a martingale.
        
        Now suppose $0<s<t$, then using the semi-martingale property followed by the Markov property and Fubini's theorem, we have
        \begin{align*}
          0   \leq \E[Y_t\:|\:\mathcal{F}_s] - Y_s 
          &= \E\left[\int_s^{t}(Af)(J_u,Y_u)du\right]
          = \int_0^{t-s}\E_{J_s,Y_s}\left[ \hat{Y}_u(Af)(\hat{J}_u,1) \right] du,
        \end{align*}        
        where $(\hat{J},\hat{Y})$ is an independent but identically distributed copy of $(J,Y)$. We can now apply standard calculus results, by differentiating the right hand-side with respect to $t$ and evaluating it at $t=s$, to get
        \begin{align*}
            0 \leq \E_{J_s,Y_s}\left[ \hat{Y}_0(Af)(\hat{J}_0,1) \right] = Y_s(Af)(J_s,1)
        \end{align*}
        and so $(Af)(J_s,1)\geq0$. Then, since $J$ is a continuous time irreducible Markov chain, for every $\alpha>0$ and $s>0$ the probability of the event $\{J_s=\alpha\}$ is non-zero, thus $A^{(\alpha)}\geq0$.
        
        The proof of the super-martingale case is similar.
    \end{proof}

    We can now adapt \cite[Chapter 5, Proposition 3.1]{gemanYor} to exponential MAP models. In both the sub-martingale and super-martingale cases we will make use of the conditions shown in Theorem \ref{thm:mgCondLampKiu}. For the sub-martingale case, the proof is then identical to \cite{gemanYor}, whilst the super-martingale case requires an adaptation.
    
    \begin{thm}[Price comparison of European and Asian call options]
        Suppose that the price process of the asset underlying an Asian call option in the equivalent martingale measure is given by $\{Y_t:t\geq0\}$. 
        Then, the following relations hold:
        \begin{enumerate}[(i)]
            \item if $A^{(\alpha)}>0$ for all $\alpha\in E$, then the Asian call option, $C^A_H$, is cheaper than the corresponding European call option, $C_H$, for any strike $k$;
            \item if $A^{(\alpha)}<0$ for all $\alpha\in E$, then there exists a $K>0$ such that the European call option, $C_H$, is cheaper than the Asian call option, $C_H^A$, for all strikes $k\leq K$.
        \end{enumerate}
    \end{thm}
    \begin{proof}
        To show (i) suppose that $A^{(\alpha)}>0$ for all $\alpha\in E$, so that, by Proposition \ref{prop:subSuperMgLk}, $Y$ is a sub-martingale. Then, the argument in the proof of \cite[Chapter 5, Proposition 3.1(i)]{gemanYor} can be followed exactly.
    
        In the case (ii), we follow the proof of \cite[Chapter 5, Proposition 3.2(ii)]{gemanYor}, but replace some of the explicit calculations for Brownian motion with an inequality.

        In this case, by Proposition \ref{prop:subSuperMgLk}, $Y$ is a super-martingale, hence  it must be integrable. Moreover, by the super-martingale property, for $0<s<T$, we have $\E[Y_s] > \E[Y_T]$, thus,
        \begin{equation*}
          \frac{1}{T}\int_0^T\E[Y_s] ds 
          > \E[Y_T].
        \end{equation*}
        However, $\E\left[ (Y_T-k)^+ +k\right]=\E[Y_T]$ when $k=0$. Then, since $Y_T>0$ for all $T\geq0$,
        there exists $K\in\R^+$, such that for all $k\leq K$
        \begin{equation*}
            \E[(Y_T-k)^++k] < \frac{1}{T}\int_0^T\E[Y_s]ds.
        \end{equation*}
        So for all $k\leq K$,
        \begin{equation}
           \label{eqn:kinequality1}
          k < \frac{1}{T}\int_0^T\E[Y_s]ds - \E[(Y_T-k)^+].
        \end{equation}
        We now  consider $k\leq K$. Then, by the convexity of $(\cdot)^+$, we may apply Jensen's inequality to obtain
        \begin{align*}
          \E\left[ \left( \frac{1}{T}\int_0^T Y_s ds - k\right)^+ \right]
          &> \E\left[ \frac{1}{T}\int_0^T Y_s ds - k \right]^+
            =\left( \E\left[\frac{1}{T}\int_0^TY_sds \right] - k \right)^+.
        \end{align*}
        However, since $(\cdot)^+$ is monotonic, substituting (\ref{eqn:kinequality1}) gives
        \begin{align*}
          \E\left[ \left( \frac{1}{T}\int_0^T Y_s ds - k\right)^+ \right]
          &>  \left( \E\left[\frac{1}{T}\int_0^T Y_s ds \right] - \E\left[ \frac{1}{T}\int_0^T Y_s ds \right] + \E\left[ (Y_T - k)^+ \right] \right)^+ \\
          &= \E\left[ (Y_T-k)^+ \right].
        \end{align*}
        Hence, the European option is cheaper than the corresponding Asian option.
    \end{proof}

        \if{}
        \section{Appendix}
        These are results which I don't believe are new but I have been unable to find a good reference for. Proofs are included for completeness.
        
        \begin{lem}
            Suppose $X$ is a \levy process taking values in $\R$ with Laplace exponent $\psi(z):=\log(\E[\exp(zX_1)])$. Then for any $R\in\R$ such that $\psi(R)<\infty$ and $u\in\R$ the function $\psi(R-iu)=\mathcal{O}(u^2)$  as $|u|\rightarrow\infty$ and is bounded.
        \end{lem}
        
        \begin{proof}
            For $u>1$ by the change of variables $v=ux$
            \begin{align*}
                \int_{-u^{-1}}^{u^{-1}}\left( e^{(R+ui)x} -1 - (R+ui)x \right) \mu(dx)
                &=\int_{-u^{-1}}^{u^{-1}}\left( \sum_{k=0}^\infty \frac{(R+iu)^nx^n}{n!} - 1 - (R+iu)x \right)\mu(dx)\\
                &=\sum_{k=2}^\infty \frac{(R+iu)^n}{n!}\int_{-u^{-1}}^{u^{-1}} x^n\mu(dx)
            \end{align*}
            but by the assumptions on the restrictions on the \levy measure $\int_{-1}^1x^2\mu(dx)<\infty$ and $[-u^{-1},u^{-1}]\subset[-1,1]$ hence
            \begin{align*}
                \left| \int_{-u^{-1}}^{u^{-1}}\left( e^{(R+ui)x} -1 - (R+ui)x \right) \mu(dx) \right|
                \leq \sum_{k=2}^\infty \frac{(R+iu)^n}{n!}\int_{-1}^1x^2\mu(dx)
            \end{align*}
            for some constant $C_1>0$. Similarly there exists a $C_2>0$ such that the result holds for $u<-1$. 
            
            We also have for $u>1$ by the change of variables $v=ux$
            \begin{align*}
                \left|\int_{\R\setminus[-u^{-1},u^{-1}]}\left(e^{iux}-1-iux\mathbbm{1}_{\{|x|\leq1\}}\right) \mu(dx)\right|
                &=u^{-1}\left|\int_{\R\setminus[-1,1]}\left(e^{iv}-1-iv\mathbbm{1}_{\{|v|\leq u\}}\right)\mu(dv)\right|\\
                &\leq u^{-1}\int_{\R\setminus[-1,1]}2\mu(dv) + u^{-1}\int_{[-u,u]\setminus[-1,1]}|v|\mu(dv)\\
                &\leq 2u^{-1}\mu(\R\setminus[-1,1]) + \mu(\R\setminus[-1,1])\\
                &\leq C_3<\infty
            \end{align*}
            for some finite constant $C_3>0$ and again the same result holds for $u<-1$. 
            
            Combining these two results we have for all $u\in\R$ with $|u|>1$
            \begin{align*}
                |\psi(iu)| \leq |au| + \frac{\sigma^2}{2}u^2 + C_1|u|^{-1} + C_3 = \mathcal{O}(u^2)
                \qquad\text{as}\qquad
                |u|\rightarrow\infty.
            \end{align*}
            
        \end{proof}
        \fi

    \if{}
            \section{To do}
            \begin{enumerate}
                \item Does the condition $\psi(R-iu)=\mathcal{O}(u^n)$ hold automatically?
                \item \st{Is an equivalent martingale measure transform needed?}
                \item Does the PIDE give insight into hedging?
            \end{enumerate}
    \fi

    \appendix
    \section{Appendix}

    \subsection{Special Functions}
    \label{app:specialFuncs}
    The \textit{modified Bessel function of the first kind} with parameter $1$ is denoted by $I_1$ and has the series representation
    \begin{align*}
      I_1(x)
      \coloneqq \sum_{n=0}^\infty \frac{1}{n!(n+1)!}\left(\frac{x}{2}\right)^{2n+1}.
    \end{align*}
    It is related to the \textit{standard Bessel function of the first kind}, $J_\alpha$, by the relation $I_1(x)=i^{-1}J_1(ix)$.

    The \textit{generalised Hypergeometric function} is given by the series
    \begin{align*}
      \pFq{p}{q}(a_1,\cdots,a_p; b_1,\cdots,b_q; x)
      \coloneqq \sum_{n=0}^\infty \frac{(a_1)_n\cdots(a_p)_n x^n}{(b_1)_n,\cdots(b_q)_n n!},
    \end{align*}
    for $z\in\C$, $p,q\in\N$, where $(a)_n$ denotes the Pochhammer symbol and is defined by
    \begin{align*}
      (a)_n \coloneqq \begin{cases}
        1  &\text{ if } n=0;\\
        a(a+1)\cdots(a+n-1) &\text{ if } n\geq1.
      \end{cases}
    \end{align*}

    \subsection{Mellin Inversion}
    \label{app:mellinInv}
    In Example \ref{eg:compoundPoissonExpJump}, we consider the following Mellin inversion for $\alpha\in E$ and $k>0$:

    \begin{align*}
      C_\alpha(k) 
      &= \sum_{\beta\in E}\frac{1}{2\pi i}\int_{c+i\R} \frac{k^{-s}}{s(s+1)}\left(e^{TF(s+1)}\right)_{\alpha,\beta}ds
        = \sum_{\beta\in E}\frac{1}{2\pi i}\int_{c+1+i\R} \frac{k^{-(s-1)}}{(s-1)s}\left(e^{TF(s)}\right)_{\alpha,\beta}ds\\
      &= \frac{1}{2\pi i}\int_{c+1+i\R} \frac{k^{-s+1}}{(s-1)s} \frac{2e^{-2qt}G_\alpha e^{qt(G_++G_-)}}{(G_++G_-)}ds + \frac{1}{2\pi i}\int_{c+1+i\R}\frac{k^{-s+1}e^{-2qt}}{s(s-1)} \left( \frac{G_{-\alpha} - G_{\alpha}}{G_++G_-} \right) ds,
    \end{align*}
    for $c\in(0,\min(\lambda_+,\lambda_-))$. As an intermediary step, for $c\in(1,\min(\lambda_+,\lambda_-))$, we consider
    \begin{align*}
      D^{c}_{\alpha}(k) \coloneqq \frac{1}{2\pi i}\int_{c+1+i\R}\frac{k^{-(s-1)}}{(s-1)s}\frac{e^{-2qt}}{(G_++G_-)}G_\alpha ds
      =\frac{1}{2\pi i}\int_{c+i\R}\frac{k^{-(s-1)}}{(s-1)s}e^{-2qt}\frac{\lambda_\alpha(\lambda_{-\alpha}-s)}{2\lambda_+\lambda_- - s(\lambda_++\lambda_-)} ds,
    \end{align*}
    where we have substituted $G_\alpha(z)=\lambda_\alpha/(\lambda_\alpha - z)$ for $z<\lambda_\alpha$ and $\alpha\in E$.
    
    The poles of the integrand are at $0$, $1$ and  $s^*\coloneqq 2\lambda_+\lambda_-/(\lambda_++\lambda_-)$. For $c\in(1,\min(\lambda_+,\lambda_-))$, we have $0,1<c$ and $s^*>c$ and
    \begin{align*}
      \text{Res}(1) &= \frac{k^{-(1-1)}}{1}e^{-2qt}\frac{\lambda_\alpha(\lambda_{-\alpha}-1)}{2\lambda_+\lambda_- - 1(\lambda_++\lambda_-)}
                      = e^{-2qt}\frac{\lambda_\alpha(\lambda_{-\alpha}-1)}{2\lambda_+\lambda_- - (\lambda_++\lambda_-)} ,\\
      \text{Res}(0) &= \frac{k^{-(0-1)}}{(0-1)}e^{-2qt}\frac{\lambda_\alpha(\lambda_{-\alpha}-0)}{2\lambda_+\lambda_- - 0(\lambda_++\lambda_-)}
                      = -\frac{1}{2}ke^{-2qt},\\
      \text{Res}(s^*) &=\frac{k^{-(s^*-1)}e^{-2qt}(\lambda_{\alpha}-\lambda_{-\alpha})}{2\left(2\lambda_+\lambda_- - (\lambda_+ + \lambda_-)\right)}.
    \end{align*}
    It is clear that $\lim_{c\rightarrow\infty}D^c_\alpha(k)=0$  when $k\geq1$ and $\lim_{c\rightarrow-\infty}D^c_\alpha(k)=0$ when $k<1$, thus, by Cauchy's residue theorem,
    \begin{align}
      \label{eqn:DalphaCases}
      D_\alpha(k) = \begin{cases}
        -\frac{1}{2}ke^{-2qt} + e^{-2qt}\frac{\lambda_i(\lambda_{-i}-1)}{2\lambda_+\lambda_- - (\lambda_++\lambda_-)}, & \text{ if } k<1;
        \\ \frac{k^{-(s^*-1)}e^{-2qt}(\lambda_{\alpha}-\lambda_{-\alpha})}{2\left(2\lambda_+\lambda_- - (\lambda_+ + \lambda_-)\right)}, &\text{ if } k\geq1.
      \end{cases}
    \end{align}

    Moreover, from a direct calculation we see that if $k>1$,
    \begin{align}
      \label{eqn:DalphadiffK}
      D_{-\alpha}(k) - D_\alpha(k) =
      \begin{cases}
        \frac{e^{-2qt}(\lambda_{-\alpha} - \lambda_{\alpha})}{2\lambda_+\lambda_- - (\lambda_+ + \lambda_-)}, &\text{ if } k<1;\\
        \frac{2k^{-s^*+1}e^{-2qt}(\lambda_{-\alpha} - \lambda_\alpha)}{2\lambda_+\lambda_- - (\lambda_++\lambda_-)}, & \text{ if } k\geq1.
      \end{cases}
    \end{align}

    We now consider the terms involving $e^{qt(G_++G_-)}$. Let $I_1$ be the modified Bessel function of the first kind of order 1. Then, for $c\in\R$ and $x<0$, from the series expansion of $I_1$, we have that
    \begin{align*}
      \mathcal{M}_k \left\{  \mathbbm{1}_{\{k\geq1\}}\frac{I_1\left(c\sqrt{\log(k)}\right)}{\sqrt{\log(k)}} \right\}(x) = \sum_{m=0}^\infty \frac{c^{2m+1}}{m!(m+1)!2^{2m+1}}\int_{1}^\infty \frac{k^{x-1}\log(k)^{m+\frac{1}{2}}}{\log(k)^{\frac{1}{2}}} dk.
    \end{align*}
    Then, consider the change of variables $y=\log(k)$ to obtain
    \begin{align*}
      \int_1^\infty k^{x-1}\log(k)^m dk
      = \int_{0}^{\infty}e^{yx}y^mdy
      =\frac{m!}{(-x)^{m+!}}.
    \end{align*}
    \if{}
    By using integration by parts when $x<0$ we have
    \begin{align*}
      J_m(x) = \left[ \frac{y^m e^{xy}}{x} \right]^{y=\infty}_{y=0}
      - \int_0^\infty \frac{m}{x}e^{xy}y^{m-1}dy = -\frac{m}{x}J_{m-1}(x),
    \end{align*} 
    and since
    \begin{align*}
      J_0(x) = \int_0^\infty e^{yx} dy = \frac{1}{x},
    \end{align*}
    by induction we have that
    \begin{align*}
      J_m(x) = \frac{m!}{(-x)^{m+1}}.
    \end{align*}
    \fi
    Hence, for $x<0$,
    \begin{align*}
      \mathcal{M}_k  \left\{\mathbbm{1}_{\{k\geq1\}}\frac{I_1\left(c\sqrt{\log(k)}\right)}{\sqrt{\log(k)}} \right\}(x) 
      &= \sum_{m=0}^\infty \frac{c^{2m+1}}{m!(m+1)!2^{2m+1}}\frac{m!}{(-x)^{m+1}} 
      =\frac{2}{c}\left(\exp\left( -\frac{c^2}{4x} \right)-1\right).
    \end{align*}        
    We now set $c\coloneqq2\sqrt{qt\lambda_\alpha}$ and consider
    \begin{align*}
      R_\alpha(k) \coloneqq \delta_{1}(k) +  \sqrt{qt\lambda_\alpha}\mathbbm{1}_{\{k\geq1\}}k^{-\lambda_\alpha}\frac{I_1\left(2\sqrt{qt\lambda_\alpha\log(k)}\right)}{\sqrt{\log(k)}},
    \end{align*}
    where $\delta_k$ denotes the Dirac delta distribution. Then, by the shift rule for the Mellin transform,  we have that, for $x<\lambda_\alpha$,
    \begin{align*}
      \left\{\mathcal{M}R_\alpha\right\}(x) = \exp\left( \frac{qt\lambda_\alpha}{\lambda_\alpha - x} \right).
    \end{align*}
    Moreover, notice that $e^{-qt}R_\alpha(e^x)$ is the probability density of a compound Poisson process of rate $q$, with exponential jumps of rate $\lambda_\alpha$, at time $t$. Hence, $e^{-qt}R_\alpha(e^x)$ is the density of $\xi^\alpha_t$ for all $t>0$. 
    
    By letting $*$ denote the Mellin type convolution, we have that 
    \begin{equation*}
      \mathcal{M}\left\{ D_\alpha * R_+ * R_- \right\}(s) = \frac{ke^{-2qt}}{s(s-1)}\frac{G_\alpha}{(G_++G_-)}\exp\left( \frac{qt\lambda_+}{\lambda_+ - x}
        + \frac{qt\lambda_-}{\lambda_- - x} \right). 
    \end{equation*}
            Hence, we have
        \begin{align*}
            \left\{\mathcal{M}_kC_\alpha(k)\right\}(s) = 2\mathcal{M}\left\{ D_\alpha * R_+ * R_- \right\}(s) + \{\mathcal{M}D_{-\alpha}\}(s) - \{\mathcal{M}D_{\alpha}\}(s),
        \end{align*}
        that is,
        \begin{align}
            C_\alpha(k) = 2\{ D_\alpha * R_+ * R_- \}(k) + D_{-\alpha}(k) - D_\alpha(k),
        \end{align}
        where an explicit expression for $D_{-\alpha}(k)-D_\alpha(k)$ is given in (\ref{eqn:DalphadiffK}).

        \subsection{Series Expansion}
        \label{app:seriesExpansion}

           To continue with Example \ref{eg:compoundPoissonExpJump} it useful to express the triple convolution (\ref{eqn:mellinCovolutionExpression}) as a triple infinite series via the following calculations.

        For $\alpha\in E$, define the function
        \begin{align}
          \label{eqn:mellinExJdefn}
          J_\alpha(k) \coloneqq \sqrt{qt\lambda_\alpha}k^{-\lambda_\alpha}\frac{I_1\left(2\sqrt{qt\lambda_\alpha\log(k)}\right)}{\sqrt{\log(k)}},
        \end{align}
        so that $R_\alpha = \delta_1(k) + \mathbbm{1}_{\{k\geq1\}}J_\alpha(k)$. Then, first consider the multiplicative convolution $R_+*R_-$, which is given by,
        \begin{align*}
          (R_+*R_-)(k)
          &= \int_0^\infty R_+(x)R_-\left(\frac{k}{x}\right)\frac{1}{x}dx\\
          &= \delta_1(k) + \mathbbm{1}_{\{k\geq1\}}J_-\left(k\right) + \mathbbm{1}_{\{k\geq1\}}\frac{1}{k}J_+(k) + \int_1^kJ_+(x)J_-\left(\frac{k}{x}\right)\frac{1}{x}dx.
        \end{align*}
        Now consider the final integral in more detail. Making use of the definition of $J_\alpha$ from (\ref{eqn:mellinExJdefn}), for $k>1$,
        \begin{align*}
          \int_1^kJ_+(x)J_-\left(\frac{k}{x}\right)\frac{1}{x}dx
          &= qt\sqrt{\lambda_+\lambda_-}\int_1^k \frac{x^{-\lambda_+}I_1(2\sqrt{qt\log(x)})}{\sqrt{\log(x)}} \left(\frac{k}{x}\right)^{-\lambda_-}\frac{I_1(2\sqrt{qt\log(k/x)})}{\sqrt{\log(k/x)}}  \frac{1}{x}dx\\  &=qt\sqrt{\lambda_+\lambda_-}k^{-\lambda_-}\int_0^{\log(k)}e^{(\lambda_- - \lambda_+)z } \frac{I_1\left(2\sqrt{qtz}\right)I_1\left(2\sqrt{qt}\sqrt{\log(k)-z}\right)}{\sqrt{\log(k)-z}}dz,
        \end{align*}
        where the second equality is obtained via the subsitituion $z=\log(x)$. 
        From the series expansions of the modified Bessel function $I_1$ and the exponential, it follows that
        \begin{align*}
          &\int_1^kJ_+(x)J_-\left(\frac{k}{x}\right)\frac{1}{x}dx\\
          & =  qt\sqrt{\lambda_+\lambda_-}k^{-\lambda_-}\sum_{n=0}^\infty\sum_{m=0}^\infty\sum_{r=0}^\infty \frac{(qt)^{m+n+1}(\lambda_--\lambda_+)^r}{m!(m+1)!n!(n+1)!r!}\int_0^{\log(k)}z^{m+r}\left(\log(k)-z\right)^n dz.
        \end{align*}
        Then, by using integration by parts for $n>0$,
        \begin{align*}
          \int_0^{\log(k)}z^{m+r}\left(\log(k)-z\right)^ndz
          &= \frac{n!(r+m)!}{(r+m+n+1)!}\log(k)^{r+m+n+1},
        \end{align*}
        and so,
        \begin{align*}
          \int_1^kJ_+(x)J_-\left(\frac{k}{x}\right)\frac{1}{x}dx
          =  \sum_{n=0}^\infty\sum_{m=0}^\infty\sum_{r=0}^\infty a_{n,m,r}(q,t,\lambda_-,\lambda_+)k^{-\lambda_-}\log(k)^{r+m+n+1},
        \end{align*}
        where
        \begin{align*}
          a_{n,m,r}(q,t,\lambda_-,\lambda_+)
          \coloneqq \frac{(qt)^{m+n+2}\sqrt{\lambda_+\lambda_-}(\lambda_--\lambda_+)^r(r+m)!}{m!(m+1)!(n+1)!r!(r+m+n+1)!}.
        \end{align*}
        
        Now consider the convolution of this series with $D_\alpha$. From (\ref{eqn:DalphaCases}) there exist constants $d_1^\alpha,d_2^\alpha,d_3^\alpha$ and $c_\alpha$, such that 
        \begin{align*}
          D_\alpha(k) =
          \begin{cases}
            d_1^\alpha k + d_2^\alpha &\text{ if } k<1;\\
            d_3^\alpha k^{c_\alpha} &\text{ if } k\geq1.
          \end{cases}
        \end{align*}
        Hence, 
        \begin{align*}
          &\left(D_\alpha * J_+ * J_- \right)(k)
            = \sum_{n=0}^\infty\sum_{m=0}^\infty\sum_{r=0}^\infty a_{n,m,r}(q,t,\lambda_-,\lambda_+) \\
          &\times\left(\int_0^1 \left(\frac{k}{x}\right)^{-\lambda_-}\log\left(\frac{k}{x}\right)^{r+m+n+1}\left(d_1^\alpha x+d_2^\alpha\right)\frac{1}{x}dx
            + \int_1^\infty \left(\frac{k}{x}\right)^{-\lambda_-}\log\left(\frac{k}{x}\right)^{r+m+n+1}d_3^\alpha x^{c_\alpha}\frac{1}{x}dx\right).
        \end{align*}
        Now consider the following integrals, for $l\in(0,\infty)$, and make the substitution $t=(a+1)\log(k/x)$ to obtain
        \begin{align*}
          \int_0^l x^a \log\left(\frac{k}{x}\right)^N dx
          &= \frac{k^{a+1}}{(a+1)^{N+1}}\int_{(a+1)\log(k/l)}^\infty e^{-t}t^N dt
            = \frac{k^{a+1}}{(a+1)^{N+1}}\Gamma\left(N+1,(a+1)\log\left(\frac{k}{l}\right)\right),\\
          \int_l^\infty x^a \log\left(\frac{k}{x}\right)^N dx
          &= \frac{k^{a+1}}{(a+1)^{N+1}}\int_0^{(a+1)\log(k/l)}e^{-t}t^N dt
            = \frac{k^{a+1}}{(a+1)^{N+1}}\gamma\left(N+1,(a+1)\log\left(\frac{k}{l}\right)\right),
        \end{align*}
        where $\Gamma(\cdot,\cdot)$ and $\gamma(\cdot,\cdot)$ are the \textit{upper} and \textit{lower incomplete Gamma functions}, respectively.
        Thus, it follows that
        \begin{align*}
          \left(D_\alpha * J_+ * J_- \right)(k)
          = \sum_{n=0}^\infty\sum_{m=0}^\infty\sum_{r=0}^\infty \left( 
            g_1(n,m,r,\lambda_-,k)
            + g_2(n,m,r,\lambda_-,k)
            + g_3(n,m,r,\lambda_-,k)
          \right),
        \end{align*}
        where,
        \begin{align*}
          g_1(n,m,r,\lambda_-,k) &\coloneqq a_{n,m,r}(q,t,\lambda_+,\lambda_-)\frac{d_1^\alpha k}{(\lambda_- +1)^{r+m+n+2}}\Gamma(r+m+n+2,(\lambda_-+1)\log(k)),\\
          g_2(n,m,r,\lambda_-,k) &\coloneqq             a_{n,m,r}(q,t,\lambda_+,\lambda_-)\frac{d_2^\alpha}{\lambda_-^{r+m+n+2}}\Gamma(r+m+n+2,\lambda_-\log(k)),\\
          g_3(n,m,r,\lambda_-,k) &\coloneqq a_{n,m,r}(q,t,\lambda_+,\lambda_-) \frac{d_3^\alpha k^{c_\alpha}}{(c_\alpha+\lambda_-)^{r+m+n+2}}\gamma(r+m+n+2,(c_\alpha+\lambda_-)\log(k)).
        \end{align*}
        Now compute the multiplicative convolution
        \begin{align*}
          \left\{\left\{\mathbbm{1}_{\{x\geq1\}}\frac{1}{x}J_\beta(x)\right\}*D_\alpha(x)\right\}(k)
          &= \int_0^\infty D_\alpha(x)\mathbbm{1}_{\{k\geq x\}}\frac{x}{k}J_\beta\left(\frac{k}{x}\right)\frac{1}{x} dx\\
          &= \frac{1}{k}\int_0^{k\wedge1}\left(d_1^\alpha x + d_2^\alpha\right)J_+\left(\frac{k}{x}\right)dx + \frac{1}{k}\int_{k\wedge1}^k d_3^\alpha x^{c_\alpha}J_\beta\left(\frac{k}{x}\right)dx.
        \end{align*}
        Consider the general integral
        \begin{align*}
          \int_l^u x^aJ_\beta\left(\frac{k}{x}\right)dx
          &= \sqrt{qt\lambda_\beta}k^{-\lambda_\beta}\int_l^ux^{a+\lambda_\beta}\frac{I_1\left(2\sqrt{qt\lambda_\beta\log(k/x)}\right)}{\sqrt{\log(k/x)}}dx\\
          &= qt\lambda_\beta\sum_{m=0}^\infty\frac{k^{-\lambda_\beta}}{m!(m+1)!}\int_l^u x^{\lambda_\beta+a}\log\left(\frac{k}{x}\right)^m dx.
        \end{align*}
        Then, using the previous results to evalute the integral, 
        \begin{align*}
          &\int_l^ux^aJ_\beta\left(\frac{k}{x}\right)
          = qt\lambda_\beta\sum_{m=0}^\infty \frac{k^{-\lambda_\beta}}{m!(m+1)!} \frac{k^{\lambda_\beta+a+1}}{(\lambda_\beta+a+1)^{m+1}}\\
          &\qquad\times\left( \Gamma\left(m+1, (\lambda_\beta+a+1)\log\left(\frac{k}{u}\right)  \right) - \Gamma\left(m+1, (\lambda_\beta+a+1)\log\left(\frac{k}{l}\right)\right)  \right),
        \end{align*}
        
        so,
        \begin{align*}
          &\left\{\left\{ \indicator{x\geq1} \frac{1}{x}J_\beta(x)  \right\} * D_\alpha(x)\right\}(k)\\
          &\qquad= \sum_{m=0}^\infty \frac{qt\lambda_\beta}{m!(m+1)!}
          \left( F(m,d_1^\alpha,\lambda_\beta,1,k)
                +F(m,d_2^\alpha,\lambda_\beta,0,k)
                +f(m,d_3^\alpha,\lambda_\beta,c_\alpha,k)
            \right)
        \end{align*}
        and
        \begin{align*}
          &\left\{\left\{ \indicator{x\geq1}J_\beta(x)  \right\} * D_\alpha(x)\right\}(k)\\
          &\qquad= \sum_{m=0}^\infty \frac{qt\lambda_\beta}{m!(m+1)!}
          \left( F(m,d_1^\alpha,\lambda_\beta,2,k)
                +F(m,d_2^\alpha,\lambda_\beta,1,k)
                +f(m,d_3^\alpha,\lambda_\beta,c_\alpha+1,k)
          \right),
        \end{align*}
        where,
        \begin{align*}
          F(m,d,\lambda,c,k) &\coloneqq \frac{d k^c\Gamma(m+1,(\lambda+c+1)\log(k\vee1)))}{(\lambda_\alpha+c+1)^{m+1}},\\
          f(m,d,\lambda,c,k) &\coloneqq \frac{d k^c\gamma(m+1,(\lambda+c+1)\log(k\vee1)))}{(\lambda_\alpha+c+1)^{m+1}}.
        \end{align*}

        Finally, notice that $\{D_\alpha * \delta_1\}(k) = D_\alpha(k)$.        
        Thus, putting these components together gives
        \begin{align*}
          C_\alpha(k)
          =&\sum_{n=0}^\infty\sum_{m=0}^\infty\sum_{r=0}^\infty 2\left( g_1(n,m,r,\lambda_-,k)+ g_2(n,m,r,\lambda_-,k)+g_3(n,m,r,\lambda_-,k)\right)\\
           &+ \sum_{m=0}^\infty \frac{qt\lambda_+}{m!(m+1)!}
             \left( F(m,d_1^\alpha,\lambda_+,1,k)
             +F(m,d_2^\alpha,\lambda_+,0,k)
             +f(m,d_3^\alpha,\lambda_+,c_\alpha,k)
             \right)\\
          &+ \sum_{m=0}^\infty \frac{qt\lambda_{-}}{m!(m+1)!}
             \left( F(m,d_1^\alpha,\lambda_{-},2,k)
             +F(m,d_2^\alpha,\lambda_{-},1,k)
             +f(m,d_3^\alpha,\lambda_{-},c_\alpha+1,k)
             \right)\\
           &+\indicator{k<1}\left(d_1^{-\alpha}+d_2^{-\alpha}\right)
             +\indicator{k\geq1}\left(d_3^{-\alpha}k^{c_{-\alpha}} \right).
        \end{align*}

  \bibliography{ref}{}
  \bibliographystyle{abbrv}        

\end{document}